\documentclass[12pt]{amsart}
\newtheorem{thm}{Theorem}[section]
\newtheorem{lem}[thm]{Lemma}
\newtheorem{cor}[thm]{Corollary}
\newtheorem{prop}[thm]{Proposition}
\theoremstyle{remark}
\newtheorem{defn}[thm]{Definition}
\newtheorem{rem}[thm]{Remark}
\newtheorem{exa}[thm]{Example}
\newtheorem{notation}[thm]{Notation}

\newtheorem*{prfofthmh0}{Proof of Theorem~\ref{h0upperthm}}
\newtheorem*{prfofthmh1}{Proof of Theorem~\ref{h1upperthm}}
\newtheorem*{acknowledgement}{Acknowledgment}
\makeatletter
 
 \@addtoreset{equation}{section}

\makeatother

\title{Upper bounds on the slope of certain fibered surfaces}
\author{Makoto Enokizono}
\address{Makoto Enokizono,
	Department of Mathematics,
	Graduate School of Science,
	Osaka University,
	Toyonaka, Osaka 560-0043, Japan}
\email{m-enokizono@cr.math.sci.osaka-u.ac.jp}
\usepackage[top=35truemm,bottom=35truemm,left=25truemm,right=25truemm]{geometry}
\usepackage{amsmath,cases}
\usepackage{amsfonts}
\usepackage[all]{xy}
\usepackage{here}
\usepackage{latexsym}
 \def\qed{\hfill $\Box$} 
\begin{document}
\maketitle

\begin{abstract}
We establish the slope equality and give an upper bound of the slope for finite cyclic covering fibrations of an elliptic surface including bielliptic fibrations of genus greater than 5. We also give an upper bound of the slope for triple cyclic covering fibrations of a ruled surface and hyperelliptic fibrations, which provides a new proof of Xiao's upper bound.
\end{abstract}

\section*{Introduction}
We shall work over the complex number field $\mathbb{C}$.
Let $S$ be a smooth projective surface, $B$ a smooth projective curve and $f\colon S\to B$ a relatively minimal fibration of curves of genus $g\ge 2$.
Let $K_f$ denote the relative canonical bundle $K_S-f^{*}K_B$.
Consider the following relative invariants:
\begin{align*}
\chi_f:=&\ 
\chi(\mathcal{O}_S)-\chi(\mathcal{O}_B)\chi(\mathcal{O}_F),\\
e_f:=&\ e(S)-e(B)e(F),
\end{align*}
where $F$ denotes a general fiber of $f$ and $e(M)$ the topological Euler characteristic of $M$.
Then the following are well known:
\begin{itemize}
\item (Noether) $12\chi_f=K_f^2+e_f$.
\item (Arakelov) $K_f$ is nef.
\item (Ueno) $\chi_f\ge 0$, and $\chi_f=0$ if and only if $f$ is locally trivial (i.e., a holomorphic fiber bundle).
\item (Segre) $e_f\ge 0$, and $e_f=0$ if and only if $f$ is smooth.
\end{itemize}

Assume that $f$ is not locally trivial.
We put 
$$
\lambda_f=\frac{K_f^2}{\chi_f}
$$
and call it the {\em slope} of $f$.
Then one sees $0<\lambda_f\le 12$ from the above results. 
In 1987, Xiao showed in \cite{xiao2} the inequality
$$
\lambda_f\ge 4-\frac{4}{g}.
$$
This inequality is sharp and it is well known that the slope attains the lower bound only if the fibration $f$ is hyperelliptic, i.e., a general fiber $F$ is a hyperelliptic curve (\cite{kon} and \cite{xiao2}).
As to the upper bound, Kodaira \cite{kod} constructed examples of fibrations with slope $12$, which are nowadays called Kodaira fibrations.
Thus the inequality $\lambda_f\le 12$ is sharp among all fibered surfaces.
On the other hand, Matsusaka \cite{ma} obtained an upper bound smaller than $12$ for hyperelliptic case
 and Xiao \cite{xiao_book} improved this bound.
In \cite{chen-tan}, upper bounds for genus $3$ fibrations are studied from another point of view. 

We studied in \cite{enoki} primitive cyclic covering fibrations of type $(g,h,n)$.
Roughly speaking, it is a fibered surface of genus $g$ obtained as the relatively minimal model of an $n$ sheeted cyclic branched covering of another fibered surface of genus $h$.
Note that hyperelliptic fibrations are nothing more than such fibrations of type $(g,0,2)$
and that bielliptic fibrations of genus $g\ge 6$ are those of type $(g,1,2)$ (cf.\ \cite{barja-naranjo} and  \cite{dbl}).
Here, a fibration is called {\em bielliptic} if a general fiber is a bielliptic curve, i.e., a non-singular projective curve obtained as a double covering of an elliptic curve.
In \cite{enoki}, we established the lower bound of the slope for such fibrations of type $(g,h,n)$ extending former results for $n=2$ in \cite{ba} and \cite{dbl}.
Furthermore, when $h=0$ and $n\ge 4$, we obtained even the upper bound (expressed as a function in $g$ and $n$) which is strictly smaller than $12$.
Recall that known examples of Kodaira fibrations, including Kodaira's original ones \cite{kod}, are presented as primitive cyclic covering fibrations with $h\ge 2$ and, in fact, there exist such fibrations for any $h\ge 2$ (see, \cite{bar} and \cite{kas}). Hence, as far as the upper bound of the slope strictly smaller than $12$ concerns, the remaining cases to be examined are $h=0$, $n=3$ and $h=1$, $n\ge 2$.

The purpose of the present article is to give an affirmative answer to the above mentioned upper bound problem by introducing numerical invariants attached to the singularities of the branch locus of the cyclic covering, and improving a coarser estimates in \cite{enoki}.
When $h=0$, a vertical component of the branch locus on a relatively minimal model is always a non-singular rational curve and this fact makes it much easier to handle singularities on the branch locus.
On the other hand, when $h\ge 0$, we must pay attention to all subcurves of fibers and their singularities in a fibration of genus $h$, which seems quite terrible.
Fortunately enough, when $h=1$, we have Kodaira's classification of singular fibers \cite{Kod} from which we know that major components are rational curves and singularities are mild.
This gives us a hope to extend results for $h=0$ to fibrations of type $(g,1,n)$.
In fact, we can show the following:

\begin{thm} \label{introh1slopeeqthm}
Assume that $g\ge (2n-1)(3n-1)/(n+1)$. Then,
there exists a function $\mathrm{Ind}\colon 
 \mathcal{A}_{g,1,n}\to \mathbb{Q}_{\geq 0}$ from the set 
 $\mathcal{A}_{g,1,n}$ of all fiber germs of primitive cyclic covering 
 fibrations of type $(g,1,n)$ such that $\mathrm{Ind}(F_p)=0$ for a 
 general $p\in B$ and
\begin{equation*}
K_f^2=\frac{12(n-1)}{2n-1}\chi_f+\sum_{p\in B}{\rm Ind}(F_p)
\end{equation*}
for any primitive cyclic covering fibration $f\colon S\rightarrow 
 B$ of type $(g,1,n)$.
\end{thm}

\begin{thm} \label{introh1upperthm}
Let $f\colon S\to B$ be a primitive cyclic covering fibration of type $(g,1,n)$.
Then
$$
\lambda_f\le 12-\left\{\begin{array}{l}
\displaystyle{\frac{6n^2}{(n+1)(g-1)}},\quad \text{if } n\ge 4, \text{ or }n=3 \text{ and } g=4 \\
\displaystyle{\frac{24}{4g-17}},\quad \quad \quad \quad \text{if } n=3 \text{ and } g>4, \\
\displaystyle{\frac{2}{g-2}},\quad \quad \quad \quad \quad \text{if } n=2 \text{ and } g\ge 3.
\end{array}
\right.
$$
\end{thm}

In particular, we have the slope equality and the upper bound of the slope for bielliptic fibrations of genus $g\ge 6$.
We remark that the upper bounds in Theorem~\ref{introh1upperthm} are ``fiberwise'' sharp as we shall see in Example~\ref{fiberexa}. 
We do not know, however, whether there exist primitive cyclic covering fibrations of type $(g,1,n)$ whose slopes attain the bounds.

It remains to investigate the case that $h=0$ and $n\le 3$.
The above method also applies to this case and we obtain the following:

\begin{thm} \label{introh0upperthm}
Let $f\colon S\to B$ be a primitive cyclic covering fibration of type $(g,0,n)$

\smallskip

\noindent
$(1)$ If $g=4$ and $n=3$, then $\lambda_f \le 129/17$
with the equality sign holding if and only if any singular fiber of $f$ is a triple fiber.

\smallskip

\noindent
$(2)$ Assume that $n=3$. If either $g>4$, or $g=4$ and $f$ has no triple fibers, then

$$
\lambda_f\le 12-\frac{72(g+1)}{4g^2+g+13-36\delta},
$$
where $\delta=0$ if $g+2\in 6\mathbb{Z}$ and $\delta=1$ otherwise.

\smallskip

\noindent
$(3)$ If $n=2$, then
$$
\lambda_f \le 12-\frac{4(2g+1)}{g^2-1+\delta}=12-\frac{4(2g+1)}{g^2+\frac{(-1)^g-1}{2}},
$$
where $\delta=0$ if $g$ is odd and $\delta=1$ if $g$ is even, i.e., $2\delta=1+(-1)^g$.

\end{thm}

We obtained the complete classification of singular fibers of primitive cyclic covering fibrations of type $(4,0,3)$ in \cite{enoki2} and computed some local invariants for each singular fiber.
One may use it to reprove Theorem~\ref{introh0upperthm} (1).
Note also that Theorem~\ref{introh0upperthm} (3) provides a new proof of Xiao's upper bound for hyperelliptic fibrations in \cite{xiao_book} referred above.
It is shown in \cite{liutan} that the inequality is optimal for given $g$.

The organization of the paper is as follows.
In \S1, we recall basic results from \cite{enoki} and \cite{enoki2} on primitive cyclic covering fibrations and introduce some notation for the later use.
In \S2, we observe the local concentration of relative invariants of primitive cyclic covering fibrations of type $(g,1,n)$ on a finite number of fiber germs and show Theorem~\ref{introh1slopeeqthm}.
\S3 will be devoted to the proof of Theorem~\ref{introh1upperthm}.
In the course of the study, we freely use Kodaira's table of singular fibers of elliptic surfaces.
Finally in \S4, we show Theorem~\ref{introh0upperthm}.

\begin{acknowledgement} \normalfont
The author expresses his sincere gratitude to Professor Kazuhiro Konno for many suggestions and warm encouragement.

\end{acknowledgement}

\section{Preliminaries}
In this section, we recall and state basic results for primitive cyclic covering fibrations in \cite{enoki}.

\begin{defn}\normalfont \label{primdef}
A relatively minimal fibration $f\colon S\to B$ of genus $g\geq 2$ is called a 
primitive cyclic covering fibration of type $(g,h,n)$, if there exist a (not 
necessarily relatively minimal) fibration 
$\widetilde{\varphi}\colon \widetilde{W}\to B$ of genus $h\geq 0$, and a 
 classical $n$-cyclic covering 
$$
\widetilde{\theta}\colon \widetilde{S}=
\mathrm{Spec}_{\widetilde{W}}\left(\bigoplus_{j=0}^{n-1}
\mathcal{O}_{\widetilde{W}}(-j\widetilde{\mathfrak{d}})\right)\to \widetilde{W}
$$ 
branched over a smooth curve $\widetilde{R}\in |n\widetilde{\mathfrak{d}}|$
for some $n\geq 2$ and $\widetilde{\mathfrak{d}}\in 
\mathrm{Pic}(\widetilde{W})$ 
such that 
$f$ is the relatively minimal 
model of $\widetilde{f}:=\widetilde{\varphi}\circ \widetilde{\theta}$.
\end{defn}

Let $f\colon S\to B$ be a primitive cyclic covering fibration of type $(g,h,n)$. We freely use the notation in Definition \ref{primdef}. Let $\widetilde{F}$ and $\widetilde{\Gamma}$ be general fibers of 
$\widetilde{f}$ and $\widetilde{\varphi}$, respectively. Then the restriction map 
$\widetilde{\theta}|_{\widetilde{F}}\colon \widetilde{F}\to 
\widetilde{\Gamma}$ is a classical $n$-cyclic covering branched over 
$\widetilde{R}\cap \widetilde{\Gamma}$. Since the genera of $\widetilde{F}$ and $\widetilde{\Gamma}$ are $g$ and 
$h$, respectively, the Hurwitz formula gives us
\begin{equation}\label{r}
r:=\widetilde{R}\widetilde{\Gamma}=\frac{2(g-1-n(h-1))}{n-1}.
\end{equation}
Note that $r$ is a multiple of $n$. Let $\widetilde{\sigma}$ be a generator of $\mathrm{Aut}(\widetilde{S}/\widetilde{W})
\simeq \mathbb{Z}/n\mathbb{Z}$ and $\rho\colon \widetilde{S}\to S$ the natural 
birational morphism. By assumption, $\mathrm{Fix}(\widetilde{\sigma})$ is a disjoint 
union of smooth curves and $\widetilde{\theta}(\mathrm{Fix}(\widetilde{\sigma}))=\widetilde{R}$. Let $\varphi\colon W\to B$ be a relatively minimal model of 
$\widetilde{\varphi}$ and $\widetilde{\psi}\colon \widetilde{W}\to W$ the 
natural birational morphism. Since $\widetilde{\psi}$ is a succession of blow-ups, 
we can write $\widetilde{\psi}=\psi_1\circ \cdots \circ \psi_N$, where 
$\psi_i\colon W_i\to W_{i-1}$ denotes the blow-up at $x_i\in W_{i-1}$ 
$(i=1,\dots,N)$ with $W_0=W$ and $W_N=\widetilde{W}$. We define reduced curves $R_i$ on $W_i$ inductively as $R_{i-1}=(\psi_i)_*R_i$ starting 
from $R_N=\widetilde{R}$ down to $R_0=:R$. We also put $E_i=\psi_i^{-1}(x_i)$ and 
$m_i=\mathrm{mult}_{x_i}(R_{i-1})$ for $i=1,2,\dots, N$.

\begin{lem}\label{multlem}
With the above notation, the following hold for any $i=1,\dots, N$.

\smallskip

\noindent
$(1)$ Either $m_i\in n\mathbb{Z}$ or $m_i\in n\mathbb{Z}+1$.
Moreover, $m_i\in n\mathbb{Z}$ holds if and only if $E_i$ is not contained in $R_i$.

\smallskip

\noindent
$(2)$ 
 $R_i=\psi_i^*R_{i-1}-n\displaystyle{\left[\frac{m_i}{n}\right]}E_i$, 
 where $[t]$ is the greatest integer not exceeding $t$.

\smallskip

\noindent
$(3)$ There exists $\mathfrak{d}_i\in \mathrm{Pic}(W_i)$ such that 
$\mathfrak{d}_i=\psi_i^*\mathfrak{d}_{i-1}$ and $R_i\sim n\mathfrak{d}_i$, $\mathfrak{d}_N=\widetilde{\mathfrak{d}}$.
\end{lem}

Let $E$ be a $(-1)$-curve on a fiber of $\widetilde{f}$. 
If $E$ is not contained in ${\rm Fix}(\widetilde{\sigma})$, then $L:=\widetilde{\theta}(E)$ is a $(-1)$-curve and $\widetilde{\theta}^{\ast}L$ is the sum of $n$ disjoint $(-1)$-curves containing $E$. 
Contracting them and $L$, we may assume that any $(-1)$-curve on a fiber of $\widetilde{f}$ is contained in ${\rm Fix}(\widetilde{\sigma})$. 
Then $\widetilde{\sigma}$ induces an automorphism $\sigma$ of $S$ over $B$ and $\rho$ is the blow-up of all isolated fixed points of $\sigma$ (cf.\ \cite{enoki}). 
One sees easily that there is a one-to-one correspondence between $(-k)$-curves contained in ${\rm Fix}(\widetilde{\sigma})$ and $(-kn)$-curves contained in $\widetilde{R}$ via $\widetilde{\theta}$. Hence, the number of blow-ups in $\rho$ is that of vertical $(-n)$-curves contained in $\widetilde{R}$.

From Lemma \ref{multlem}, we have
\begin{eqnarray}
K_{\widetilde{\varphi}}&=
&\widetilde{\psi}^{\ast}K_{\varphi}+\sum_{i=1}^N {\bf E}_i \label{kphi},\\
\widetilde{\mathfrak{d}}&=
&\widetilde{\psi}^{\ast}\mathfrak{d}-\sum_{i=1}^N \left[\frac{m_i}{n}\right]{\bf E}_i, \label{delta}
\end{eqnarray}
where ${\bf E}_i$ denotes the total transform of $E_i$.
Since
$$
K_{\widetilde{S}}=
\displaystyle{\widetilde{\theta}^{\ast}\left(K_{\widetilde{W}}+(n-1)\widetilde{\mathfrak{d}}\right)}
$$
and
$$
\chi(\mathcal{O}_{\widetilde{S}})=
n\chi(\mathcal{O}_{\widetilde{W}})+\frac{1}{2}\sum_{j=1}^{n-1}j\widetilde{\mathfrak{d}}(j\widetilde{\mathfrak{d}}+K_{\widetilde{W}}),
$$
we get
\begin{eqnarray}
K_{\widetilde{f}}^2&=
&n(K_{\widetilde{\varphi}}^2+2(n-1)K_{\widetilde{\varphi}}\widetilde{\mathfrak{d}}+(n-1)^2\widetilde{\mathfrak{d}}^2), \label{kftilde} \\
\chi_{\widetilde{f}}&=
&n\chi_{\widetilde{\varphi}}+\frac{1}{2}\sum_{j=1}^{n-1}j\widetilde{\mathfrak{d}}(j\widetilde{\mathfrak{d}}+K_{\widetilde{\varphi}}). \label{chiftilde}
\end{eqnarray}

\begin{defn}[Singularity index $\alpha$]\label{sinddef}
Let $k$ be a positive integer.
For $p\in B$, we consider all the singular points (including infinitely near 
 ones) of $R$ on the fiber $\Gamma_p$ of $\varphi\colon W\to B$ over $p$.
We let $\alpha_k(F_p)$ be the number of singular points of multiplicity 
 either $kn$ or $kn+1$ among them, and call it the {\em $k$-th singularity 
 index} of $F_p$, the fiber of $f\colon S\to B$ over $p$.
Clearly, we have $\alpha_k(F_p)=0$ except for a finite number of $p\in B$.
We put $\alpha_k=\sum_{p\in B}\alpha(F_p)$ and call it the {\em $k$-th 
 singularity index} of $f$.

Let $D_1$ be the sum of all $\widetilde{\varphi}$-vertical $(-n)$-curves 
 contained in $\widetilde{R}$ and put $\widetilde{R}_0=\widetilde{R}-D_1$.
We denote by $\alpha_0(F_p)$ the ramification index of 
 $\widetilde{\varphi}|_{\widetilde{R}_0}\colon \widetilde{R}_0\to B$ over $p$, 
 that is, the ramification index of 
 $\widetilde{\varphi}|_{(\widetilde{R}_0)_h}\colon (\widetilde{R}_0)_h\to B$ 
 over $p$ minus the sum of the topological Euler number of irreducible 
 components of $(\widetilde{R}_0)_v$ over $p$.
Then $\alpha_0(F_p)=0$ except for a finite number of $p\in B$, and we have
$$
\sum_{p\in B}\alpha_0(F_p)=(K_{\widetilde{\varphi}}+\widetilde{R}_0)\widetilde{R}_0
$$
by definition.
We put $\alpha_0=\sum_{p\in B}\alpha(F_p)$ and call it the {\em $0$-th singularity index} of $f$.
\end{defn}

Let $\varepsilon(F_p)$ be the number of $(-1)$-curves contained in $\widetilde{F}_p$, 
and put $\varepsilon=\sum_{p\in B}\varepsilon(F_p)$.
This is no more than the number of blowing-ups appearing in 
$\rho\colon \widetilde{S}\to S$.

From \eqref{kphi} and \eqref{delta}, we have
\begin{align}
(K_{\widetilde{\varphi}}+\widetilde{R})\widetilde{R}=&\;
\left(\widetilde{\psi}^*(K_\varphi+R)+\sum_{i=1}^N
\left(1-n\left[\frac{m_i}{n}\right]\right)\mathbf{E}_i\right)
\left(\widetilde{\psi}^*R-n\left[\frac{m_i}{n}\right]\mathbf{E}_i\right) \nonumber \\
=&\;(K_\varphi+R)R-\sum_{i=1}^N n\left[\frac{m_i}{n}\right]
\left(n\left[\frac{m_i}{n}\right]-1\right) \nonumber \\
=&\;(K_\varphi+R)R-n\sum_{k\ge 1}k(nk-1)\alpha_k. \label{2cal1}
\end{align}
On the other hand, we have
\begin{equation}\label{2cal2}
(K_{\widetilde{\varphi}}+\widetilde{R})\widetilde{R}=(K_{\widetilde{\varphi}}+\widetilde{R}_0)\widetilde{R}_0
+D_1(K_{\widetilde{\varphi}}+D_1)=\alpha_0-2\varepsilon.
\end{equation}
Hence,
\begin{equation} \label{KReq}
(K_\varphi+R)R=n\sum_{k\ge 1}k(nk-1)\alpha_k+\alpha_0-2\varepsilon.
\end{equation}
by \eqref{2cal1} and \eqref{2cal2}.
Since $K_{f}^2=K_{\widetilde{f}}^2+\varepsilon$, $\chi_{\widetilde{f}}=\chi_f$, \eqref{kphi}, \eqref{delta}, \eqref{kftilde} and \eqref{chiftilde}, we get
\begin{align} \label{kfeq}
K_{f}^2
=&\;nK_{\varphi}^2+2(n-1)K_{\varphi}R+\frac{(n-1)^2}{n}R^2-
\sum_{k\ge 1}((n-1)k-1)^2\alpha_k+\varepsilon
\end{align}
and
\begin{align} \label{chifeq}
\chi_f
=&\;n\chi_{\varphi}+\frac{(n-1)(2n-1)}{12n}R^2+\frac{n-1}{4}K_{\varphi}R
-\frac{n(n-1)}{12}\sum_{k\ge 1}((2n-1)k^2-3k)\alpha_k.
\end{align}
From \eqref{KReq}, \eqref{kfeq}, \eqref{chifeq} and Noether's formula, we have
\begin{equation} \label{efeq}
e_f=ne_{\varphi}+n\sum_{k\ge 1}\alpha_k+(n-1)\alpha_0-(2n-1)\varepsilon.
\end{equation}

We define some notation for the later use. For a vertical divisor $T$ and $p\in B$, we denote by $T(p)$ the 
greatest subdivisor of $T$ consisting of components of the fiber over $p$.
Then $T=\sum_{p\in B}T(p)$.
We consider a family $\{L^{i}\}_i$ of vertical irreducible curves in 
$\widetilde{R}$ over $p$ satisfying:

\smallskip

(i) $L^{1}$ is the proper transform of an irreducible curve $\Gamma^1$ contained in the fiber $\Gamma_p$ or a 
$(-1)$-curve $E^{1}$ appearing in $\widetilde{\psi}$.

\smallskip

(ii) For $i\geq 2$, $L^{i}$ is the proper transform of an irreducible curve $\Gamma^i$ contained in the fiber $\Gamma_p$ intersecting $\Gamma^{k}$ for some $k<i$
or an exceptional 
$(-1)$-curve $E^{i}$ that contracts to a point $x^{i}$ on $C^{k}$ 
(or on its proper transform) for some $k<i$, where we define $C^{j}$ to be 
$E^{j}$ or $\Gamma^{j}$ according to whether $L^{j}$ is the proper 
transform of which curve.

\smallskip

(iii) $\{L^{i}\}_i$ is the largest among those satisfying (i) and (ii).

\medskip

The set of all vertical irreducible curves in $\widetilde{R}$ over $p$ 
is decomposed into the disjoint union of such families uniquely.
We denote it as 
$$
\widetilde{R}_v(p)=D^1(p)+\cdots+D^{\eta_p}(p), \quad D^t(p)=\sum_{k\ge 1}L^{t,k}
$$
where $\eta_p$ denotes the number of the decomposition and $\{L^{t,k}\}_{k}$ satisfies 
(i), (ii), (iii). 
Let $C^{t,k}$ be the exceptional curve or the component of the fiber $\Gamma_p$ the proper transform of which is $L^{t,k}$.
Let $D'^t(p)$ be the sum of all irreducible components of $D^t(p)$ which are the proper transforms of curves contained in $\Gamma_p$ and $D''^t(p)=D^t(p)-D'^t(p)$.
Let $\eta'_p$ be the cardinality of the set $\{t=1,\dots,\eta_p|D'^t(p)\neq0\}$ and $\eta''_p=\eta_p-\eta'_p$.

\begin{defn}[Index $j$]
Let $j_{b,a}(F_p)$ (resp. $j_{b,a}^t(F_p)$, $j'^t_{b,a}(F_p)$, $j''^t_{b,a}(F_p)$) be the number of irreducible curves of genus $b$ with self-intersection number $-an$ contained in $\widetilde{R}_v(p)$ (resp. $D^t(p)$, $D'^t(p)$, $D''^t(p)$).
Put 
$$
j^{t}_{\bullet,a}(F_p)=\sum_{b\ge 0}j^{t}_{b,a}(F_p),\quad j^{t}_{b,\bullet}(F_p)=\sum_{a\ge 0}j^{t}_{b,a}(F_p), \quad j_{b,a}(F_p)=\sum_{t\ge 1}j^{t}_{b,a}(F_p).
$$
Similarly, we define $j^{t}(F_p)=j^{t}_{\bullet,\bullet}(F_p)$, $j'^{t}_{\bullet,a}(F_p)$, $j''^{t}_{\bullet,a}(F_p)$, etc.
Clearly, we have $j''_{b,\bullet}(F_p)=0$ for any $b\ge 1$ by the definition of $D''^t(p)$.

Rearranging the index if necessary, we may assume that 
$D'^t(p)=\sum_{k=1}^{j'^{t}(F_p)}L^{t,k}$, $D''^t(p)=\sum_{k=j'^{t}(F_p)+1}^{j^{t}(F_p)}L^{t,k}$.
Put $L'^{t,k}=L^{t,k}$, $L''^{t,k}=L^{t,j'^t(F_p)+k}$, $C'^{t,k}=C^{t,k}$, $C''^{t,k}=C^{t,j'^t(F_p)+k}$.

Let $\alpha_0^{+}(F_p)$ be the ramification index of 
$\widetilde{\varphi}:\widetilde{R}_h\to B$ over $p$ and put $\alpha_0^{-}(F_p)=\alpha_0(F_p)-\alpha_0^{+}(F_p)$.
It is clear that $\varepsilon(F_p)=j_{0,1}(F_p)$ and $\alpha_0^{-}(F_p)=\sum_{b\ge 0}(2b-2)j_{b,\bullet}(F_p)+2\varepsilon(F_p)$.

Let $\overline{\eta}_p$ be the number of $t=1,\ldots,\eta_p$ such that $j^{t}(F_p)=j''^{t}_{0,1}(F_p)$ and $\widehat{\eta}_p=\eta''_p-\overline{\eta}_p$.
\end{defn}

\begin{defn}[Vertical type singularity]
Let $x$ be a singular point of $R$. For $t=1,\dots, \eta_p$ and $u\ge 1$, $x$ is a {\em $(t,u)$-vertical type singularity} or simply a $u$-vertical type singularity if the number of $C^{t,k}$'s whose proper transforms pass through $x$ is $u$.
If $x$ is a $(t,u)$-vertical type singularity and the multiplicity of it belongs to $n\mathbb{Z}$ (resp. $n\mathbb{Z}+1$), 
we call it a $(t,u)$-vertical $n\mathbb{Z}$ type singularity (resp. $(t,u)$-vertical $n\mathbb{Z}+1$ type singularity).

Let $\iota^{t,(u)}(F_p)$, $\kappa^{t,(u)}(F_p)$ respectively be the number of $(t,u)$-vertical $n\mathbb{Z}$, $n\mathbb{Z}+1$ type singularities over $p$
and put 
$$
\iota^{t}(F_p)=\sum_{u\ge1}(u-1)\iota^{t,(u)}(F_p), \quad \kappa^{t}(F_p)=\sum_{u\ge1}(u-1)\kappa^{t,(u)}(F_p),
$$
$\iota(F_p)=\sum_{t=1}^{\eta_p}\iota^{t}(F_p)$ and $\kappa(F_p)=\sum_{t=1}^{\eta_p}\kappa^{t}(F_p)$.
Let $\iota^{t,(u)}_{k}(F_p)$, $\kappa^{t,(u)}_{k}(F_p)$ respectively be the number of $(t,u)$-vertical type singularities with multiplicity $kn$, $kn+1$ and we define $\iota^{t}_{k}(F_p)$, $\kappa^{t}_{k}(F_p)$, $\iota_{k}(F_p)$ and $\kappa_{k}(F_p)$ similarly.
\end{defn}

\begin{defn}[Indices $\alpha'$, $\alpha''$]
We say that a singular point $x$ of $R$ is {\em involved in} $D^t(p)$ if there exists $C^{t,k}$ such that it or its proper transform passes through $x$ or it contracts to $x$.
A singular point $x$ of $R$ is involved in $\widetilde{R}_v(p)$ if it is involved in $D^t(p)$ for some $t$.
Let $\alpha'_k(F_p)$ (resp. $\alpha''_k(F_p)$) denotes the number of singularities with multiplicity $kn$ or $kn+1$ over $p$ not involved in $\widetilde{R}_v(p)$ (resp. involved in $\widetilde{R}_v(p)$).
Clearly, we have $\alpha_k(F_p)=\alpha'_k(F_p)+\alpha''_k(F_p)$.
Let $\alpha''^{t}_k(F_p)$ denotes the number of singularities with multiplicity $kn$ or $kn+1$ over $p$ involved in $D^t(p)$.
Then, we have $\alpha''_k(F_p)=\sum_{t=1}^{\eta_p}\alpha''^{t}_k(F_p)$ by the definition of the decomposition $\widetilde{R}_v(p)=D^{1}(p)+\cdots+D^{\eta_p}(p)$.
Let $\alpha^{n\mathbb{Z}}_k(F_p)$, $\alpha^{n\mathbb{Z}+1}_k(F_p)$ respectively denote the number of singularities with multiplicity $kn$, $kn+1$ over $p$. Similarly, we define $\alpha''^{n\mathbb{Z}}_k(F_p)$, $\alpha''^{n\mathbb{Z}+1}_k(F_p)$, etc.
\end{defn}

\begin{defn}[Singularity of type $(i\to i)$]
Suppose that $n=2$. 
If the exceptional curve $E_{x}$ of the blow-up at a singularity $x$ of $R$ with multiplicity $2k+1$ contains only one singularity $y$, then the multiplicity at $y$ is $2k+2$ and $E_{x}$ contributes to $j''_{0,1}(F_p)$. 
Conversely, the exceptional curve $E$ contributing to $j''_{0,1}(F_p)$ has such a pair $(x,y)$.
Then we call the pair $(x,y)$ a singularity of type $(2k+1\to 2k+1)$ (cf.\ \cite{pi1}).
Let $\alpha_{(2k+1\to 2k+1)}(F_p)$ be the number of singularities of type $(2k+1\to 2k+1)$ over $p$ (i.e., $s_{2k+1}(F_p)$ in the notation of \cite{pi1}).
Then we have 
\begin{equation} \label{n2j''01eq}
j''_{0,1}(F_p)=\sum_{k\ge 1}\alpha_{(2k+1\to 2k+1)}(F_p).
\end{equation}
We decompose 
$$
\alpha_{(2k+1\to 2k+1)}(F_p)=\alpha^{\mathrm{tr}}_{(2k+1\to 2k+1)}(F_p)+\alpha^{\mathrm{co}}_{(2k+1\to 2k+1)}(F_p)
$$ and 
$$
\alpha^{\mathrm{co}}_{(2k+1\to 2k+1)}(F_p)=\alpha^{\mathrm{co},0}_{(2k+1\to 2k+1)}(F_p)+\alpha^{\mathrm{co},1}_{(2k+1\to 2k+1)}(F_p)
$$
 as follows.
Let $\alpha^{\mathrm{tr}}_{(2k+1\to 2k+1)}(F_p)$ be the number of  singularities of type $(2k+1\to 2k+1)$ over $p$ at which any local branch of $R_h$ intersects the fiber over $p$ transversely.
Let $\alpha^{\mathrm{co},0}_{(2k+1\to 2k+1)}(F_p)$ (resp.\ $\alpha^{\mathrm{co},1}_{(2k+1\to 2k+1)}(F_p)$) be the number of singularities $(x,y)$ of type $(2k+1\to 2k+1)$ over $p$ such that the proper transform of the vertical component passing through $x$ also passes through $y$ and is not contained in $R$ (resp.\ is contained in $R$).
\end{defn}

\begin{notation}
For a condition or a Roman numeral $\mathcal{P}$, 
we put $\delta_{\mathcal{P}}=1$ if the condition $\mathcal{P}$ holds or $\Gamma_p$ is a singular fiber of type $\mathcal{P}$, and $\delta_{\mathcal{P}}=0$ otherwise.
\end{notation}

Let $C=C^{t,k}$ and assume that it is smooth. If $C$ is on $W_i$, we drop the index and set $R=R_i$ for simplicity. 
Let $R^{\prime}=R-C$.
Let $x_1,\ldots,x_l$ be all the points of $C\cap R^{\prime}$.
We put $x_{i,1}=x_i$ and $m_{i,1}=m_i$. 
We define $\psi_{i,1}\colon W_{i,1}\rightarrow W$ to be the blow-up at $x_{i,1}$ and put $E_{i,1}=\psi_{i,1}^{-1}(x_{i,1})$ and $R_{i,1}=\psi_{i,1}^\ast R-n[m_{i,1}/n]E_{i,1}$. 
Inductively, we define $x_{i,j}$, $m_{i,j}$ to be the intersection point of the proper transform of $C$ and $E_{i,j-1}$, the multiplicity of $R_{i,j-1}$ at $x_{i,j}$, 
and if $m_{i,j}>1$, we define $\psi_{i,j}\colon W_{i,j}\to W_{i,j-1}$, $E_{i,j}$ and $R_{i,j}$ to be the blow-up at $x_{i,j}$, the exceptional curve for $\psi_{i,j}$ and $R_{i,j}=\psi_{i,j}^\ast R_{i,j-1}-n[m_{i,j}/n]E_{i,j}$, respectively. 
Put $i_{\mathrm{bm}}=\max\{j\mid m_{i,j}>1\}$, that is, the number of 
blowing-ups occuring over $x_i$.
We may assume that $i_{\mathrm{bm}}\geq (i+1)_{\mathrm{bm}}$ for 
$i=1,\dots,l-1$ after rearranging the index if necessary.
Put $t=R^{\prime}C$ and $c=\sum_{i=1}^l i_{\rm bm}$. 
If $C$ is a fiber $\Gamma$ of $\varphi$, $t$ is the number of branch points $r$. 
If $C$ is an exceptional curve, $t$ is the multiplicity of $R$ at the point to which $C$ is contracted. 
Clearly, $c$ is the number of blow-ups on $C$.
Set $d_{i,j}=[m_{i,j}/n]$. Then the following lemmas hold (cf.\ \cite{enoki}).

\begin{lem}[\cite{enoki}] \label{tclem}
We have
\begin{equation}
\frac{t+c}{n}=\sum_{i=1}^l \sum_{j=1}^{i_{\rm bm}} d_{i,j}.  \label{type1}
\end{equation}

\end{lem}

This is a special case of the following lemma:

\begin{lem}\label{easylem}
Let $f\colon S\to B$ be a primitive cyclic covering fibration of type $(g,h,n)$.
Let $C$ be a curve contained in $R$, $L$ the proper transform of $C$ on $\widetilde{W}$
 and $x_1$, $x_2$,\dots, $x_c$ all the singularities of $R$ on $C$ $($including infinitely near ones $)$.
Put $m_i=\mathrm{mult}_{x_i}(R)$, $k_i=\mathrm{mult}_{x_i}(C)$ and $d_i=[m_i/n]$.
Then, we have
$$
\frac{RC-L^2}{n}=\sum_{i=1}^{c}k_id_i.
$$
\end{lem}

\begin{proof}
We may assume that $\psi_i$ is the blow-up at $x_i$ for $i=1,\dots, c$. 
Then, we can write $L=\widetilde{\psi}^{*}C-\sum_{i=1}^{c}k_i\mathbf{E}_i$ 
 and $\widetilde{R}=\widetilde{\psi}^{*}R-\sum_{i=1}^{N}nd_i\mathbf{E}_i$.
Thus, we have $\widetilde{R}L=RC-\sum_{i=1}^{c}nk_id_i$.
On the other hand, since $\widetilde{R}-L$ and $L$ are disjoint, we have $L^2=\widetilde{R}L$.
From these equalities, the assertion follows.
\end{proof}

\begin{lem} [\cite{enoki}] \label{mijlem}
The following hold.

\smallskip

\noindent
$(1)$ When $n\ge 3$, then $m_{i,j}\ge m_{i,j+1}$. 
When $n=2$, then $m_{i,j}+1\ge m_{i,j+1}$ with equality holds only if $m_{i,j-1}\in 2\mathbb{Z}$ $($if $j>1)$ and $m_{i,j}\in 2\mathbb{Z}+1$.

\smallskip

\noindent
$(2)$ If $m_{i,j-1}\in n\mathbb{Z}+1$ and $m_{i,j}\in n\mathbb{Z}$, then $m_{i,j} > m_{i,j+1}$.

\smallskip

\noindent
$(3)$ $m_{i,i_{\mathrm{bm}}}\in n\mathbb{Z}$.
\end{lem}

\begin{defn}
By using the datum $\{m_{i,j}\}$, one can construct a diagram as in Table \ref{mij}. We call it the {\it singularity diagram} of $C$.

\begin{table}[H] 
\begin{center}
\caption{singularity diagram}
 \begin{tabular}{cccc} \label{mij}
  $\#$& & & \\ \cline{1-1}
 \multicolumn{1}{|c|}{$(x_{1,1_{\rm bm}},m_{1,1_{\rm bm}})$}& & & \\ \cline{1-1}
 \multicolumn{1}{|c|}{}& & & $\#$\\ \cline{4-4}                     
 \multicolumn{1}{|c|}{$\cdots$}& & & \multicolumn{1}{|c|}{$(x_{l,l_{\rm bm}},m_{l,l_{\rm bm}})$}\\ \cline{4-4}
 \multicolumn{1}{|c|}{}&$\cdots$ & & \multicolumn{1}{|c|}{$\cdots$}\\ \cline{1-4}
 \multicolumn{1}{|c|}{$(x_{1,1},m_{1,1})$}&$\cdots$ & & \multicolumn{1}{|c|}{$(x_{l,1},m_{l,1})$}\\ \cline{1-4}
 \end{tabular}
\end{center}
\end{table}

\noindent 
On the top of the $i$-th column (indicated by $\#$ in Table~\ref{mij}), we write $\#=(i_{\rm max}-i_{\rm bm})$ if $i_{\rm bm}< i_{\rm max}$ and leave it blank when $i_{\rm bm}=i_{\rm max}$.
We say that the singularity diagram of $C$ is {\em of type} $0$ (resp. {\em of type} $1$) if $C\not \subset R$ (resp. $C\subset R$).
\end{defn}

\begin{defn}
Let $\mathcal{D}^{t,k}$ be the singularity diagram of $C^{t,k}$.
We call $\mathcal{D}^{t,1}, \mathcal{D}^{t,2},\ldots, \mathcal{D}^{t,j^{t}(F_p)}$ a {\em sequence of singularity diagrams associated with} $D^{t}(p)$.
\end{defn}

Then the following lemma is clear.

\begin{lem} \label{connlem}
Let $\mathcal{D}^{t,1}, \mathcal{D}^{t,2},\ldots, \mathcal{D}^{t,j^{t}(F_p)}$ be a sequence of singularity diagrams associated with $D^{t}(p)$.
Let $l^{t,k}:=\#(R'\cap C^{t,k})$ and $(x_{i,j}^{t,k},m_{i,j}^{t,k})$, $i=1,\dots,l^{t,k}$, $j=1,\dots,i_{\mathrm{bm}}$ denote entries of $\mathcal{D}^{t,k}$.
Let $(x_{i,j}^{t,p},m_{i,j}^{t,p})$ be a singularity on $C^{t,p}$ such that $m_{i,j}^{t,p}\in n\mathbb{Z}+1$, and $m_{i,j-1}^{t,p}\in n\mathbb{Z}$ when $j>1$.
Let $q>p$ be the integer such that $C^{t,q}$ is the exceptional curve for the blow-up at $x_{i,j}^p$.
Then, for every $1\le p'\le p$, $i'$, $j'$ satisfying $(x_{i',j'}^{t,p'},m_{i',j'}^{t,p'})=(x_{i,j}^{t,p},m_{i,j}^{t,p})$, 
the diagram $\mathcal{D}^{t,q}$ has $(x_{i',j'+1}^{t,p'},m_{i',j'+1}^{t,p'})$ as an entry in the bottom row.
\end{lem}

\begin{exa} \label{diagexa}
Suppose that $t$ contributes to $\overline{\eta}_p$.
Then $C^{t,k}=C''^{t,k}$ is a $(-1)$-curve and blown up $n-1$ times for any $k$.

\smallskip

\noindent
(1) If $n=2$, then the point to which $C^{t,1}$ is contracted is a singularity of type $(m\to m)$ for some odd integer $m$.
Indeed, $R'C^{t,1}=m$ and from Lemma~\ref{tclem}, the singularity diagram of $C^{t,1}$ is the following: 
 
\begin{table}[H] 
\begin{center}
 \begin{tabular}{c}
 \\ \cline{1-1}
 \multicolumn{1}{|c|}{$m+1$} \\ \cline{1-1}
 \multicolumn{1}{c}{\lower1ex\hbox{$\mathcal{D}^{t,1}$}}
 \end{tabular}
\end{center}
\end{table}

\noindent
where we drop the symbol indicating the singular point on $C^{t,1}$ for simplicity.
Since $m+1$ is even, we have $j^{t}(F_p)=1$.
This observation gives us 
\begin{equation} \label{n2etaeq}
\overline{\eta}_p=\sum_{k\ge 1}\left(\alpha^{\mathrm{tr}}_{(2k+1\to 2k+1)}(F_p)+\alpha^{\mathrm{co},0}_{(2k+1\to 2k+1)}(F_p)\right).
\end{equation}

\smallskip

\noindent
(2) Suppose that $n=3$. Let $m$ be the multiplicity of the singular point to which $C^{t,1}$ is contracted.
Then $R'C^{t,1}=m$ and from Lemma~\ref{tclem}, all possible singularity diagram of $C^{t,1}$ are the following: 

\begin{table}[H]
\begin{center}
\begin{tabular}{c}

 \begin{minipage}{0.3\hsize}
(i)
 \begin{tabular}{cc}
   & \\ \cline{1-2}
 \multicolumn{1}{|c|}{$n_1$} & \multicolumn{1}{|c|}{$n_2$}\\ \cline{1-2}
 \multicolumn{2}{c}{\lower1ex\hbox{$\mathcal{D}^{t,1}$}}
 \end{tabular}
 \end{minipage}

\begin{minipage}{0.3\hsize}
(ii)
 \begin{tabular}{c}
   \\ \cline{1-1}
 \multicolumn{1}{|c|}{$n_2$}\\ \cline{1-1}
 \multicolumn{1}{|c|}{$n_1$}\\ \cline{1-1}
 \multicolumn{1}{c}{\lower1ex\hbox{$\mathcal{D}^{t,1}$}}
 \end{tabular}
 \end{minipage}

\begin{minipage}{0.3\hsize}
(iii)
 \begin{tabular}{c}
   \\ \cline{1-1}
 \multicolumn{1}{|c|}{$n_1$}\\ \cline{1-1}
 \multicolumn{1}{|c|}{$m_1$}\\ \cline{1-1}
 \multicolumn{1}{c}{\lower1ex\hbox{$\mathcal{D}^{t,1}$}}
 \end{tabular}
 \end{minipage}

\end{tabular}
\end{center}
\end{table}

\noindent
where the integers $n_i\in 3\mathbb{Z}$ and $m_i\in 3\mathbb{Z}+1$ satisfy that
$m+2=n_1+n_2$ in the case (i), $m+2=n_1+n_2$ and $n_2\le n_1$ in the case (ii), $m+3=m_1+n_1$  and $n_1<m_1$ in the case (iii).
If the diagram $\mathcal{D}^{t,1}$ is (i) or (ii), then $j^{t}(F_p)=1$ since there are no $3\mathbb{Z}+1$ type singularities on $C^{t,1}$.
If the diagram $\mathcal{D}^{t,1}$ is (iii), then $j^{t}(F_p)>1$ and
 the singularity diagram $\mathcal{D}^{t,2}$ of $C^{t,2}$ which is obtained by the blow-up at the singularity with multiplicity $m_1$ is (i) or (ii) from Lemma~\ref{connlem}.
Thus we have $j^{t}(F_p)=2$.

\smallskip

\noindent
(3) When $n\ge 4$, then the number $j^{t}(F_p)$ is not bounded. 
For example, we can consider the following sequence of singularity diagrams associated with $D^{t}(p)$ when $n=4$:

\begin{table}[H]
\begin{center}
\begin{tabular}{c}

 \begin{minipage}{0.2\hsize}
 \begin{tabular}{cc}
   & \\ \cline{1-1}
\multicolumn{1}{|c|}{$n_1$} & \\ \cline{1-2}
 \multicolumn{1}{|c|}{$m_1$} & \multicolumn{1}{|c|}{$n_2$}\\ \cline{1-2}
 \multicolumn{2}{c}{\lower1ex\hbox{$\mathcal{D}^{t,1}$}}
 \end{tabular}
 \end{minipage}

\begin{minipage}{0.2\hsize}
 \begin{tabular}{cc}
   & \\ \cline{1-1}
 \multicolumn{1}{|c|}{$n_3$} & \\ \cline{1-2}
 \multicolumn{1}{|c|}{$m_2$} & \multicolumn{1}{|c|}{$n_4$}\\ \cline{1-2}
 \multicolumn{1}{c}{\lower1ex\hbox{$\mathcal{D}^{t,2}$}}
 \end{tabular}
 \end{minipage}

\begin{minipage}{0.18\hsize}
 \begin{tabular}{c}
  $\cdots$
 \end{tabular}
 \end{minipage}

\begin{minipage}{0.2\hsize}
 \begin{tabular}{cc}
   & \\ \cline{1-1}
 \multicolumn{1}{|c|}{$n_{2N-1}$} & \\ \cline{1-2}
 \multicolumn{1}{|c|}{$m_{N}$} & \multicolumn{1}{|c|}{$n_{2N}$}\\ \cline{1-2}
 \multicolumn{1}{c}{\lower1ex\hbox{$\mathcal{D}^{t,N}$}}
 \end{tabular}
 \end{minipage}

\end{tabular}
\end{center}
\end{table}

\noindent
where $n_k\in 4\mathbb{Z}$, $m_k\in 4\mathbb{Z}+1$ and $C^{t,k}$, $k\ge 2$ is the exceptional curve obtained by the blow-up at the multiplicity on $C^{t,k-1}$ with multiplicity $m_{k-1}$.
From Lemma~\ref{tclem}, we have $m_{k}+4=m_{k+1}+n_{2k+1}+n_{2k+2}$ for any $k\ge 1$.

\end{exa}

Recall that the gonality $\mathrm{gon}(C)$ of a non-singular projective curve $C$ is the 
minimum of the degree of morphisms onto $\mathbb{P}^1$.
The gonality $\mathrm{gon}(f)$ of a fibered surface $f\colon S\to B$ is defined to be that of a general 
fiber (cf. \cite{cliff}).

\begin{prop}\label{gonprop}
Let $\theta \colon F \to \Gamma$ be a totally ramified covering of degree $n$ between smooth projective curves branched over $r$ points.
If $r\ge 2n\,\mathrm{gon}(\Gamma)$, then $\mathrm{gon}(F)=n\,\mathrm{gon}(\Gamma)$.
In particular, the gonality of a primitive cyclic covering fibration of type $(g,h,n)$ is $n\,\mathrm{gon}(\varphi)$, when $r\ge 2n\,\mathrm{gon}(\varphi)$.
\end{prop}

\begin{proof}
Assume contrary that $F$ 
 has a morphism onto $\mathbb{P}^1$ of degree $k<n\,\mathrm{gon}(\Gamma)$.
This together with the covering $\theta\colon F\to \Gamma$ defines a morphism $\Phi:F\to \mathbb{P}^1\times \Gamma$.  
If $\Phi$ is of degree $m$ onto the image $\Phi(F)$, then $m$ is a 
 common divisor of $n,k$ and the arithmetic genus of 
 $\Phi(F)$ is $(n/m-1)k/m+(h-1)n/m+1$ by the genus formula.
Now, let $F'$ be the normalization of $\Phi(F)$.
Since the covering $F\to \Gamma$ factors through $F'$, we 
 see that the induced covering $F'\to \Gamma$ of degree $n/m$ is a 
 totally ramified covering branched over $r$ points.
Then, by the Hurwitz formula, we have $2g(F')-2=(2h-2)n/m+(n/m-1)r$. 
Since the genus $g(F')$ of $F'$ is not bigger than the arithmetic genus 
 of $\Phi(F)$, we get $r\leq 2(k/m)$ when $n>m$, which is impossible, since $r\geq 2n\,\mathrm{gon}(\Gamma)$ and $k<n\,\mathrm{gon}(\Gamma)$.
Thus, we get $n=m$. Then $F'$ is isomorphic to $\Gamma$ 
and therefore the morphism $F\to \mathbb{P}^1$ factors through $\Gamma$.
Hence we have $k\ge n\,\mathrm{gon}(\Gamma)$ by the difinition of the gonality of $\Gamma$, 
which contradicts $k<n\,\mathrm{gon}(\Gamma)$.
A more careful study shows that any gonality pencil of $F$ is the pull-back of a gonality pencil of $\Gamma$ when 
 $r> 2n\,\mathrm{gon}(\Gamma)$.
\end{proof}

\section{Primitive cyclic covering fibrations of type $(g,1,n)$}
Let $f\colon S\to B$ be a primitive cyclic covering fibration of type $(g,1,n)$.
Since $\varphi\colon W\to B$ is a relatively minimal elliptic surface, 
$K_{\varphi}$ is numerically equivalent to $\left(\chi_{\varphi}+\sum_{p\in B}\left(1-\frac{1}{m_p}\right)\right)\Gamma$
by the canonical bundle formula, where $m_p$ denotes the multiplicity of the fiber $\Gamma_p$ of $\varphi$ over $p$. In particular, we have $K_{\varphi}^2=0$.
For $p\in B$, we put $\nu(F_p)=1-1/m_p$ and $\nu=\sum_{p\in B}\nu(F_p)$. 
Then, we have $K_{\varphi}R=(\chi_{\varphi}+\nu)r$.
Combining these equalities with \eqref{KReq}, \eqref{kfeq}, \eqref{chifeq} and \eqref{efeq},
we get the following lemma:

\begin{lem} \label{locallem}
The following equalities hold.

\begin{align*}
K_f^2=&\;\sum_{k\ge 1}\left((n+1)(n-1)k-n\right)\alpha_k+\frac{(n-1)^2}{n}(\alpha_0-2\varepsilon) \\
&+\frac{(n+1)(n-1)r}{n}(\chi_{\varphi}+\nu)+\varepsilon. \\
\chi_f=&\;\frac{1}{12}(n-1)(n+1)\sum_{k\ge 1}k\alpha_k+\frac{(n-1)(2n-1)}{12n}(\alpha_0-2\varepsilon) \\
&+\frac{(n+1)(n-1)r}{12n}(\chi_{\varphi}+\nu)+n\chi_{\varphi}. \\
e_f=&\;(n-1)\alpha_0+n\sum_{k\ge 1}\alpha_k-(2n-1)\varepsilon+12n\chi_{\varphi}. \\
\end{align*}
\end{lem}

For $p\in B$, we put $\chi_{\varphi}(F_p)=e_{\varphi}(\Gamma_p)/12$ and
\begin{align*}
K_f^2(F_p)=&\;\sum_{k\ge 1}\left((n+1)(n-1)k-n\right)\alpha_k(F_p)+\frac{(n-1)^2}{n}(\alpha_0(F_p)-2\varepsilon(F_p)) \\
&+\frac{(n+1)(n-1)r}{n}(\chi_{\varphi}(F_p)+\nu(F_p))+\varepsilon(F_p), \\
\chi_f(F_p)=&\;\frac{1}{12}(n-1)(n+1)\sum_{k\ge 1}k\alpha_k(F_p)+\frac{(n-1)(2n-1)}{12n}(\alpha_0(F_p)-2\varepsilon(F_p)) \\
&+\frac{(n+1)(n-1)r}{12n}(\chi_{\varphi}(F_p)+\nu(F_p))+n\chi_{\varphi}(F_p), \\
e_f(F_p)=&\;(n-1)\alpha_0(F_p)+n\sum_{k\ge 1}\alpha_k(F_p)-(2n-1)\varepsilon(F_p)+12n\chi_{\varphi}(F_p). \\
\end{align*}
Then, the following slope equality holds:
\begin{thm} 
Let $f\colon S\to B$ be a primitive cyclic covering fibration of type $(g,1,n)$.
Then
$$
K_f^2=\lambda_{g,1,n}\chi_f+\sum_{p\in B}\mathrm{Ind}(F_p),
$$
where $\lambda_{g,1,n}:=12(n-1)/(2n-1)$ and $\mathrm{Ind}(F_p)$ is defined by
\begin{align*}
\mathrm{Ind}(F_p)=&n\sum_{k\ge 1}\left(\frac{(n+1)(n-1)}{2n-1}k-1\right)\alpha_k(F_p)+\frac{n-1}{2n-1}\left((n+1)r-12n\right)\chi_{\varphi}(F_p) \\
&+\frac{(n+1)(n-1)r}{2n-1}\nu(F_p)+\varepsilon(F_p). \\
\end{align*}
Moreover, if $r\ge \displaystyle{\frac{12n}{n+1}}$, then $\mathrm{Ind}(F_p)$ is non-negative for any $p\in B$.
\end{thm}

\begin{proof}
Since $K_f^2=\sum_{p\in B}K_f^2(F_p)$, $\chi_f=\sum_{p\in B}\chi_f(F_p)$ and $K_{f}^2(F_p)-\lambda_{g,1,n}\chi_f(F_p)=\mathrm{Ind}(F_p)$, the claim follows.
\end{proof}

For an oriented compact real $4$-dimensional manifold $X$, the signature 
$\mathrm{Sign}(X)$ of $X$ is defined to be the number of positive eigenvalues 
minus the number of negative eigenvalues of the intersection form on $H^2(X)$.
From Lemma~\ref{locallem}, we observe the local 
concentration of $\mathrm{Sign}(S)$ to a finite number of fiber germs.

\begin{cor}[cf. \cite{ak}]
Let $f\colon S\to B$ be a primitive cyclic covering fibration of type $(g,1,n)$.
Then
$$
\mathrm{Sign}(S)=\sum_{p\in B}\sigma(F_p),
$$
where $\sigma(F_{p})$ is defined by
\begin{align*}
\sigma(F_p)=&n\sum_{k\ge 1}\left(\frac{(n+1)(n-1)}{3}k-1\right)\alpha_k(F_p)+\left(\frac{(n-1)(n+1)r}{3n}-8n\right)\chi_{\varphi}(F_p) \\
&+\frac{(n+1)(2n-1)}{3n}\varepsilon(F_p)+\frac{(n+1)(n-1)r}{3n}\nu(F_p)-\frac{(n+1)(n-1)}{3n}\alpha_0(F_p). \\
\end{align*}

\end{cor}

\begin{proof}
By the index theorem (cf. \cite[p.~126]{pag}), we have 
$$
\mathrm{Sign}(S)=\sum_{p+q\equiv 0 (\mathrm{mod} 2)}h^{p,q}(S)=K_f^2-8\chi_f.
$$
On the other hand, we can see that
$$
\sigma(F_p)=K_f^2(F_p)-8\chi_f(F_p)
$$
by a computation.
\end{proof}

\section{Upper bound of the slope: the case of type $(g,1,n)$}

In this section, we prove the following theorem:

\begin{thm} \label{h1upperthm}
Let $f\colon S\to B$ be a primitive cyclic covering fibration of type $(g,1,n)$.

\smallskip

\noindent
$(1)$ If $n\ge 4$, then we have
$$
K_f^2\le \left(12-\frac{12n^2}{r(n-1)(n+1)}\right)\chi_f.
$$

\smallskip

\noindent
$(2)$ If $n=3$ and $g\ge 7$, then we have
$$
K_f^2\le \left(12-\frac{24}{4g-17}\right)\chi_f.
$$

\smallskip

\noindent
$(3)$ If $n=2$ and $g\ge 3$, then we have
$$
K_f^2\le \left(12-\frac{2}{g-2}\right)\chi_f.
$$

\end{thm}

\begin{cor}
Let $f\colon S\to B$ be a relatively minimal bielliptic fibered surface of genus $g\ge 6$. Then, we have
$$
K_f^2\le \left(12-\frac{2}{g-2}\right)\chi_f.
$$
\end{cor}

\begin{proof}
Let $f\colon S\to B$ be a relatively minimal fibered surface of genus $g$ whose general fiber $F$ is a double cover of a smooth curve $\Gamma$ of genus $h$.
If $g>4h+1$, an involution of the general fiber $F$ of $f$ over $\Gamma$ is unique.
Then, the fibration $f$ has a global involution since it is relatively minimal (cf.\ \cite{dbl}).
Hence $f$ is a primitive cyclic covering fibration of type $(g,h,2)$.
In particular, a relatively minimal bielliptic fibered surface of genus $g\ge 6$ is a primitive cyclic covering fibration of type $(g,1,2)$.
\end{proof}

Let $f\colon S\to B$ be a primitive cyclic covering fibration of type $(g,1,n)$.
We fix $p\in B$. Let $m=m_p$ be the multiplicity of the fiber $\Gamma_p$ of $\varphi$ over $p$.
Since $h=1$, we have $j_{b,\bullet}(F_p)=0$ for any $b\ge 2$.
From the classification of singular fibers of relatively minimal elliptic surfaces (\cite{Kod}), we have the following lemma for $u$-vertical type singularities:

\begin{lem}\label{3verlem}
There exist no $u$-vertical type singularities of $R$ for $u\ge 4$.
All possible $3$-vertical type singularities are as follows.

\smallskip

\noindent
Type $(\mathrm{I\hspace{-.1em}I}):$ $\Gamma_p$ is a singular fiber of type $(\mathrm{I\hspace{-.1em}I})$ in the Kodaira's table $($\cite{Kod}$) ($i.e., it is a singular rational curve with one cusp$)$ and it is contained in $R$. 
The cusp on $\Gamma_p$ is a singularity of type $n\mathbb{Z}+1$ and the singularity at which the proper transform of $\Gamma_p$ and the exceptional curve $E_1$ for the blow-up at the cusp intersect is also of type $n\mathbb{Z}+1$.
Then, the proper transforms of $\Gamma_p$ and $E_1$ and the exceptional curve $E_2$ for the blow-up at this singularity form a $3$-vertical type singularity.

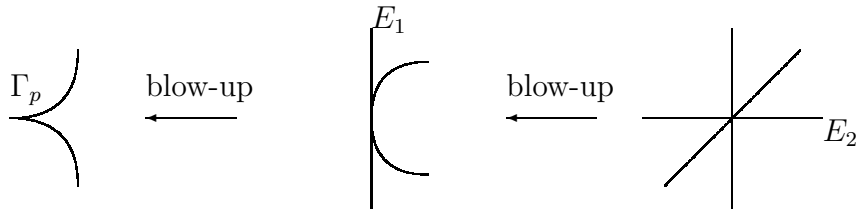
\setlength\unitlength{0.3cm}
\begin{figure}[H]
\begin{center}
\begin{tabular}{c}

 \begin{minipage}{0.7\hsize}
 \begin{picture}(20,5)
 \qbezier(0,0)(3,0)(3,3)
 \qbezier(0,0)(3,0)(3,-3)
 \qbezier(16,0)(16,2.5)(18.5,2.5)
 \qbezier(16,0)(16,-2.5)(18.5,-2.5)
 \put(10,0){\vector(-1,0){4}}
 \put(26,0){\vector(-1,0){4}}
 \put(16,-4){\line(0,1){8}}
 \put(32,-4){\line(0,1){8}}
 \put(28,0){\line(1,0){8}}
 \qbezier(29,-3)(32,0)(35,3)
 \put(0,1){$\Gamma_p$}
 \put(16,4){$E_1$}
 \put(36,-1){$E_2$}
 \put(6,1){\rm{blow-up}}
 \put(22,1){\rm{blow-up}}
 \end{picture}
 \end{minipage}

\end{tabular}
\end{center}
\end{figure}

\ \\

\noindent
Type $(\mathrm{I\hspace{-.1em}I\hspace{-.1em}I}):$ $\Gamma_p$ is a singular fiber of type $(\mathrm{I\hspace{-.1em}I\hspace{-.1em}I})$ in the Kodaira's table $($i.e., it consists of two nonsingular rational curves intersecting each other at one point of  order two$)$ and it is contained in $R$. 
The singularity on $\Gamma_p$ is a singularity of type $n\mathbb{Z}+1$.
Then, the proper transforms of $\Gamma_p$ and the exceptional curve $E_1$ for the blow-up at this singularity form a $3$-vertical type singularity.

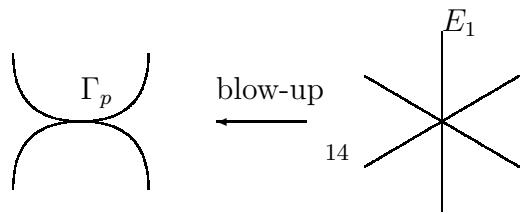
\setlength\unitlength{0.3cm}
\begin{figure}[H]
\begin{center}
\begin{tabular}{c}

 \begin{minipage}{0.7\hsize}
 \begin{picture}(20,5)
 \qbezier(8,0)(11,0)(11,3)
 \qbezier(8,0)(11,0)(11,-3)
 \qbezier(8,0)(5,0)(5,3)
 \qbezier(8,0)(5,0)(5,-3)
 \qbezier(20.6,-2)(24,0)(27.4,2)
 \qbezier(20.6,2)(24,0)(27.4,-2)
 \put(18,0){\vector(-1,0){4}}
 \put(24,-4){\line(0,1){8}}
 \put(8,1){$\Gamma_p$}
 \put(24,4){$E_1$}
 \put(14,1){\rm{blow-up}}
 \end{picture}
 \end{minipage}

\end{tabular}
\end{center}
\end{figure}

\ \\

\noindent
Type $(\mathrm{I\hspace{-.1em}V}):$ $\Gamma_p$ is a singular fiber of type $(\mathrm{I\hspace{-.1em}V})$ in the Kodaira's table $($i.e.\ it consists of three nonsingular rational curves intersecting one another at one point transversely$)$ and it is contained in $R$. 
The singularity on $\Gamma_p$ is a $3$-vertical type singularity.

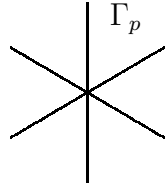
\setlength\unitlength{0.3cm}
\begin{figure}[H]
\begin{center}
\begin{tabular}{c}

 \begin{minipage}{0.7\hsize}
 \begin{picture}(20,5)
 \qbezier(12.6,-2)(16,0)(19.4,2)
 \qbezier(12.6,2)(16,0)(19.4,-2)
 \put(16,-4){\line(0,1){8}}
 \put(17,3){$\Gamma_p$}
 \end{picture}
 \end{minipage}

\end{tabular}
\end{center}
\end{figure}

\ \\

\noindent
In particular, we have $\iota^{(u)}(F_p)=\kappa^{(u)}(F_p)=0$ for $u\ge 4$ 
and $0\le \iota^{(3)}(F_p)+\kappa^{(3)}(F_p)\le 1$.
\end{lem}

Next, we give a lower bound of $\alpha^+_0(F_p)$ by using $\iota(F_p)$ and $\kappa(F_p)$.

\begin{lem}\label{alpha0lem}
We have 
$$
\alpha^+_0(F_p)\ge \left(1-\frac{1}{m}\right)r+(n-2)(\iota(F_p)+2\kappa(F_p))+\beta_p.
$$
where $\beta_p:=\delta_{n\neq 2}\left((n-7)\delta_{\mathrm{I\hspace{-.1em}I}}-(n+1)\delta_{\mathrm{I\hspace{-.1em}I\hspace{-.1em}I}}-2\delta_{\mathrm{I\hspace{-.1em}V}}\right)\iota^{(3)}(F_p)+\delta_{n=2}\sum_{k\ge 1}2k\alpha^{\mathrm{co}}_{(2k+1\to 2k+1)}(F_p)$.
\end{lem}

\begin{proof}
Let $\widetilde{\Gamma}_p=m\widetilde{D}_p$ and $\widetilde{D}_p=\sum_{i}m_iG_i$ the irreducible decomposition. Then we have

\begin{align*}
\alpha^+_0(F_p)=&\; r-\#(\mathrm{Supp}(\widetilde{R}_h)\cap \mathrm{Supp}(\widetilde{\Gamma}_p))\\
                   \ge& r-\sum_{i}\widetilde{R}_hG_i \\
                     =&\; \left(1-\frac{1}{m}\right)r+\sum_{i}(m_i-1)\widetilde{R}_hG_i.
\end{align*}
For a $(t,2)$ or $(t,3)$-vertical $n\mathbb{Z}$ type singularity $x$, we denote by $E^t_x$ the exceptional curve for the blow-up at $x$.
Let $m^t_x$ be the multiplicity of $\widetilde{D}_p$ along $\widehat{E}^t_x$, the proper transform of $E^t_x$ on $\widetilde{W}$.
Then, we have
$$
\sum_{i}(m_i-1)\widetilde{R}_hG_i\ge \sum_{t=1}^{\eta_p}\sum_{x\text{:$(t,u)$ $n\mathbb{Z}$, $u\ge 2$}}(m^t_x-1)\widetilde{R}_h\widehat{E}^t_x.
$$
If there exists a singular point of type $n\mathbb{Z}$ on $E^t_x$, we replace $E^t_x$ with the exceptional curve $E$ obtained by blowing up at this point. Repeating this procedure, we may assume that there exist no singular points of type $n\mathbb{Z}$ on $E^t_x$. If there exists  a singular point of type $n\mathbb{Z}+1$ on 
 $E^t_x$, the proper transform of the exceptional curve obtained by blowing 
 up at this point belongs to 
 other $D^u(p)$.
Since the multiplicity of $\widetilde{\Gamma}_p$ along it is not less 
 than $m^t_x>1$, we do not have to consider this situation. 
Thus, we may assume that there exist no singular points on $E^t_x$
and we have $\widetilde{R}_h\widehat{E}^t_x\ge n-2$ if $x$ is a $2$-vertical type singularity,
 and $\widetilde{R}_h\widehat{E}^t_x\ge n-3$ if $x$ is a $3$-vertical type singularity.
We can see that $\sum_{x\text{:$(t,2)$ $n\mathbb{Z}$}}m^t_x\ge 2\iota^t(F_p)+2\kappa^t(F_p)$ for any $t$ with $\iota^{t,(3)}(F_p)=0$. Thus, if $\iota^{(3)}(F_p)=0$, we have
$$
\alpha^+_0(F_p)\ge \left(1-\frac{1}{m}\right)r+(n-2)(\iota(F_p)+2\kappa(F_p)).
$$
If $\iota^{(3)}(F_p)=1$, then $m=1$ and $\Gamma_p$ is of type $(\mathrm{I\hspace{-.1em}I})$, $(\mathrm{I\hspace{-.1em}I\hspace{-.1em}I})$ or $(\mathrm{I\hspace{-.1em}V})$ from Lemma~\ref{3verlem}.
We may assume that $D'^1(p)\neq 0$.
Let $x_0$ be the $3$-vertical $n\mathbb{Z}$ type singularity over $p$.
Suppose that $\Gamma_p$ is of type~$(\mathrm{I\hspace{-.1em}I})$. 
Then, we can see that $m^1_{x_0}=6$ and $\sum_{x\text{:$(1,2)$ $n\mathbb{Z}$}}m^1_x\ge 2\iota^{1,(2)}(F_p)+2(\kappa^1(F_p)-1)$.
Then we have
\begin{align*}
\alpha^+_0(F_p)\ge& \left(1-\frac{1}{m}\right)r+5(n-3)+(n-2)(\iota^{1,(2)}(F_p)+2(\kappa^1(F_p)-1))\\
   &+(n-2)\sum_{t=2}^{\eta_p}(\iota^t(F_p)+2\kappa^t(F_p))\\
                     =&\; \left(1-\frac{1}{m}\right)r+n-7+(n-2)(\iota(F_p)+2\kappa(F_p)).
\end{align*}
Suppose that $\Gamma_p$ is of type~$(\mathrm{I\hspace{-.1em}I\hspace{-.1em}I})$. 
Then, we can see that $m^1_{x_0}=4$ and $\sum_{x\text{:$(1,2)$ $n\mathbb{Z}$}}m^1_x\ge 2\iota^{1,(2)}(F_p)+2(\kappa^1(F_p)-1)$.
Then we have
\begin{align*}
\alpha^+_0(F_p)\ge& \left(1-\frac{1}{m}\right)r+3(n-3)+(n-2)(\iota^{1,(2)}(F_p)+2(\kappa^1(F_p)-1))\\
   &+(n-2)\sum_{t=2}^{\eta_p}(\iota^t(F_p)+2\kappa^t(F_p))\\
                     =&\; \left(1-\frac{1}{m}\right)r-n-1+(n-2)(\iota(F_p)+2\kappa(F_p)).
\end{align*}
Suppose that $\Gamma_p$ is of type~$(\mathrm{I\hspace{-.1em}V})$. 
Then, we can see that $m^1_{x_0}=3$ and $\sum_{x\text{:$(1,2)$ $n\mathbb{Z}$}}m^1_x\ge 2\iota^{1,(2)}(F_p)+2\kappa^1(F_p)$.
Then we have
\begin{align*}
\alpha^+_0(F_p)\ge& \left(1-\frac{1}{m}\right)r+2(n-3)+(n-2)(\iota^{1,(2)}(F_p)+2\kappa^1(F_p))\\
   &+(n-2)\sum_{t=2}^{\eta_p}(\iota^t(F_p)+2\kappa^t(F_p))\\
                     =&\; \left(1-\frac{1}{m}\right)r-2+(n-2)(\iota(F_p)+2\kappa(F_p)).
\end{align*}
Suppose that $n=2$. For a $(2k+1\to 2k+1)$ singularity $(x,y)$, let $E_y$ denotes the exceptional curve for the blow-up at $y$ and $m_y$ the multiplicity of $\widetilde{\Gamma}_p$ along $\widehat{E}_y$, the proper transform of $E_y$.
Then we have
$$
\sum_{i}(m_i-1)\widetilde{R}_hG_i\ge \sum_{k\ge 1}\sum_{(x,y)\text{:$(2k+1\to 2k+1)$}}(m_y-1)\widetilde{R}_h\widehat{E}_y.
$$
By an argument similar to the above, we may assume that there are no singular points on $E_y$.
Then we have $\widetilde{R}_h\widehat{E}_y=2k$ for any $(2k+1\to 2k+1)$ singularity $(x,y)$.
On the other hand, we have $m_y\ge 2$ for any $(2k+1\to 2k+1)$ singularity $(x,y)$ involved in $\alpha^{\mathrm{co}}_{(2k+1\to 2k+1)}(F_p)$.
Thus, we obtain
$\sum_{(x,y)\text{:$(2k+1\to 2k+1)$}}(m_y-1)\widetilde{R}_h\widehat{E}_y
\ge 2k\alpha^{\mathrm{co}}_{(2k+1\to 2k+1)}(F_p)$.

\end{proof}

We can translate the index $\alpha''$ into other indices as follows.

\begin{lem} \label{alphaklem}
The following equalities hold.

\begin{align}
\sum_{k\ge 1}\alpha''_k(F_p)=&\; \eta''_p+\sum_{a\ge 1}anj'_{1,a}(F_p)+\sum_{a\ge 1}(an-2-\delta_{{}_m\mathrm{I}_1,\mathrm{I\hspace{-.1em}I}})j'_{0,a}(F_p) \label{alphakeq}\\
&+\sum_{a\ge 1}(an-1)j''_{0,a}(F_p)-\iota(F_p)-\kappa(F_p) \nonumber.
\end{align}

\begin{align}
\sum_{k\ge 1}k\alpha''_k(F_p)=&\; \gamma_p+\sum_{a\ge 1}aj_{\bullet,a}(F_p)
+\sum_{k\ge 1}k\left(\alpha_k^{n\mathbb{Z}+1}(F_p)-\iota_k(F_p)-\kappa_k(F_p)\right) \label{kalphakeq}.
\end{align}
where $\gamma_p=\sum_{t=1}^{\eta_p}\gamma^{t}_p$ 
and $\gamma^{t}_p$ is defined to be the following $(1)$, $(2)$, $(3)$$:$

\smallskip

\noindent
$(1)$ $\gamma^{t}_p=d^{t,1}$ if $D'^t(p)=0$, where $m^{t,1}$ is the multiplicity of the singularity to which $C^{t,1}$ contracts and $d^{t,1}=[m^{t,1}/n]$.

\smallskip

\noindent
$(2)$ $\gamma^{t}_p=\sum_{k=1}^{j'^t(F_p)}RC'^{t,k}/n$ if $D'^t(p)\neq 0$ and any $C'^{t,k}$ is smooth.

\smallskip

\noindent
$(3)$ $\gamma^{t}_p=r/mn-d'^{t,1}$ if $C'^{t,1}=(\Gamma_p)_{\mathrm{red}}$ is singular, where $m'^{t,1}$ is the multiplicity of the singular point of $R$ which is singular for $C'^{t,1}$ and $d'^{t,1}=[m'^{t,1}/n]$.
\end{lem}

\begin{proof}
Let $a'^{t,k}$, $a''^{t,k}$ be the integers such that $(L'^{t,k})^2=-a'^{t,k}n$, $(L''^{t,k})^2=-a''^{t,k}n$. 
If $C'^{t,k}$ is smooth, then $C'^{t,k}$ is blown up $a'^{t,k}n+(C'^{t,k})^2$ times.
If $C'^{t,k}$ is a singular rational curve, then $C'^{t,k}$ is blown up $a'^{t,k}n-3$ times.
Since $C''^{t,k}$ is a $(-1)$-curve, $C''^{t,k}$ is blown up $a''^{t,k}n-1$ times.
Hence, if $D'^t(p)\neq 0$ and every $C'^{t,k}$ is smooth (resp. $C'^{t,1}$ is singular rational), the number of singular points associated with $D^t(p)$ is $\sum_{k}\left(a'^{t,k}n+(C'^{t,k})^2\right)+$ (resp. $\sum_{k}\left(a'^{t,k}n-3\right)+$)
$\sum_{k}\left(a''^{t,k}n-1\right)-\iota^t(F_p)-\kappa^t(F_p)$. Namely, we have
$$
\sum_{k\ge 1}\alpha''^t_k(F_p)=\sum_{a\ge 1}anj'^t_{1,a}(F_p)+\sum_{a\ge 1}(an-2-\delta_{{}_m\mathrm{I}_1,\mathrm{I\hspace{-.1em}I}})j'^t_{0,a}(F_p)
+\sum_{a\ge 1}(an-1)j''^t_{0,a}(F_p)-\iota^t(F_p)-\kappa^t(F_p).
$$
If $D'^t(p)= 0$, we have
$$
\sum_{k\ge 1}\alpha''^t_k(F_p)=1
+\sum_{a\ge 1}(an-1)j''^t_{0,a}(F_p)-\iota^t(F_p)-\kappa^t(F_p).
$$
Summing up for $t=1,\dots,\eta_p$, we have \eqref{alphakeq}.

Let $r^{t,k}=RC'^{t,k}$, $m^{t,k}$ the multiplicity of $R$ at the point to which $C''^{t,k}$ is contracted and $d^{t,k}=[m^{t,k}/n]$. 
Let $x'^{t,k}_1,\dots,x'^{t,k}_{c'}$ (resp. $x''^{t,k}_1,\dots,x''^{t,k}_{c''}$) be all the singular points on $C'^{t,k}$ (resp. on $C''^{t,k}$), including infinitely near ones.
Put $m'^{t,k}_i=\mathrm{mult}_{x'^{t,k}_i}(R)$, $d'^{t,k}_i=[m'^{t,k}_i/n]$, $m''^{t,k}_i=\mathrm{mult}_{x''^{t,k}_i}(R)$ and $d''^{t,k}_i=[m''^{t,k}_i/n]$.
Applying Lemma~\ref{easylem} to $C'^{t,k}$ and $C''^{t,k}$, we get that
$r^{t,k}/n+a'^{t,k}=\sum_{i}d'^{t,k}_i$ if $C'^{t,k}$ is smooth,
$r^{t,1}/n+a'^{t,1}=d'^{t,1}_1+\sum_{i}d'^{t,1}_i$ if $C'^{t,1}$ is singular rational,
 and $d^{t,k}+a''^{t,k}=\sum_{i}d''^{t,k}_i$.
If $D'^t(p)\neq 0$ and every $C'^{t,k}$ is smooth, then 
\begin{align*}
\sum_{k}\frac{r^{t,k}}{n}+\sum_{a\ge 1}aj^t_{\bullet,a}(F_p)+\sum_{k\ge 1}k\alpha^{t,n\mathbb{Z}+1}_k(F_p)=&\;
\sum_{k}\left(\frac{r^{t,k}}{n}+a'^{t,k}\right)+\sum_{k}\left(d^{t,k}+a''^{t,k}\right)\\
=&\; \sum_{k}\sum_{i}d'^{t,k}_i+\sum_{k}\sum_{i}d''^{t,k}_i\\
=&\; \sum_{k\ge 1}k\left(\alpha''^t_k(F_p)+\iota^t_k(F_p)+\kappa^t_k(F_p)\right).
\end{align*}
Similarly, if $D'^t(p)\neq 0$ and $C'^{t,1}$ is singular rational, we have
$$
\frac{r}{nm}-d'^{t,1}_1+\sum_{a\ge 1}aj^t_{\bullet,a}(F_p)+\sum_{k\ge 1}k\alpha^{t,n\mathbb{Z}+1}_k(F_p)
=\sum_{k\ge 1}k\left(\alpha''^t_k(F_p)+\iota^t_k(F_p)+\kappa^t_k(F_p)\right).
$$
If $D'^t(p)=0$, then 
\begin{align*}
\sum_{a\ge 1}aj^t_{\bullet,a}(F_p)+\sum_{k\ge 1}k\alpha^{t,n\mathbb{Z}+1}_k(F_p)=&\;
\sum_{k}\left(d^{t,k}+a''^{t,k}\right)\\
=&\; \sum_{k}\sum_{i}d''^{t,k}_i\\
=&\; \sum_{k\ge 1}k\left(\alpha''^t_k(F_p)+\iota^t_k(F_p)+\kappa^t_k(F_p)\right)-d^{t,1}.
\end{align*}
Summing up for $t=1,\dots,\eta_p$, we get \eqref{kalphakeq}.
\end{proof}

\begin{lem} \label{gammaiotalem}
The following hold.

\begin{align*}
& \gamma_p\le \left(\frac{r}{n}-j'_{0,1}(F_p)\delta_{n=2}-\delta_{{}_m\mathrm{I}_1,\mathrm{I\hspace{-.1em}I}}\right)\delta_{\eta'_p\neq 0}+\left(\frac{r}{n}-1\right)\eta''_p. \\
& \iota(F_p)= j(F_p)-\eta_p+\delta_{\mathrm{cyc}},
\end{align*}
where $\delta_{\mathrm{cyc}}$ is defined to be $1$ if the following $(1)$, $(2)$, $(3)$ and $(4)$ hold 
and $\delta_{\mathrm{cyc}}=0$ otherwise.

\smallskip

\noindent
$(1)$ $\Gamma_p$ is a singular fiber of type $({}_m\mathrm{I}_k)_{k\ge 1}$, $(\mathrm{I\hspace{-.1em}I})$, $(\mathrm{I\hspace{-.1em}I\hspace{-.1em}I})$ or $(\mathrm{I\hspace{-.1em}V})$.

\smallskip

\noindent
$(2)$ Any irreducible component of $\Gamma_p$ is contained in $R$.

\smallskip

\noindent
$(3)$ $\iota^{(3)}(F_p)=\kappa^{(3)}(F_p)=0$.

\smallskip

\noindent
$(4)$ The multiplicity of the singular point of $R$ which is singular for $(\Gamma_p)_{\mathrm{red}}$ belongs to $n\mathbb{Z}+1$ if $\Gamma_p$ is a singular fiber of type $({}_m\mathrm{I}_1)$ or $(\mathrm{I\hspace{-.1em}I})$.
\end{lem}

\begin{proof}
By the definition of $\gamma_p$, the first inequality is clear.
We consider the following graph $\mathbf{G}^t$:
The vertex set $V(\mathbf{G}^t)$ is defined by the symbol set $\{v^{t,k}\}_{k=1}^{j^t(F_p)}$.
The edge set $E(\mathbf{G}^t)$ is defined by the symbol set $\{e_x\}_{x}\cup \{e_y\}_{y}\cup \{e'_y\}_{y}$, where $x$, $y$ respectively move among $(t,2)$, $(t,3)$-vertical $n\mathbb{Z}$ type singularities. If the proper transform of $C^{t,k}$ meets that of $C^{t,k'}$ at a $(t,2)$-vertical $n\mathbb{Z}$ type singularity $x$, the edge $e_x$ connects $v^{t,k}$ and $v^{t,k'}$.
If the proper transforms of $C^{t,k}$, $C^{t,k'}$ and $C^{t,k''}$ ($k<k'<k''$) intersects in a $(t,3)$-vertical $n\mathbb{Z}$ type singularity $y$, the edge $e_y$ connects $v^{t,k}$ and $v^{t,k'}$, and $e'_y$ connects $v^{t,k'}$ and $v^{t,k''}$.
By the definition of the decomposition $\widetilde{R}_v(p)=D^1(p)+\cdots+D^{\eta_p}(p)$, the graph $\mathbf{G}^t$ is connected for any $t=1,\dots, \eta_p$.
Clearly, $\iota(F_p)$ is the cardinality of $E(\mathbf{G}^t)$.
Thus, the number of cycles in $\mathbf{G}^t$ is $\iota^t(F_p)-j^t(F_p)+1$.
One sees that $\mathbf{G}^t$ has at most one cycle,
 and it has one cycle only if $\{C^{t,k}\}_k$ contains all irreducible components of $\Gamma_p$.
Hence at most one $\mathbf{G}^t$ has one cycle.
We can see that $\mathbf{G}^t$ has one cycle for some $t$ if and only if $\delta_{\mathrm{cyc}}=1$.
Thus, we get $\iota(F_p)=j(F_p)-\eta_p+\delta_{\mathrm{cyc}}$.
\end{proof}

For any singular point $x$ of $R$, the multiplicity $\mathrm{mult}_x(R)$ at $x$ does not exceed $r/m+1$ since $R(\Gamma_p)_{\mathrm{red}}=r/m$.
Thus we have $\alpha_k=0$ for $k\ge r/nm+1$.
Moreover, the following lemma holds.

\begin{lem}\label{chainlem}
If $n\ge 3$, then we have $\alpha_{\frac{r}{nm}}^{n\mathbb{Z}+1}(F_p)=0$.
If $n=2$, then we have $\kappa_{\frac{r}{2m}}(F_p)=0$.
\end{lem}

\begin{proof}
If $\alpha_{\frac{r}{nm}}^{n\mathbb{Z}+1}(F_p)\neq 0$, then there exists an irreducible component $C$ of $\Gamma_p$ contained in $R$ and a singular point $x$ of $R$ on $C$ with multiplicity $r/m+1$ such that
 any local horizontal branch of $R$ around $x$ is not tangential to $C$
since $RD_p=r/m$.
Then, the exceptional curve $E$ for the blow-up at $x$ and the proper transform of $C$ form a singular point of multiplicity $2$. 
Hence we have $n=2$ from Lemma~\ref{multlem}.
It is clear that all singular points with multiplicity $r/m+1$ are infinitely near to $x$ and the exceptional curves for blow-ups of these singularities form a chain.
In particular, any singular point with multiplicity $r/m+1$ is a $1$-vertical type singularity.
\end{proof}

To prove Theorem~\ref{h1upperthm}, we need some inequalities among several indices.

\begin{lem}\label{clearlem}
$(1)$ The following holds.

\begin{align*}
\sum_{k\ge 1}k \left(\alpha^{n\mathbb{Z}+1}_k(F_p)-\kappa_k(F_p)\right)\le& \left(\frac{r}{n}-1\right)\left(j''(F_p)-\kappa(F_p)\right) \\
& +\left(\frac{r}{n}-2\right)\kappa^{(3)}(F_p)+\alpha_{\frac{r}{nm}}^{n\mathbb{Z}+1}(F_p).
\end{align*}

\smallskip

\noindent
$(2)$ If $n=2$, then the following holds more strongly.

\begin{align*}
\sum_{k\ge 1}k \left(\alpha^{2\mathbb{Z}+1}_k(F_p)-\kappa_k(F_p)\right)\le& \sum_{k\ge 1}k\alpha_{(2k+1\to 2k+1)}(F_p)+
\left(\frac{r}{2}-1\right)\left(\sum_{a\ge 2}j''_{0,a}(F_p)-\kappa(F_p)\right) \\
& +\left(\frac{r}{2}-2\right)\kappa^{(3)}(F_p)+\alpha_{\frac{r}{2m}}^{2\mathbb{Z}+1}(F_p).
\end{align*}

\end{lem}

\begin{proof}
From Lemma~\ref{chainlem}, we have

\begin{align*}
&\sum_{k\ge 1}k \left(\alpha^{n\mathbb{Z}+1}_k(F_p)-\kappa_k(F_p)\right) = \sum_{k=1}^{\frac{r}{nm}-1}k \left(\alpha^{n\mathbb{Z}+1}_k(F_p)-\kappa_k(F_p)\right)+\frac{r}{nm}\alpha_{\frac{r}{nm}}^{n\mathbb{Z}+1}(F_p) \\
&= \sum_{k=1}^{\frac{r}{nm}-1}k \left(\alpha^{n\mathbb{Z}+1}_k(F_p)-\kappa^{(2)}_k(F_p)-\kappa^{(3)}_k(F_p)\right)
-\sum_{k=1}^{\frac{r}{nm}-1}k\kappa^{(3)}_k(F_p)+\frac{r}{nm}\alpha_{\frac{r}{nm}}^{n\mathbb{Z}+1}(F_p)
\end{align*}
Since $\alpha^{n\mathbb{Z}+1}_k(F_p)-\kappa^{(2)}_k(F_p)-\kappa^{(3)}_k(F_p)\ge 0$ 
and $\sum_{k=1}^{\frac{r}{nm}-1}k\kappa^{(3)}_k(F_p)=k_0\kappa^{(3)}(F_p)$ for some $1\le k_0\le r/n-1$ from Lemma~\ref{3verlem}, we have

\begin{align*}
&\sum_{k=1}^{\frac{r}{nm}-1}k \left(\alpha^{n\mathbb{Z}+1}_k(F_p)-\kappa^{(2)}_k(F_p)-\kappa^{(3)}_k(F_p)\right)
-\sum_{k=1}^{\frac{r}{nm}-1}k\kappa^{(3)}_k(F_p)+\frac{r}{nm}\alpha_{\frac{r}{nm}}^{n\mathbb{Z}+1}(F_p) \\
&\le \left(\frac{r}{n}-1\right)\left(\sum_{k=1}^{\frac{r}{nm}-1}\alpha^{n\mathbb{Z}+1}_k(F_p)-\kappa^{(2)}(F_p)-\kappa^{(3)}(F_p)\right)-k_0\kappa^{(3)}(F_p)+\frac{r}{n}\alpha_{\frac{r}{nm}}^{n\mathbb{Z}+1}(F_p)\\
&=\left(\frac{r}{n}-1\right)\left(\sum_{k=1}^{\frac{r}{nm}-1}\alpha^{n\mathbb{Z}+1}_k(F_p)-\kappa(F_p)\right)+\left(\frac{r}{n}-1-k_0\right)\kappa^{(3)}(F_p)+\frac{r}{n}\alpha_{\frac{r}{nm}}^{n\mathbb{Z}+1}(F_p).
\end{align*}
Combining the above inequality with $j''(F_p)=\sum_{k=1}^{\frac{r}{nm}}\alpha^{n\mathbb{Z}+1}_k(F_p)$ and $r/n-1-k_0\le r/n-2$, the assertion $(1)$ follows.

Assume $n=2$. 
Note that any $(2k+1\to 2k+1)$ singularity is not involved in $\kappa(F_p)$.
Then we have

\begin{align*}
\sum_{k\ge 1}k \left(\alpha^{2\mathbb{Z}+1}_k(F_p)-\kappa_k(F_p)\right)=&\sum_{k=1}^{\frac{r}{2m}-1}k \left(\alpha^{2\mathbb{Z}+1}_k(F_p)-\alpha_{(2k+1\to 2k+1)}(F_p)-\kappa_k(F_p)\right)\\
&+\sum_{k\ge 1}k\alpha_{(2k+1\to 2k+1)}(F_p)+\frac{r}{2m}\alpha_{\frac{r}{2m}}^{2\mathbb{Z}+1}(F_p).
\end{align*}
Similarly as in (1), the assertion (2) follows.
\end{proof}

\begin{lem} \label{n2h1kappalem}
If $n=2$, then we have
$$
\kappa(F_p)\le \frac{2}{3}\sum_{a\ge 2}(a-1)j_{\bullet,a}(F_p)-\frac{2}{3}\alpha_{\frac{r}{2m}}^{2\mathbb{Z}+1}(F_p).
$$
\end{lem}

\begin{proof}
It is sufficient to show that
\begin{equation} \label{n2h1kappaeq}
\kappa^{t}(F_p)\le \frac{2}{3}\sum_{a\ge 2}(a-1)j^{t}_{0,a}(F_p)-\frac{2}{3}\alpha_{\frac{r}{2m}}^{t,2\mathbb{Z}+1}(F_p)
\end{equation}
for any $t$.
If $\kappa^{t}(F_p)=0$, then it is clear.
Thus, we may assume $\kappa^{t}(F_p)>0$.
Then clearly we have 
\begin{equation} \label{n2jteq}
j^{t}(F_p)\ge \kappa^{t,(2)}(F_p)+\kappa^{t,(3)}(F_p)+\alpha_{\frac{r}{2m}}^{t,2\mathbb{Z}+1}(F_p)+2.
\end{equation}
Since any blow-up at a $(t,u)$-vertical type singularity contributes $-u$ to the number 
$$
\sum_{k\ge 1}(L^{t,k})^2=-\sum_{a\ge 1}2aj^{t}_{\bullet,a}(F_p)
$$
and $\Gamma_p$ contains no $u$-vertical type singularity for $u\ge 2$ if $\Gamma_p$ is of type $({}_m\mathrm{I}_0)$, we get
\begin{align*}
\sum_{a\ge 1}2aj^{t}_{\bullet,a}(F_p)\ge& j_{1,\bullet}^{t}(F_p)+(2+\delta_{{}_m\mathrm{I}_1,\mathrm{I\hspace{-.1em}I}})j'^{t}_{0,\bullet}(F_p)+j''^{t}_{0,\bullet}(F_p) \\
&+\alpha_{\frac{r}{2m}}^{t,2\mathbb{Z}+1}(F_p)+\sum_{u=2,3}u\left(\iota^{t,(u)}(F_p)+\kappa^{t,(u)}(F_p)\right).
\end{align*}
Combining this inequality with $\iota^{t}(F_p)\ge j^{t}(F_p)-1$ and \eqref{n2jteq}, we have
\begin{align*}
\sum_{a\ge 1}2aj^{t}_{\bullet,a}(F_p)\ge& 3j^{t}(F_p)+(1+\delta_{{}_m\mathrm{I}_1,\mathrm{I\hspace{-.1em}I}})j'^{t}_{0,\bullet}(F_p)+\alpha_{\frac{r}{2m}}^{t,2\mathbb{Z}+1}(F_p)\\
&+2\kappa^{t}(F_p)-2-\left(\iota^{t,(3)}(F_p)+\kappa^{t,(3)}(F_p)\right) \\
\ge& 2j^{t}(F_p)+(1+\delta_{{}_m\mathrm{I}_1,\mathrm{I\hspace{-.1em}I}})j'^{t}_{0,\bullet}(F_p)+2\alpha_{\frac{r}{2m}}^{t,2\mathbb{Z}+1}(F_p)\\
&+3\kappa^{t}(F_p)-\iota^{t,(3)}(F_p)-2\kappa^{t,(3)}(F_p).
\end{align*}
On the other hand, it is easily seen that
$$
(1+\delta_{{}_m\mathrm{I}_1,\mathrm{I\hspace{-.1em}I}})j'^{t}_{0,\bullet}(F_p)-\iota^{t,(3)}(F_p)-2\kappa^{t,(3)}(F_p)\ge 0.
$$
Hence we get \eqref{n2h1kappaeq}, as desired.
\end{proof}

\begin{lem} \label{etalem}
$(1)$ If $n=3$, then the following hold.

\smallskip

\noindent
$(1,\mathrm{i})$ If $j'^{t}_{0,1}(F_p)\le 2$ for any $t$, then
$$
\frac{1}{2}j'_{0,1}(F_p)\le \eta'_p-\delta_{\mathrm{cyc}}.
$$

\smallskip

\noindent
$(1,\mathrm{ii})$ If $j'^{t}_{0,1}(F_p)=3$ for some $t$, then
$\Gamma_p$ is a singular fiber of type $(\mathrm{I\hspace{-.1em}V})$, $(\mathrm{I}^{*}_k)$, $(\mathrm{I\hspace{-.1em}I}^{*})$, $(\mathrm{I\hspace{-.1em}I\hspace{-.1em}I}^{*})$ or  $(\mathrm{I\hspace{-.1em}V}^{*})$ and 
$$
\frac{1}{3}j'_{0,1}(F_p)\le \eta'_p,\quad \delta_{\mathrm{cyc}}=0.
$$

\smallskip

\noindent
$(1,\mathrm{iii})$ If $j'^{t}_{0,1}(F_p)=4$ for some $t$, then
$\Gamma_p$ is a singular fiber of type $(\mathrm{I}^{*}_k)$ and 
any component of $\Gamma_p$ is contained in $R$.
Moreover, we have
$\eta'_p=1$, $j'_{0,1}(F_p)=4$ and $\delta_{\mathrm{cyc}}=0$.

\smallskip

\noindent
$(2)$ If $n=2$, then the following hold.

\smallskip

\noindent
$(2,\mathrm{i})$ If $j'^{t}_{0,2,\mathrm{odd}}(F_p)\le 2$ for any $t$, then
$$
j'_{0,1}(F_p)+\frac{1}{2}j'_{0,2,\mathrm{odd}}(F_p)\le \eta'_p-\delta_{\mathrm{cyc}},
$$
where $j'_{0,2,\mathrm{odd}}(F_p)$ denotes the number of irreducible components $C$ of $\Gamma_p$ involved in $j'_{0,2}(F_p)$ which has a singular point of $R$ of odd multiplicity.

\smallskip

\noindent
$(2,\mathrm{ii})$ If $j'^{t}_{0,2,\mathrm{odd}}(F_p)=3$ for some $t$, then
$\Gamma_p$ is a singular fiber of type $(\mathrm{I\hspace{-.1em}V})$, $(\mathrm{I}^{*}_k)$, $(\mathrm{I\hspace{-.1em}I}^{*})$, $(\mathrm{I\hspace{-.1em}I\hspace{-.1em}I}^{*})$ or  $(\mathrm{I\hspace{-.1em}V}^{*})$ and 
$$
j'_{0,1}(F_p)+\frac{1}{3}j'_{0,2,\mathrm{odd}}(F_p)\le \eta'_p,\quad \delta_{\mathrm{cyc}}=0.
$$

\smallskip

\noindent
$(2,\mathrm{iii})$ If $j'^{t}_{0,2,\mathrm{odd}}(F_p)=4$ for some $t$, then
$\Gamma_p$ is a singular fiber of type $(\mathrm{I}^{*}_k)$ and 
any component of $\Gamma_p$ is contained in $R$.
Moreover, we have
$\eta'_p=1$, $j'_{0,1}(F_p)=0$, $j'_{0,2,\mathrm{odd}}(F_p)=4$ and $\delta_{\mathrm{cyc}}=0$.
\end{lem}

\begin{proof}
If $n=3$, then any curve $C$ in $\Gamma_p$ contributing to $j'_{0,1}(F_p)$ intersects at most one component of $\Gamma_p$ contained in $R$, since $C$ is blown up just once.
Thus, considering the classification of singular fibers of elliptic surfaces, we can show easily the assertion (1).

Suppose that $n=2$. Any curve in $\Gamma_p$ contributing to $j'_{0,1}(F_p)$ is not blown up and any curve in $\Gamma_p$ contributing to $j'_{0,2,\mathrm{odd}}(F_p)$ intersects at most one component of $\Gamma_p$ contained in $R$.
Hence we can show the assertion (2) similarly.
\end{proof}

\begin{lem} \label{j''01lem}
$(1)$ If $n=3$, then we have
$$
j''_{0,1}(F_p)\le 2\eta''_p+\sum_{a\ge 1}2aj'_{1,a}(F_p)+\sum_{a\ge 2}(2a-2)j'_{0,a}(F_p)+\sum_{a\ge 2}(2a-1)j''_{0,a}(F_p).
$$

\smallskip

\noindent
$(2)$ If $n=2$, then we have

\begin{align*}
\sum_{k\ge 1}\alpha^{\mathrm{co},1}_{(2k+1\to 2k+1)}(F_p)\le&\;
j'_{0,2,\mathrm{odd}}(F_p)+\sum_{a\ge 3}(a-1)j'_{0,a}(F_p)+\sum_{a\ge 2}(a-1)j'_{1,a}(F_p) \\
&+\sum_{a\ge 3}(a-2)j''_{0,a}(F_p)+\widehat{\eta}_p.
\end{align*}

\end{lem}

\begin{proof}
Suppose that $n=3$. Let $C_1,\dots, C_{j''^{t}_{0,1}(F_p)}$ be all $(-1)$-curves in $\{C''^{t,k}\}_k$ contributing to $j''^{t}_{0,1}(F_p)$ and $x_i$ the point to which $C_i$ contracts for $i=1,\dots,j''^{t}_{0,1}(F_p)$.
If $C_i\neq C^{t,1}$, then $x_i$ is contained in some $C^{t,k}$. 
If $C^{t,k}$ contributes to $j''^{t}_{0,1}(F_p)$ and $k=1$, then $j''^{t}(F_p)=j''^{t}_{0,1}(F_p)=2$ from Example~\ref{diagexa}~(2).
If $C^{t,k}$ contributes to $j''^{t}_{0,1}(F_p)$ and $k\neq 1$, then the point $x^{t,k}$ to which $C^{t,k}$ contracts is contained in another $C^{t,k'}$ which does not contribute to  $j''^{t}_{0,1}(F_p)$ from the argument of Example~\ref{diagexa}~(2). Moreover, $x_i$ is also contained in $C^{t,k'}$ since the singularity diagram of $C^{t,k}$ is type (iii) in Example~\ref{diagexa}~(2) and Lemma~\ref{tclem}.
For a curve $C^{t,k}$ which does not contribute to $j''^{t}_{0,1}(F_p)$, we consider how many points among $x_1,\dots, x_{j''^{t}_{0,1}(F_p)}$ it contains.

\smallskip

\noindent
(i) Assume that $C^{t,k}$ contributes to $j''^{t}_{0,a}(F_p)$ for some $a\ge 2$.
Then $C^{t,k}$ is blown up $3a-1$ times.
Let $(x_{i,j},m_{i,j})$, $i=1,\dots,l$, $j=1,\dots,i_{\mathrm{bm}}$ be entries of the singularity diagram of $C^{t,k}$.
We consider a subset of entries of the $i$-th column of its diagram $\{(x_{i,j},m_{i,j})\}_{j=j_0+1,\dots,j_0+N}$ satisfying that 

\smallskip

\noindent
$(*)$ $m_{i,j_0}\in 3\mathbb{Z}$ if $j_0>0$, $m_{i,j}\in 3\mathbb{Z}+1$ for $j_0<j<j_0+N$ and $m_{i,j_0+N}\in 3\mathbb{Z}.$

\smallskip

\noindent
Note that the set of all entries of the singularity diagram is the union of these subsets.
Then we can see that the exceptional curve $C^{t,k'}$ obtained by the blow-up at $x^{i,j}$, $j_0+1<j<j_0+N$ does not contribute to $j''_{0,1}(F_p)$ from Lemma~\ref{connlem}.
Hence it contains at most $2a-1$ points among $x_1,\dots, x_{j''^{t}_{0,1}(F_p)}$.

\smallskip

\noindent
(ii) Assume that $C^{t,k}$ contributes to $j'^{t}_{0,a}(F_p)$.
Then $C^{t,k}$ is blown up $3a-2$ times when it is a $(-2)$-curve or $3a-3$ times when it is a singular rational curve.
Hence it contains at most $2a-2$ points among $x_1,\dots, x_{j''^{t}_{0,1}(F_p)}$ by the same argument as in (i).

\smallskip

\noindent
(iii) Assume that $C^{t,k}$ contributes to $j'^{t}_{1,a}(F_p)$.
Then $C^{t,k}$ is blown up $3a$ times.
Hence it contains at most $2a$ points among $x_1,\dots, x_{j''^{t}_{0,1}(F_p)}$ by the same argument as in (i).

We estimate $j''^{t}_{0,1}(F_p)$ from (i), (ii), (iii) as follows.

\smallskip

\noindent
(a) If $D'^{t}(p)=0$ and $j''^{t}_{0,a}(F_p)=0$ for any $a\ge 2$,
then we have shown that $j''^{t}_{0,1}(F_p)\le 2$ in Example~\ref{diagexa}~(2).

\smallskip

\noindent
(b) If $D'^{t}(p)=0$ and $j''^{t}_{0,a}(F_p)>0$ for some $a\ge 2$,
then $x_i$ is the point to which $C^{t,1}$ contracts or contained in some $C^{t,k}$ which contributes to $j''^{t}_{0,a}(F_p)$ for some $a\ge 2$. Hence we have
$$
j''^{t}_{0,1}(F_p)\le 1+\sum_{a\ge 2}(2a-1)j''^{t}_{0,a}(F_p).
$$

\smallskip

\noindent
(c) If $D'^{t}(p)\neq 0$,
then $x_i$ is contained in some $C^{t,k}$ which does not contribute to $j''^{t}_{0,1}(F_p)$. Hence we have
$$
j''^{t}_{0,1}(F_p)\le \sum_{a\ge 2}(2a-2)j'^{t}_{0,a}(F_p)+\sum_{a\ge 1}2aj'^{t}_{1,a}(F_p)+\sum_{a\ge 2}(2a-1)j''^{t}_{0,a}(F_p).
$$
From (a), (b) and (c), we have
$$
j''_{0,1}(F_p)\le \overline{\eta}_p+\eta''_p+\sum_{a\ge 2}(2a-2)j'_{0,a}(F_p)+\sum_{a\ge 1}2aj'_{1,a}(F_p)+\sum_{a\ge 2}(2a-1)j''_{0,a}(F_p)
$$
by summing up for $t=1,\dots,\eta_p$.
Combining this with $\overline{\eta}_p\le \eta''_p$, the claim (1) follows.

Suppose $n=2$. Let $x^{t,k}$ be the point to which $C''^{t,k}$ is contracted and $m^{t,k}$ the multiplicity of $R$ at $x^{t,k}$.
If $D'^{t}(p)=0$, then $x^{t,k}$, $k\ge 2$ is contained in $C^{t,k'}$ for some $k'<k$.
Otherwise, $x^{t,1}$ is also contained in $C^{t,k'}$ for some $k'$.
If $C^{t,k}$ is smooth, a singularity with odd multiplicity which is not contained in $C^{t,k'}$ for any $k'>k$ corresponds to an entry $(x_{i,j},m_{i,j})$ of the singularity diagram $\mathcal{D}^{t,k}$ of $C^{t,k}$ satisfying that $m_{i,j-1}$ is even if $j>1$, and $m_{i,j}$ is odd
and then corresponds to a subset of entries of the diagram satisfying $(*)$.
For a curve $C^{t,k}$, we consider how many such subsets of entries of its singularity diagram there are.

\smallskip 

\noindent
(iv) If $C^{t,k}$ contributes to $j''^{t}_{0,a}(F_p)$, then $C^{t,k}$ is blown up $2a-1$ times.
Then the singularity diagram of $C^{t,k}$ has at most $a-1$ subsets satisfying $(*)$.

\smallskip

\noindent
(v) If $C^{t,k}$ contributes to $j'^{t}_{0,a}(F_p)$ and it is a $(-2)$-curve, then $C^{t,k}$ is blown up $2a-2$ times.
Then the singularity diagram of $C^{t,k}$ has at most $a-1$ subsets satisfying $(*)$.

\smallskip

\noindent
(vi) If $C^{t,k}$ contributes to $j'^{t}_{0,a}(F_p)$ and it is a singular rational curve, then $C^{t,k}$ is blown up $2a-3$ times.
Considering the singularity diagram of the proper transform of $C^{t,k}$ by the blow-up at its singular point, 
$C^{t,k}$ has at most $a-1$ singularities with odd multiplicity which is not contained in $C^{t,k'}$ for any $k'>k$.

\smallskip

\noindent
(vii) If $C^{t,k}$ contributes to $j'^{t}_{1,a}(F_p)$, then $C^{t,k}$ is blown up $2a$ times.
Then the singularity diagram of $C^{t,k}$ has at most $a$ subsets satisfying $(*)$.

We estimate $j''^{t}(F_p)$ using (iv), (v), (vi) and (vii) as follows.

\smallskip

\noindent
(d) If $D'^{t}(p)=0$,
then the number of singularities with odd multiplicity appearing in $\{C^{t,k}\}_k$ is $j''^{t}(F_p)-1$.
Hence we have 
$$
j''^{t}(F_p)-1\le \sum_{a\ge 2}(a-1)j''^{t}_{0,a}(F_p).
$$

\smallskip

\noindent
(e) If $D'^{t}(p)\neq 0$,
then the number of singularities with odd multiplicity appearing in $\{C^{t,k}\}_k$ is $j''^{t}(F_p)$.
Hence we have 
$$
j''^{t}(F_p)\le j'^{t}_{0,1,\mathrm{odd}}(F_p)+\sum_{a\ge 3}(a-1)j'^{t}_{0,a}(F_p)+\sum_{a\ge 1}aj'^{t}_{1,a}(F_p)+\sum_{a\ge 2}(a-1)j''^{t}_{0,a}(F_p).
$$
From (d) and (e), we have
$$
j''(F_p)\le \eta''_p+j'_{0,1,\mathrm{odd}}(F_p)+\sum_{a\ge 2}(a-1)j'_{0,a}(F_p)+\sum_{a\ge 1}aj'_{1,a}(F_p)+\sum_{a\ge 2}(a-1)j''_{0,a}(F_p)
$$
by summing up for $t=1,\dots,\eta_p$.
Combining this with \eqref{n2j''01eq} and \eqref{n2etaeq}, the claim (2) follows.

\end{proof}

\begin{lem} \label{chiphilem}
$(1)$ If $n=3$ and $j'_{0,1}(F_p)\neq 0$, then we have
$$
\chi_{\varphi}(F_p)\ge \frac{1}{12}(j'_{0,1}(F_p)+1).
$$

\smallskip

\noindent
$(2)$ If $n=2$, then the following hold.

\smallskip

\noindent
$(2,\mathrm{i})$ If $\Gamma_p$ is a singular fiber not of type $({}_m\mathrm{I}_k)$, $(\mathrm{I}^{*}_k)$, $(\mathrm{I\hspace{-.1em}I}^{*})$, $(\mathrm{I\hspace{-.1em}I\hspace{-.1em}I}^{*})$, $(\mathrm{I\hspace{-.1em}V}^{*})$, then we have
$$
\chi_{\varphi}(F_p)\ge \frac{1}{12}(2j'_{0,1}(F_p)+j'_{0,2}(F_p)+j'_{0,3}(F_p)+1).
$$
Moreover, all the cases where $\Gamma_p$ is a singular fiber of type $({}_m\mathrm{I}_k)$,  $(\mathrm{I\hspace{-.1em}I}^{*})$, $(\mathrm{I\hspace{-.1em}I\hspace{-.1em}I}^{*})$, $(\mathrm{I\hspace{-.1em}V}^{*})$ and 
$$
\chi_{\varphi}(F_p)< \frac{1}{12}(2j'_{0,1}(F_p)+j'_{0,2}(F_p)+j'_{0,3}(F_p)+1)
$$
are as follows.

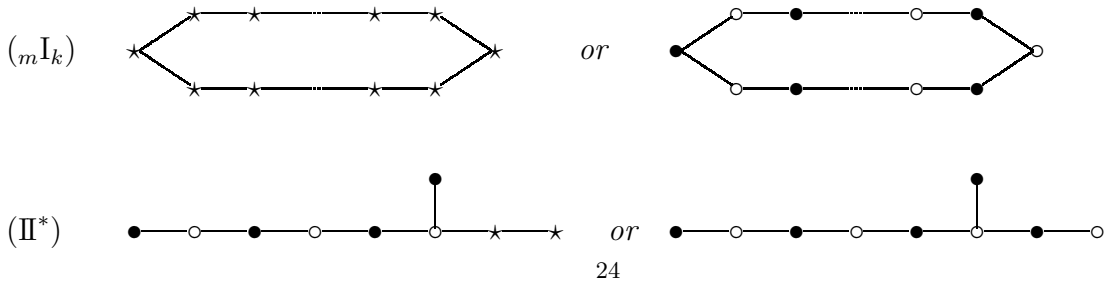
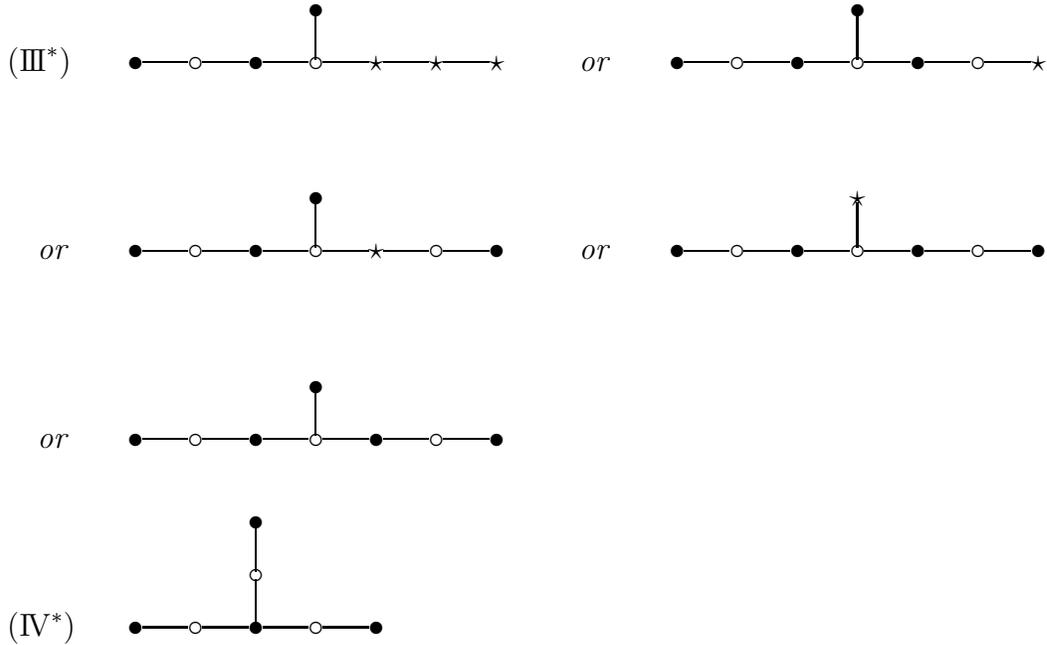
\setlength\unitlength{0.2cm}
\begin{figure}[H]
\begin{center}
\begin{tabular}{c}

 \begin{minipage}{0.6\hsize}
 \begin{center}
\begin{picture}(0,6)
 \multiput(-27,3.05)(4,0){4}{\line(1,0){3}}
 \multiput(-27,-1.95)(4,0){4}{\line(1,0){3}}
 \multiput(-22,3.05)(0.3,0){20}{\line(1,0){0.15}}
 \multiput(-22,-1.95)(0.3,0){20}{\line(1,0){0.15}}
 \qbezier(-7.8,0.6)(-7.8,0.6)(-11.1,2.8)
 \qbezier(-7.8,0.5)(-7.8,0.5)(-11.1,-1.7)
 \qbezier(-31.1,0.6)(-31.1,0.6)(-27.8,2.8)
 \qbezier(-31.1,0.5)(-31.1,0.5)(-27.8,-1.7)
 \put(-40,0){$({}_m\mathrm{I}_k)$}

 \put(-32,0){$\star$}
 \put(-28,2.5){$\star$}
 \put(-24,2.5){$\star$}
 \put(-16,2.5){$\star$}
 \put(-12,2.5){$\star$}
 \put(-8,0){$\star$}

 \put(-28,-2.5){$\star$}
 \put(-24,-2.5){$\star$}
 \put(-16,-2.5){$\star$}
 \put(-12,-2.5){$\star$}

 \put(-2,0){or}

\multiput(9,3.05)(4,0){4}{\line(1,0){3}}
 \multiput(9,-1.95)(4,0){4}{\line(1,0){3}}
 \multiput(14,3.05)(0.3,0){20}{\line(1,0){0.15}}
 \multiput(14,-1.95)(0.3,0){20}{\line(1,0){0.15}}
 \qbezier(28.2,0.6)(28.2,0.6)(24.9,2.8)
 \qbezier(28.2,0.5)(28.2,0.5)(24.9,-1.7)
 \qbezier(4.9,0.6)(4.9,0.6)(8.2,2.8)
 \qbezier(4.9,0.5)(4.9,0.5)(8.2,-1.7)

 \put(4,0){$\bullet$}
 \put(8,2.5){$\circ$}
 \put(12,2.5){$\bullet$}
 \put(20,2.5){$\circ$}
 \put(24,2.5){$\bullet$}
 \put(28,0){$\circ$}

 \put(8,-2.5){$\circ$}
 \put(12,-2.5){$\bullet$}
 \put(20,-2.5){$\circ$}
 \put(24,-2.5){$\bullet$}

 \end{picture}
 \end{center}
 \end{minipage}

\end{tabular}
\end{center}
\end{figure}

\setlength\unitlength{0.2cm}
\begin{figure}[H]
\begin{center}
\begin{tabular}{c}

 \begin{minipage}{0.6\hsize}
 \begin{center}
\begin{picture}(0,6)
 \multiput(-31,0.55)(4,0){7}{\line(1,0){3}}
 \qbezier(-11.5,0.8)(-11.5,0.8)(-11.5,3.7)
 \put(-40,0){$(\mathrm{I\hspace{-.1em}I}^{*})$}

 \put(-32,0){$\bullet$}
 \put(-28,0){$\circ$}
 \put(-24,0){$\bullet$}
 \put(-20,0){$\circ$}
 \put(-16,0){$\bullet$}
 \put(-12,0){$\circ$}
 \put(-8,0){$\star$}
 \put(-4,0){$\star$}
 \put(-12,3.5){$\bullet$}

\put(0,0){or}

 \multiput(5,0.55)(4,0){7}{\line(1,0){3}}
 \qbezier(24.5,0.8)(24.5,0.8)(24.5,3.7)

 \put(4,0){$\bullet$}
 \put(8,0){$\circ$}
 \put(12,0){$\bullet$}
 \put(16,0){$\circ$}
 \put(20,0){$\bullet$}
 \put(24,0){$\circ$}
 \put(28,0){$\bullet$}
 \put(32,0){$\circ$}
 \put(24,3.5){$\bullet$}

 \end{picture}
 \end{center}
 \end{minipage}

\end{tabular}
\end{center}
\end{figure}

\setlength\unitlength{0.2cm}
\begin{figure}[H]
\begin{center}
\begin{tabular}{c}

 \begin{minipage}{0.6\hsize}
 \begin{center}
\begin{picture}(0,6)
 \multiput(-31,0.55)(4,0){6}{\line(1,0){3}}
 \qbezier(-19.5,0.8)(-19.5,0.8)(-19.5,3.7)
 \put(-40,0){$(\mathrm{I\hspace{-.1em}I\hspace{-.1em}I}^{*})$}

 \put(-32,0){$\bullet$}
 \put(-28,0){$\circ$}
 \put(-24,0){$\bullet$}
 \put(-20,0){$\circ$}
 \put(-16,0){$\star$}
 \put(-12,0){$\star$}
 \put(-8,0){$\star$}
 \put(-20,3.5){$\bullet$}

 \put(-2,0){or}

 \multiput(5,0.55)(4,0){6}{\line(1,0){3}}
 \qbezier(16.5,0.8)(16.5,0.8)(16.5,3.7)

 \put(4,0){$\bullet$}
 \put(8,0){$\circ$}
 \put(12,0){$\bullet$}
 \put(16,0){$\circ$}
 \put(20,0){$\bullet$}
 \put(24,0){$\circ$}
 \put(28,0){$\star$}
 \put(16,3.5){$\bullet$}

\end{picture}
 \end{center}
 \end{minipage}

\end{tabular}
\end{center}
\end{figure}

\setlength\unitlength{0.2cm}
\begin{figure}[H]
\begin{center}
\begin{tabular}{c}

 \begin{minipage}{0.6\hsize}
 \begin{center}
\begin{picture}(0,6)
 \multiput(-31,0.55)(4,0){6}{\line(1,0){3}}
 \qbezier(-19.5,0.8)(-19.5,0.8)(-19.5,3.7)
  \put(-38,0){or}

 \put(-32,0){$\bullet$}
 \put(-28,0){$\circ$}
 \put(-24,0){$\bullet$}
 \put(-20,0){$\circ$}
 \put(-16,0){$\star$}
 \put(-12,0){$\circ$}
 \put(-8,0){$\bullet$}
 \put(-20,3.5){$\bullet$}

 \put(-2,0){or}

 \multiput(5,0.55)(4,0){6}{\line(1,0){3}}
 \qbezier(16.5,0.8)(16.5,0.8)(16.5,3.7)

 \put(4,0){$\bullet$}
 \put(8,0){$\circ$}
 \put(12,0){$\bullet$}
 \put(16,0){$\circ$}
 \put(20,0){$\bullet$}
 \put(24,0){$\circ$}
 \put(28,0){$\bullet$}
 \put(16,3.5){$\star$}

\end{picture}
 \end{center}
 \end{minipage}

\end{tabular}
\end{center}
\end{figure}

\setlength\unitlength{0.2cm}
\begin{figure}[H]
\begin{center}
\begin{tabular}{c}

 \begin{minipage}{0.6\hsize}
 \begin{center}
\begin{picture}(0,6)
 \multiput(-31,0.55)(4,0){6}{\line(1,0){3}}
 \qbezier(-19.5,0.8)(-19.5,0.8)(-19.5,3.7)
  \put(-38,0){or}

 \put(-32,0){$\bullet$}
 \put(-28,0){$\circ$}
 \put(-24,0){$\bullet$}
 \put(-20,0){$\circ$}
 \put(-16,0){$\bullet$}
 \put(-12,0){$\circ$}
 \put(-8,0){$\bullet$}
 \put(-20,3.5){$\bullet$}

\end{picture}
 \end{center}
 \end{minipage}

\end{tabular}
\end{center}
\end{figure}

\setlength\unitlength{0.2cm}
\begin{figure}[H]
\begin{center}
\begin{tabular}{c}

 \begin{minipage}{0.6\hsize}
 \begin{center}
\begin{picture}(0,6)
 \multiput(-31,0.55)(4,0){4}{\line(1,0){3}}
 \qbezier(-23.5,0.8)(-23.5,0.8)(-23.5,3.7)
 \qbezier(-23.5,4.3)(-23.5,4.3)(-23.5,7.2)
 \put(-40,0){$(\mathrm{I\hspace{-.1em}V}^{*})$}

 \put(-32,0){$\bullet$}
 \put(-28,0){$\circ$}
 \put(-24,0){$\bullet$}
 \put(-20,0){$\circ$}
 \put(-16,0){$\bullet$}
 \put(-24,3.5){$\circ$}
 \put(-24,7){$\bullet$}

\end{picture}
 \end{center}
 \end{minipage}

\end{tabular}
\end{center}
\end{figure}

\noindent
where, in the dual graphs of $\Gamma_p$, the symbol $\circ$, $\bullet$, $\star$ respectively denotes a $(-2)$-curve not contained in $R$, contributing to $j'_{0,1}(F_p)$, contributing to $j'_{0,2}(F_p)$ or $j'_{0,3}(F_p)$.
In these cases, we have
$$
\chi_{\varphi}(F_p)=\frac{1}{12}(2j'_{0,1}(F_p)+j'_{0,2}(F_p)+j'_{0,3}(F_p)-1)
$$
when $\Gamma_p$ is of type $(\mathrm{I\hspace{-.1em}I\hspace{-.1em}I}^{*})$ and
$j'_{0,2}(F_p)=j'_{0,3}(F_p)=0$ and
$$
\chi_{\varphi}(F_p)=\frac{1}{12}(2j'_{0,1}(F_p)+j'_{0,2}(F_p)+j'_{0,3}(F_p))
$$
otherwise.

\smallskip

\noindent
$(2,\mathrm{ii})$ If $\Gamma_p$ is a singular fiber of type $(\mathrm{I}^{*}_k)$, then
$$
\chi_{\varphi}(F_p)\ge \frac{1}{12}(2j'_{0,1}(F_p)+j'_{0,2}(F_p)+j'_{0,3}(F_p)-2)
$$
with equality holding if and only if $\Gamma_p$ and $R$ satisfies the condition indicated in the following figure.

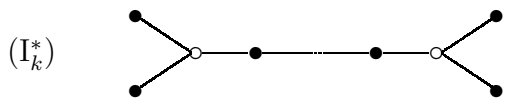
\setlength\unitlength{0.2cm}
\begin{figure}[H]
\begin{center}
\begin{tabular}{c}

 \begin{minipage}{0.6\hsize}
 \begin{center}
\begin{picture}(0,6)
 \multiput(-27,0.55)(4,0){4}{\line(1,0){3}}
 \multiput(-22,0.55)(0.3,0){20}{\line(1,0){0.15}}
 \qbezier(-7.8,2.8)(-7.8,2.8)(-11.1,0.6)
 \qbezier(-7.8,-1.7)(-7.8,-1.7)(-11.1,0.5)
 \qbezier(-31.1,2.8)(-31.1,2.8)(-27.8,0.6)
 \qbezier(-31.1,-1.7)(-31.1,-1.7)(-27.8,0.5)
 \put(-40,0){$(\mathrm{I}^{*}_k)$}

 \put(-32,2.5){$\bullet$}
 \put(-32,-2.5){$\bullet$}
 \put(-28,0){$\circ$}
 \put(-24,0){$\bullet$}
 \put(-16,0){$\bullet$}
 \put(-12,0){$\circ$}
 \put(-8,2.5){$\bullet$}
 \put(-8,-2.5){$\bullet$}

 \end{picture}
 \end{center}
 \end{minipage}

\end{tabular}
\end{center}
\end{figure}

\end{lem}

\ \\

\begin{proof}
$(1)$ Suppose that $n=3$ and $j'_{0,1}(F_p)\neq 0$. 
If $\Gamma_p$ is not of type $({}_m\mathrm{I}_k)$, the claim is clear.
Thus we may assume that $\Gamma_p$ is of type $({}_m\mathrm{I}_k)$.
If $\chi_{\varphi}(F_p)=j'_{0,1}(F_p)/12$, then any component of $\Gamma_p$ contributes to $j'_{0,1}(F_p)$ and contains at least $2$ singular points of $R$, which is a contradiction.

$(2)$ Suppose that $n=2$. 
Any irreducible component $C$ of $\Gamma_p$ contributing to $j'_{0,1}(F_p)$ has no singular points of $R$.
Thus any component of $\Gamma_p$ intersecting with the curve $C$ is not contained in $R$.
From this observation and the classification of singular fibers of elliptic surfaces, 
the claims (2,i) and (2,ii) follow by an easy combinatorial argument.
\end{proof}

Now, we are ready to prove Theorem~\ref{h1upperthm}.

\begin{prfofthmh1}
Let $f\colon S\to B$ be a primitive cyclic covering fibration of type $(g,1,n)$.
From Lemma~\ref{alphaklem}, we have
\begin{align*}
e_f(F_p)&-\mu\chi_f(F_p)= A_n\alpha^{+}_0(F_p)+\sum_{k\ge 1}\left(n-\mu' k\right)\alpha'_k(F_p)+n\sum_{k\ge 1}\alpha''_k(F_p) \\
&-\mu' \sum_{k\ge 1}k\alpha''_k(F_p)-\sum_{a\ge 1}(2A_n+\delta_{a=1})j_{0,a}(F_p)+C_n\chi_{\varphi}(F_p)-\frac{r\mu'}{n}\nu(F_p) \\
=&\; A_n\alpha^{+}_0(F_p)+\sum_{k\ge 1}\left(n-\mu' k\right)\alpha'_k(F_p) +C_n\chi_{\varphi}(F_p) \\
&-\frac{r\mu'}{n}\left(1-\frac{1}{m}\right)-\mu' \gamma_p+n\eta''_p+\sum_{a\ge 1}\left(an^2-a\mu' \right)j'_{1,a}(F_p) \\
&+\sum_{a\ge 1}\left(n(an-2-\delta_{{}_m\mathrm{I}_1,\mathrm{I\hspace{-.1em}I}})-a\mu'-2A_n-\delta_{a=1}\right)j'_{0,a}(F_p) \\
&+\sum_{a\ge 1}\left(n(an-1)-a\mu'-2A_n-\delta_{a=1}\right)j''_{0,a}(F_p) \\
&-\mu' \sum_{k\ge 1}k\alpha^{n\mathbb{Z}+1}_k(F_p)-\sum_{k\ge 1}\left(n-\mu' k\right)\iota_k(F_p)-\sum_{k\ge 1}\left(n-\mu' k\right)\kappa_k(F_p),
\end{align*}
where 
$$
A_n:=n-1-\frac{(n-1)(2n-1)}{12n}\mu, \quad 
C_n:=12n-\left(\frac{r}{12n}(n-1)(n+1)+n\right)\mu
$$
and $\mu':=(n=1)(n+1)\mu/12$.
Combining Lemma~\ref{gammaiotalem} with the above equality, we have

\begin{align}
&e_f(F_p)-\mu\chi_f(F_p) \nonumber \\
\ge &A_n\alpha^{+}_0(F_p)+\sum_{k\ge 1}\left(n-\mu' k\right)\alpha'_k(F_p) +C_n\chi_{\varphi}(F_p)-\frac{r\mu'}{n}\left(1-\frac{1}{m}\right) \nonumber \\
&-\mu' \left(\frac{r}{n}-j'_{0,1}(F_p)\delta_{n=2}-\delta_{{}_m\mathrm{I}_1,\mathrm{I\hspace{-.1em}I}}\right)\delta_{\eta'_p\neq 0}+\left(n-\mu' \right)(\eta'_p-\delta_{\mathrm{cyc}})+\left(2n-\frac{r\mu'}{n}\right)\eta''_p
\nonumber \\
&+\sum_{a\ge 1}\left(n(an-3-\delta_{{}_m\mathrm{I}_1,\mathrm{I\hspace{-.1em}I}})-(a-1)\mu'-2A_n-\delta_{a=1}\right)j'_{0,a}(F_p) \label{presteq} \\
&+\sum_{a\ge 1}\left(n(an-2)-(a-1)\mu'-2A_n-\delta_{a=1}\right)j''_{0,a}(F_p) \nonumber \\
&+\sum_{a\ge 1}\left(n(an-1)-(a-1)\mu' \right)j'_{1,a}(F_p)-\mu' \sum_{k\ge 1}k\alpha^{n\mathbb{Z}+1}_k(F_p)-\sum_{k\ge 1}\left(n-\mu' k\right)\kappa_k(F_p). \nonumber 
\end{align}

Assume that $A_n\ge 0$ and $C_n\ge 0$. We obtain by using Lemmas \ref{alpha0lem} and \ref{clearlem}~(1) that

\begin{align}
&e_f(F_p)-\mu\chi_f(F_p) \nonumber \\
\ge &\sum_{k\ge 1}\left(n-\mu' k\right)\alpha'_k(F_p) +\frac{r(m-1)(n-1)(4-\mu)}{4m}+C_n\chi_{\varphi}(F_p)-\mu' \left(\frac{r}{n}-\delta_{{}_m\mathrm{I}_1,\mathrm{I\hspace{-.1em}I}}\right)\delta_{\eta'_p\neq 0}  \nonumber \\
&+A_n\beta_p+\left(n-\mu'-(n-2)A_n\right)(\eta'_p-\delta_{\mathrm{cyc}})+\left(2n-\frac{r\mu'}{n}-(n-2)A_n\right)\eta''_p
 \nonumber \\
&+\sum_{a\ge 1}\left((n-4)A_n+n(an-3-\delta_{{}_m\mathrm{I}_1,\mathrm{I\hspace{-.1em}I}})
-(a-1)\mu'-(1-\delta_{n=2}\mu')\delta_{a=1}\right)j'_{0,a}(F_p)   \label{steq} \\
&+\sum_{a\ge 1}\left((n-4)A_n+n(an-2)-\mu'\left(a+\frac{r}{n}-2\right)-\delta_{a=1}\right)j''_{0,a}(F_p)  \nonumber \\
&+\sum_{a\ge 1}\left((n-2)A_n+n(an-1)-(a-1)\mu' \right)j'_{1,a}(F_p) \nonumber \\
&+\left(2(n-2)A_n-n+\left(\frac{r}{n}-1\right)\mu' \right)\kappa(F_p)-\mu' \left(\frac{r}{n}-2\right)\kappa^{(3)}(F_p)
-\mu' \alpha_{\frac{r}{nm}}^{n\mathbb{Z}+1}(F_p).  \nonumber 
\end{align}

We put
$$
\mu=\left\{\begin{array}{l}
\displaystyle{\frac{12n^2}{r(n-1)(n+1)}},\;\; \text{if } n\ge 4, \\
\displaystyle{\frac{24}{4r-13}},\;\;\;\;\; \text{if } n=3, \\
\displaystyle{\frac{4}{r-2}},\;\;\;\;\;\;\; \text{if } n=2.
\end{array}
\right.
$$

(i) We assume $n\ge 4$. 
We write $r=kn$. 
The coefficient of $\eta'_p$ is
\begin{align*}
n-\frac{1}{12}(n-1)(n+1)\mu -(n-2)A_n=-(n^2-4n+2)+\frac{n^2-6n+2}{k(n+1)}<0.
\end{align*}
The coefficient of $\eta''_p$ is
\begin{align*}
2n-\frac{(n-1)(n+1)r\mu}{12n}-(n-2)A_n=-(n^2-4n+2)+\frac{(n-2)(2n-1)}{k(n+1)}.
\end{align*}
It is negative if $n\ge 5$ or $k\ge 2$. 
Note that $\eta''_p=0$ if $k<n-1$ since the multiplicity $m'$ of a singular point of type $n\mathbb{Z}+1$ satisfies $(n-1)^2\le m' \le r-n+1$.
Thus, we may not consider the case where $n=4$ and $k=1$.
Using $\eta'_p\le j'(F_p)$ and $\eta''_p\le j''(F_p)$, \eqref{steq} is greater than or equal to

\begin{align}
&\sum_{k\ge 1}\left(n-\mu' k\right)\alpha'_k(F_p) +\frac{r(m-1)(n-1)(4-\mu)}{4m}+C_n\chi_{\varphi}(F_p)+A_n\beta_p \nonumber \\
&-\mu' \left(\frac{r}{n}-\delta_{{}_m\mathrm{I}_1,\mathrm{I\hspace{-.1em}I}}\right)\delta_{\eta'_p\neq 0}+\left((n-2)A_n-n+\mu' \right)\delta_{\mathrm{cyc}}+\sum_{a\ge 1}\left(an^2-a\mu' \right)j'_{1,a}(F_p)
 \nonumber \\
&+\sum_{a\ge 1}\left(n(an-2-\delta_{{}_m\mathrm{I}_1,\mathrm{I\hspace{-.1em}I}})
-a\mu'-2A_n-\delta_{a=1}\right)j'_{0,a}(F_p)   \label{stn4eq} \\
&+\sum_{a\ge 1}\left(an^2-\mu'\left(a+\frac{2r}{n}-2\right)-2A_n-\delta_{a=1}\right)j''_{0,a}(F_p)  \nonumber \\
&+\left(2(n-2)A_n-n+\left(\frac{r}{n}-1\right)\mu' \right)\kappa(F_p)-\mu' \left(\frac{r}{n}-2\right)\kappa^{(3)}(F_p).  \nonumber 
\end{align}

Since $2(n-2)A_n-n+\mu' >0$, we have
\begin{align*}
&\left(2(n-2)A_n-n+\left(\frac{r}{n}-1\right)\mu' \right)\kappa(F_p)-\mu' \left(\frac{r}{n}-2\right)\kappa^{(3)}(F_p) \\
&\ge \left(2(n-2)A_n-n+\mu' \right)\kappa(F_p) \\
&\ge 0.
\end{align*}
The coefficient of $j'_{0,1}(F_p)$ in \eqref{stn4eq} is 
\begin{align}
&n(n-2-\delta_{{}_m\mathrm{I}_1,\mathrm{I\hspace{-.1em}I}})
-\mu'-2A_n-\delta_{a=1} \nonumber \\
&=n^2-(4+\delta_{{}_m\mathrm{I}_1,\mathrm{I\hspace{-.1em}I}})n+1-\frac{(n-1)(n-2)}{k(n+1)}. \label{j'01eq}
\end{align}
If $\delta_{{}_m\mathrm{I}_1,\mathrm{I\hspace{-.1em}I}}=0$,
\eqref{j'01eq} is negative if and only if $n=4$ and $k=1$. 
If $\delta_{{}_m\mathrm{I}_1,\mathrm{I\hspace{-.1em}I}}=1$ and $n\ge 5$, 
\eqref{j'01eq} is non-negative.
If $\delta_{{}_m\mathrm{I}_1,\mathrm{I\hspace{-.1em}I}}=1$ and $n=4$, 
\eqref{j'01eq} is $-3-6/5k<0$.
Note that $j'(F_p)=0$ if $k=1$ since $r=n$ and any singularity of $R$ has the multiplicity $n$.
Thus, we may not consider the case where $\delta_{{}_m\mathrm{I}_1,\mathrm{I\hspace{-.1em}I}}=0$, $n=4$ and $k=1$.
We can check that the coefficient of $j'_{0,a}(F_p)$ in \eqref{stn4eq} is positive for $a\ge 2$.
Moreover, we also can check that the coefficient of $j''_{0,a}(F_p)$ in \eqref{stn4eq} is positive for $a\ge 1$.

\smallskip

\noindent
(i,1) We assume that $\eta'_p=0$. Then, clearly \eqref{stn4eq} is non-negative.

\smallskip

\noindent
(i,2) We assume that $\eta'_p\neq0$, $\chi_{\varphi}(F_p)\ge 1/6$ (i.e. $\Gamma_p$ is a singular fiber not of type $(\mathrm{{}_m\mathrm{I}_1})$) and $\iota^{(3)}(F_p)=0$. 
Then we have
$$
C_n\chi_{\varphi}(F_p)\ge \frac{1}{6}C_n=\frac{n}{6}\left(11-\frac{12n}{k(n-1)(n+1)}\right).
$$
The coefficient of $\delta_{\eta'_p\neq 0}$ is
$$
\mu' \left(\frac{r}{n}-\delta_{{}_m\mathrm{I}_1,\mathrm{I\hspace{-.1em}I}}\right)=-n+\frac{n}{k}\delta_{{}_m\mathrm{I}_1,\mathrm{I\hspace{-.1em}I}}.
$$
If $n\neq 4$ or $\delta_{{}_m\mathrm{I}_1,\mathrm{I\hspace{-.1em}I}}=0$ or $j'_{0,1}(F_p)=0$,
then \eqref{stn4eq} is positive since 
$$
\frac{n}{6}\left(11-\frac{12n}{k(n-1)(n+1)}\right)-n=\frac{n}{6}\left(5-\frac{12n}{k(n-1)(n+1)}\right)>0.
$$
If $n=4$ and $\delta_{{}_m\mathrm{I}_1,\mathrm{I\hspace{-.1em}I}}=1$ and $j'_{0,1}(F_p)\neq0$ (we denote this condition by $(\#)$),
then we have $j'_{0,1}(F_p)=1$ and then \eqref{stn4eq} is positive
since
$$
\frac{n}{6}\left(11-\frac{12n}{k(n-1)(n+1)}\right)-n+\frac{n}{k}-3-\frac{6}{5k}=\frac{1}{3}+\frac{26}{15k}>0.
$$

\smallskip

\noindent
(i,3) We assume that $\eta'_p\neq0$ and $\chi_{\varphi}(F_p)=1/12$ (i.e. $\Gamma_p$ is a singular fiber of type $(\mathrm{{}_m\mathrm{I}_1})$).
Let $m_1$ be the multiplicity of the singular point $x_1$ of $R$ which is singular for $\Gamma_p$.
If $m_1\in n\mathbb{Z}$, then $x_1$ contributes $m_1-2$ to $\alpha^{+}_0(F_p)$ 
(Note that $x_1$ is a $1$-vertical type singularity).
In particular, $x_1$ contributes at least $n-2$ to $\alpha^{+}_0(F_p)$.
If the condition $(\#)$ does not hold, \eqref{stn4eq} is positive 
since 
$$
\frac{n}{12}\left(11-\frac{12n}{k(n-1)(n+1)}\right)-n+(n-2)A_n=n^2-\frac{37}{12}n+2-\frac{n^2+(n-2)(n-1)(2n-1)}{k(n-1)(n+1)}
$$
increases monotonically with respect to $n$ and
$$
\frac{17}{3}-\frac{58}{15k}>0
$$
when $n=4$.
If the condition $(\#)$ holds, \eqref{stn4eq} is also positive 
since
$$
\frac{17}{3}-\frac{58}{15k}+\frac{4}{k}-3-\frac{6}{5k}=\frac{8}{3}-\frac{16}{15k}>0.
$$

\smallskip

\noindent
(i,4) We assume that $\eta'_p\neq0$ and $\chi_{\varphi}(F_p)=0$ (i.e. $\Gamma_p$ is a smooth elliptic curve).
Then \eqref{stn4eq} is positive since $j'_{1,\bullet}(F_p)=1$.

\smallskip

\noindent
(i,5) We assume that $\iota^{(3)}(F_p)=1$. 
From Lemma~\ref{3verlem}, we may consider the following $3$ cases.

\smallskip

\noindent
(i,5,I\hspace{-.1em}I) If $\Gamma_p$ is a singular fiber of type $(\mathrm{I\hspace{-.1em}I})$,
then $\beta_p=n-7$, $\chi_{\varphi}(F_p)=1/6$ and $\kappa(F_p)=\kappa^{(2)}(F_p)\ge 1$.
From the argument in (i,2), it is sufficient to show that
$$
A_n\beta_p+\left(2(n-2)A_n-n+\left(\frac{r}{n}-1\right)\mu' \right)\kappa(F_p)>0.
$$
This inequality is true, since 
\begin{align*}
&(n-7)A_n+2(n-2)A_n-n+\left(\frac{r}{n}-1\right)\mu' \\
&=(3n-11)\left(n-1-\frac{2n-1}{k(n+1)}\right)-\frac{n}{k} \\
&>0.
\end{align*}
Note that $k\ge n-1$ since $j''(F_p)\neq 0$.

\smallskip

\noindent
(i,5,I\hspace{-.1em}I\hspace{-.1em}I) If $\Gamma_p$ is a singular fiber of type $(\mathrm{I\hspace{-.1em}I\hspace{-.1em}I})$,
then $\beta_p=-n-1$, $\chi_{\varphi}(F_p)=1/4$ and $\kappa(F_p)=\kappa^{(2)}(F_p)\ge 1$.
Since $j''(F_p)\ge 1$ and the coefficient of $j''_{0,a}(F_p)$ is greater than $1$ for any $a\ge 1$,
it is sufficient to show that
$$
1+A_n\beta_p+\frac{1}{12}C_n+\left(2(n-2)A_n-n+\left(\frac{r}{n}-1\right)\mu' \right)\kappa(F_p)\ge 0.
$$
The left hand side of it is greater than or equal to
\begin{align}
&1-(n+1)A_n+\frac{1}{12}C_n+2(n-2)A_n-n+\left(\frac{r}{n}-1\right)\mu' \nonumber \\
&=n^2-\frac{61}{12}n+6-\frac{3n^3-12n^2+15n-5}{k(n-1)(n+1)} \label{i5IIIeq}
\end{align}
and \eqref{i5IIIeq} increases monotonically with respect to $n$.
If $n=4$, \eqref{i5IIIeq} is $5/3-11/3k>0$.
Note that $k\ge n-1$ since $j''(F_p)\neq 0$.

\smallskip

\noindent
(i,5,I\hspace{-.1em}V) If $\Gamma_p$ is a singular fiber of type $(\mathrm{I\hspace{-.1em}V})$,
then $\beta_p=-2$, $\chi_{\varphi}(F_p)=1/3$ and $j'_{0,\bullet}(F_p)\ge 3$.
Thus, it is sufficient to show that $A_n\beta_p+C_n/6+3\cdot\eqref{j'01eq}$ is positive.
By a computation, this is equal to
$$
3n^2-\frac{73}{6}n+5-\frac{3n^3-14n^2+21n-8}{k(n-1)(n+1)}
$$
and we can check that it is positive.

\smallskip

From (i,1) through (i,5), we have $e_f(F_p)-\mu\chi_f(F_p)\ge 0$ for $n\ge 4$.
On the other hand, if $n=3$ and $g=4$, one can easily classify all singular fibers of primitive cyclic covering fibrations of type $(4,1,3)$ because $R$ has no singularities of multiplicity greater than $3$, and check $e_f(F_p)\ge (9/2)\chi_f(F_p)$ for any fiber germ $F_p$.
Thus, Theorem~\ref{h1upperthm}~(1) follows.

\medskip

\noindent
(ii) We assume $n=3$ and $g>4$. 
The coefficient of $j''_{0,1}(F_p)$ in \eqref{steq} is
$$
-\left(\frac{2}{9}r-\frac{17}{18}\right)\mu<0.
$$
Applying Lemma~\ref{j''01lem}~(1) to the term of $j''_{0,1}(F_p)$, \eqref{steq} is greater than or equal to

\begin{align}
&\sum_{k\ge 1}\left(3-\frac{2}{3}\mu k\right)\alpha'_k(F_p) +\frac{r(m-1)(4-\mu)}{2m} +\left(36-\left(\frac{2}{9}r+3\right)\mu \right)\chi_{\varphi}(F_p) \nonumber \\
&+\left(2-\frac{5}{18}\mu \right)\beta_p-\frac{2}{3}\mu \left(\frac{r}{3}-\delta_{{}_m\mathrm{I}_1,\mathrm{I\hspace{-.1em}I}}\right)\delta_{\eta'_p\neq 0}+\left(1-\frac{7}{18}\mu \right)(\eta'_p-\delta_{\mathrm{cyc}})
 \nonumber \\
&+\sum_{a\ge 2}\left(9a-11-3\delta_{{}_m\mathrm{I}_1,\mathrm{I\hspace{-.1em}I}}-\left(\frac{4}{9}(a-1)r-\frac{11}{9}a+\frac{17}{18}\right)\mu \right)j'_{0,a}(F_p)   \label{stn3eq} \\
&+\left(-3+\frac{5}{18}\mu \right)j'_{0,1}(F_p)+\sum_{a\ge 2}\left(9a-8-\left(\frac{4}{9}ar-\frac{11}{9}a-\frac{2}{3}\right)\mu \right)j''_{0,a}(F_p)  \nonumber \\
&+\sum_{a\ge 1}\left(9a-1-\left(\frac{4}{9}ar-\frac{11}{9}a-\frac{7}{18}\right)\mu \right)j'_{1,a}(F_p) \nonumber \\
&+\left(1+\left(\frac{2}{9}r-\frac{11}{9}\right)\mu \right)\kappa(F_p)-\frac{2}{3}\mu \left(\frac{r}{3}-2\right)\kappa^{(3)}(F_p). \nonumber
\end{align}
We remark that the term of $\eta''_p$ vanishes by the definition of $\mu$ and
\begin{align*}
&\left(1+\left(\frac{2}{9}r-\frac{11}{9}\right)\mu \right)\kappa(F_p)-\frac{2}{3}\mu \left(\frac{r}{3}-2\right)\kappa^{(3)}(F_p) \\
&\ge \left(1+\frac{1}{9}\mu \right)\kappa(F_p) \\
&\ge 0.
\end{align*}
We can check that the coefficient of $j'_{0,a}(F_p)$ (resp. $j''_{0,a}(F_p)$, $j'_{1,a}(F_p)$) in \eqref{stn3eq} are positive for $a\ge 2$ (resp. $a\ge 2$, $a\ge 1$).

\smallskip

\noindent
(ii,1) Assume that $\eta'_p=0$. 
Then \eqref{stn3eq} is non-negative. 

\smallskip

\noindent
(ii,2) Assume that $\eta'_p\neq 0$ and $j'^{t}_{0,1}(F_p)=4$ for some $t$.
From Lemma~\ref{etalem}~(1,iii), it follows that $\Gamma_p$ is a singular fiber of type $(\mathrm{I}^{*}_k)$ for some $k$, $\chi_{\varphi}(F_p)=(k+6)/12$, $r\ge 9$, $\Gamma_p\subset R$, $\eta'_p=1$, $j'_{0,1}(F_p)=4$ and $\delta_{\mathrm{cyc}}=0$.
Considering the terms of $\chi_{\varphi}(F_p)$, $\delta_{\eta'_p\neq 0}$, $\eta'_p$ and $j'_{0,1}(F_p)$, \eqref{stn3eq} is greater than or equal to
\begin{align*}
&\left(36-\left(\frac{2}{9}r+3\right)\mu \right)\frac{k+6}{12}-\frac{2}{9}r\mu+\left(1-\frac{7}{18}\mu \right)+4\left(-3+\frac{5}{18}\mu \right) \\
&=7-\frac{8(3r+2)}{3(4r-13)}.
\end{align*}
This is positive since $r\ge 9$.

\smallskip

\noindent
(ii,3) Assume that $\eta'_p\neq 0$, $\iota^{(3)}(F_p)=0$, $r\ge 9$ and $j'^{t}_{0,1}(F_p)\ge 3$ for any $t$.
Then 
$$
\frac{1}{3}j'_{0,1}(F_p)\le \eta'_p-\delta_{\mathrm{cyc}}
$$ 
from Lemma~\ref{etalem}~(1,i), (1,ii).
From this and Lemma~\ref{chiphilem}~(1), \eqref{stn3eq} is greater than or equal to
\begin{align*}
\left(36-\left(\frac{2}{9}r+3\right)\mu \right)\frac{j'_{0,1}(F_p)+1}{12}-\frac{2}{9}r\mu+\frac{1}{3}\left(1-\frac{7}{18}\mu \right)j'_{0,1}(F_p)+\left(-3+\frac{5}{18}\mu \right)j'_{0,1}(F_p).
\end{align*}
One can check by a computation that this is positive since $r\ge 9$.

\smallskip

\noindent
(ii,4) Assume that $\eta'_p\neq 0$, $\iota^{(3)}(F_p)=0$, $r=6$ and $\chi_{\varphi}(F_p)\ge (j'_{0,1}(F_p)+2)/12$.
By the same argument as in (ii,3), \eqref{stn3eq} is greater than or equal to
\begin{align*}
&\left(36-\left(\frac{2}{9}r+3\right)\mu \right)\frac{j'_{0,1}(F_p)+2}{12}-\frac{2}{9}r\mu+\frac{1}{3}\left(1-\frac{7}{18}\mu \right)j'_{0,1}(F_p)+\left(-3+\frac{5}{18}\mu \right)j'_{0,1}(F_p) \\
&=-\frac{13}{99}j'_{0,1}(F_p)+\frac{50}{33}.
\end{align*}
On the other hand, one sees that $j'_{0,1}(F_p)\le 6$ since $r=6$. Then
$$
-\frac{13}{99}j'_{0,1}(F_p)+\frac{50}{33}>0.
$$

\smallskip

\noindent
(ii,5) Assume that $\eta'_p\neq 0$, $\iota^{(3)}(F_p)=0$, $r=6$ and $\chi_{\varphi}(F_p)<(j'_{0,1}(F_p)+2)/12$.
If $j'_{0,1}(F_p)=0$, then \eqref{stn3eq} is positive since $j'(F_p)\neq 0$.
Then there are the following two cases only.

\smallskip

\noindent
(ii,5,$\mathrm{I}_2$) $\Gamma_p$ is a singular fiber of type $({}_m\mathrm{I}_2)$ 
and only one component of $\Gamma_p$ is contained in $R$ and brown up just once.

\smallskip

\noindent
(ii,5,$\mathrm{I}_3$) $\Gamma_p$ is a singular fiber of type $({}_m\mathrm{I}_3)$ 
and only two component of $\Gamma_p$ are contained in $R$ and brown up just once at the intersection point of these.

In both cases, we can see that 
$$
\gamma_p \le \left(\frac{r}{n}-1\right)\delta_{\eta'_p\neq 0}+\left(\frac{r}{n}-1\right)\eta''_p 
$$
from the proof of Lemma~\ref{gammaiotalem},
since the component of $\Gamma_p$ not contained in $R$ intersects with $R_h$.
Thus, it is sufficient to show that
\begin{align*}
\left(36-\left(\frac{2}{9}r+3\right)\mu \right)\frac{j'_{0,1}(F_p)+1}{12}-\frac{2}{3}\mu\left(\frac{r}{3}-1\right)+\left(1-\frac{7}{18}\mu \right)+\left(-3+\frac{5}{18}\mu \right)j'_{0,1}(F_p)
\end{align*}
is positive.
This is equal to $(-2j'_{0,1}(F_p)+10)/11>0$ by a computation.

\smallskip

\noindent
(ii,6) Assume that $\iota^{(3)}(F_p)=1$. 
From Lemma~\ref{3verlem}, we may consider the following $3$ cases.

\smallskip

\noindent
(ii,6,I\hspace{-.1em}I) If $\Gamma_p$ is a singular fiber of type $(\mathrm{I\hspace{-.1em}I})$,
then we have $\chi_{\varphi}(F_p)=1/6$, $\beta_p=-4$, $\delta_{\mathrm{cyc}}=0$, $j'_{0,1}(F_p)=0$, $j'_{0,\bullet}(F_p)=1$, $j''_{0,\bullet}(F_p)\ge 2$ and $\kappa(F_p)=\kappa^{(2)}(F_p)\ge 1$.
One can see easily that \eqref{stn3eq} is positive by a computation.

\smallskip

\noindent
(ii,6,I\hspace{-.1em}I\hspace{-.1em}I) If $\Gamma_p$ is a singular fiber of type $(\mathrm{I\hspace{-.1em}I\hspace{-.1em}I})$,
then we have $\chi_{\varphi}(F_p)=1/4$, $\beta_p=-4$, $\delta_{\mathrm{cyc}}=0$, $j'_{0,1}(F_p)=0$, $j'_{0,\bullet}(F_p)=2$, $j''_{0,\bullet}(F_p)\ge 1$ and $\kappa(F_p)=\kappa^{(2)}(F_p)\ge 1$.
One can see easily that \eqref{stn3eq} is positive by a computation.

\smallskip

\noindent
(ii,6,I\hspace{-.1em}V) If $\Gamma_p$ is a singular fiber of type $(\mathrm{I\hspace{-.1em}V})$,
then we have $\chi_{\varphi}(F_p)=1/3$, $\beta_p=-2$, $\delta_{\mathrm{cyc}}=0$, $j'_{0,\bullet}(F_p)=3$ and $\kappa^{(3)}(F_p)=0$.
If $j'_{0,1}(F_p)\le 2$, then we can check that \eqref{stn3eq} is positive.
Suppose that $j'_{0,1}(F_p)=3$.
Then any component of $\Gamma_p$ is brown up only once.
Thus, the multiplicity of the singularity of $R$ on $\Gamma_p$ is $r/3+3\ge 6$.
In particular, $r\ge 9$.
Since this singularity is a $3$-vertical $3\mathbb{Z}$ type singularity, we have
$$
\alpha^+_0(F_p)\ge \iota(F_p)+2\kappa(F_p)+\beta_p+3
$$
from the proof of Lemma~\ref{alpha0lem}.
Then it suffices to show that
\begin{align*}
\left(2-\frac{5}{18}\mu \right)+\left(36-\left(\frac{2}{9}r+3\right)\mu \right)\frac{1}{3}-\frac{2}{9}r\mu+\left(1-\frac{7}{18}\mu \right)+3\left(-3+\frac{5}{18}\mu \right)
\end{align*}
is positive.
This is equal to 
$$
6-\frac{4(16r+45)}{9(4r-13)}>0.
$$

\smallskip

From (ii,1) through (ii,6), we have $e_f(F_p)-\mu\chi_f(F_p)\ge 0$.
Thus, Theorem~\ref{h1upperthm}~(2) follows.

\medskip

\noindent
(iii) We assume $n=2$ and $g\ge 3$. 
From Lemmas~\ref{alpha0lem} and \ref{clearlem}~(2), \eqref{presteq} is greater than or equal to

\begin{align}
&\sum_{k\ge 1}\left(2-\frac{1}{4}\mu k\right)\alpha'_k(F_p) +\frac{r(m-1)(4-\mu)}{4m}+\left(24-\left(\frac{r}{8}+2\right)\mu \right)\chi_{\varphi}(F_p)\nonumber \\
&-\frac{1}{4}\mu \left(\frac{r}{2}-j'_{0,1}(F_p)-\delta_{{}_m\mathrm{I}_1,\mathrm{I\hspace{-.1em}I}}\right)\delta_{\eta'_p\neq 0}+\left(2-\frac{1}{4}\mu \right)(\eta'_p-\delta_{\mathrm{cyc}})+\left(4-\frac{1}{8}r\mu \right)\widehat{\eta}_p \nonumber \\
&+\sum_{a\ge 1}\left(4a-8-2\delta_{{}_m\mathrm{I}_1,\mathrm{I\hspace{-.1em}I}}-\delta_{a=1}-\frac{1}{4}(a-2)\mu \right)j'_{0,a}(F_p) \nonumber \\
&+\sum_{a\ge 2}\left(4a-6-\frac{1}{4}\left(a+\frac{r}{2}-3\right)\mu \right)j''_{0,a}(F_p)+\sum_{a\ge 1}\left(4a-2-\frac{1}{4}\left(a-1\right)\mu \right)j'_{1,a}(F_p)   \label{prestn2eq} \\
&+\sum_{k\ge 1}\left(1-\frac{1}{8}(r-2)\mu-\frac{1}{4}\mu k\right)\alpha^{\mathrm{tr}}_{(2k+1\to 2k+1)}(F_p)-\frac{1}{4}\mu \alpha_{\frac{r}{2m}}^{2\mathbb{Z}+1}(F_p) \nonumber \\
&+\sum_{k\ge 1}\left(1-\frac{1}{8}(r-2)\mu+\left(2-\frac{1}{2}\mu \right) k\right)\alpha^{\mathrm{co},0}_{(2k+1\to 2k+1)}(F_p)-\left(2-\frac{1}{8}(r-2)\mu \right)\kappa(F_p) \nonumber \\
&+\sum_{k\ge 1}\left(-3+\frac{1}{4}\mu+\left(2-\frac{1}{2}\mu \right) k\right)\alpha^{\mathrm{co},1}_{(2k+1\to 2k+1)}(F_p)-\frac{1}{8}(r-4) \mu \kappa^{(3)}(F_p). \nonumber 
\end{align}
The coefficients of $\alpha^{\mathrm{tr}}_{(2k+1\to 2k+1)}(F_p)$ and $\alpha^{\mathrm{co},0}_{(2k+1\to 2k+1)}(F_p)$ in \eqref{prestn2eq} are non-negative and
that of $\alpha^{\mathrm{tr}}_{(r-1\to r-1)}(F_p)$ is $0$ by the definition of $\mu$.
The coefficient of $\alpha^{\mathrm{co},1}_{(2k+1\to 2k+1)}(F_p)$ in \eqref{prestn2eq} is positive except for $k=1$ and
that of $\alpha^{\mathrm{co},1}_{(3\to 3)}(F_p)$ is $-1-\mu/4<0$.
The coefficient of $\kappa(F_p)$ in \eqref{prestn2eq} is $-3/2$.
Applying Lemmas \ref{n2h1kappalem} and \ref{j''01lem}~(2) to $(3/2)\kappa(F_p)$ and $(1+\mu/4)\sum_{k\ge 1}\alpha^{\mathrm{co},1}_{(2k+1\to 2k+1)}(F_p)$, we see that \eqref{prestn2eq} is greater than or equal to

\begin{align}
&\sum_{k\ge 1}\left(2-\frac{1}{4}\mu k\right)\alpha'_k(F_p) +\frac{r(m-1)(4-\mu)}{4m}+\left(24-\left(\frac{r}{8}+2\right)\mu \right)\chi_{\varphi}(F_p) \nonumber \\
&-\frac{1}{4}\mu \left(\frac{r}{2}-\delta_{{}_m\mathrm{I}_1,\mathrm{I\hspace{-.1em}I}}\right)\delta_{\eta'_p\neq 0}+\left(2-\frac{1}{4}\mu \right)(\eta'_p-\delta_{\mathrm{cyc}})+\left(3-\frac{1}{8}(r+2)\mu \right)\widehat{\eta}_p \nonumber \\
&-\left(5-\frac{1}{2}\mu \right)j'_{0,1}(F_p)-\left(1+2\delta_{{}_m\mathrm{I}_1,\mathrm{I\hspace{-.1em}I}}\right)j'_{0,2,\mathrm{even}}(F_p)-\left(2+\frac{1}{4}\mu \right)j'_{0,2,\mathrm{odd}}(F_p) \nonumber \\
&-\left(2\delta_{{}_m\mathrm{I}_1,\mathrm{I\hspace{-.1em}I}}+\frac{3}{4}\mu \right)j'_{0,3}(F_p)
+\sum_{a\ge 4}\left(2a-6-2\delta_{{}_m\mathrm{I}_1,\mathrm{I\hspace{-.1em}I}}-\frac{1}{4}(2a-3)\mu\right)j'_{0,a}(F_p) \nonumber \\
&+\sum_{a\ge 2}\left(2a-3-\frac{1}{4}\left(2a+\frac{r}{2}-5\right)\mu \right)j''_{0,a}(F_p)+\sum_{a\ge 1}\left(2a-\frac{1}{2}\left(a-1\right)\mu \right)j'_{1,a}(F_p)   \label{stn2eq} \\
&+\sum_{k\ge 1}\left(1-\frac{1}{8}(r-2)\mu-\frac{1}{4}\mu k\right)\alpha^{\mathrm{tr}}_{(2k+1\to 2k+1)}(F_p)+\left(1-\frac{1}{4}\mu \right) \alpha_{\frac{r}{2m}}^{2\mathbb{Z}+1}(F_p) \nonumber \\
&+\sum_{k\ge 1}\left(1-\frac{1}{8}(r-2)\mu+\left(2-\frac{1}{2}\mu \right) k\right)\alpha^{\mathrm{co},0}_{(2k+1\to 2k+1)}(F_p) \nonumber \\
&+\sum_{k\ge 1}\left(2-\frac{1}{2}\mu\right)(k-1)\alpha^{\mathrm{co},1}_{(2k+1\to 2k+1)}(F_p)-\frac{1}{8} \left(r-4\right)\mu \kappa^{(3)}(F_p), \nonumber 
\end{align}
where $j'_{0,2,\mathrm{even}}(F_p)=j'_{0,2}(F_p)-j'_{0,2,\mathrm{odd}}(F_p)$.

\smallskip

\noindent
(iii,1) Assume that $\eta'_p=0$. Then \eqref{stn2eq} is clearly non-negative.

\smallskip

\noindent
(iii,2) Assume that $\eta'_p\neq 0$ and $j'^{t}_{0,2,\mathrm{odd}}(F_p)=4$ for some $t$.
From Lemma~\ref{etalem}~(2,iii), $\Gamma_p$ is a singular fiber of type $(\mathrm{I}^{*}_k)$,
 $\Gamma_p\subset R$, $\kappa^{(3)}(F_p)=0$, $\eta'_p=1$, $\delta_{\mathrm{cyc}}=0$, $j'_{0,1}(F_p)=0$ and $j'_{0,2,\mathrm{odd}}(F_p)=4$.
Clearly we have $\chi_{\varphi}(F_p)=(j'_{0,2}(F_p)+j'_{0,3}(F_p)+1)/12$.
Considering the terms of $\chi_{\varphi}(F_p)$, $\delta_{\eta'_p\neq 0}$, $\eta'_p$, $j'_{0,2,\mathrm{even}}(F_p)$, $j'_{0,2,\mathrm{odd}}(F_p)$ and $j'_{0,3}(F_p)$, \eqref{stn2eq} is greater than or equal to
\begin{align*}
&\left(24-\left(\frac{r}{8}+2\right)\mu \right)\frac{1}{12}(j'_{0,2,\mathrm{even}}(F_p)+j'_{0,3}(F_p)+5)-\frac{1}{8}r\mu+\left(2-\frac{1}{4}\mu \right) \\
&-j'_{0,2,\mathrm{even}}(F_p)-4\left(2+\frac{1}{4}\mu \right)-\frac{3}{4}\mu j'_{0,3}(F_p) \\
&=\left(1-\left(\frac{r}{96}+\frac{1}{6}\right)\mu \right)j'_{0,2,\mathrm{even}}(F_p)+\left(2-\left(\frac{r}{96}+\frac{11}{12}\right)\mu \right)j'_{0,3}(F_p)+4-\left(\frac{17}{96}r+\frac{25}{12}\right)\mu,
\end{align*}
which is positive.

\smallskip

\noindent
(iii,3) Assume that $\eta'_p\neq 0$, $\kappa^{(3)}(F_p)=0$, $\delta_{{}_m\mathrm{I}_1,\mathrm{I\hspace{-.1em}I}}=0$, $j'^{t}_{0,2,\mathrm{odd}}(F_p)\le 3$ for any $t$ and 
$$
\chi_{\varphi}(F_p)\ge \frac{1}{12}(2j'_{0,1}(F_p)+j'_{0,2}(F_p)+j'_{0,3}(F_p)+1).
$$
From Lemma~\ref{etalem}~(2,i), (2,ii), we have $j'_{0,1}(F_p)+j'_{0,2,\mathrm{odd}}(F_p)/3\le \eta'_p-\delta_{\mathrm{cyc}}$.
Then \eqref{stn2eq} is greater than or equal to
\begin{align*}
&\left(24-\left(\frac{r}{8}+2\right)\mu \right)\frac{1}{12}(2j'_{0,1}(F_p)+j'_{0,2}(F_p)+j'_{0,3}(F_p)+1)-\frac{1}{8}r\mu \\
&+\left(2-\frac{1}{4}\mu \right)\left(j'_{0,1}(F_p)+\frac{1}{3}j'_{0,2,\mathrm{odd}}(F_p)\right)-\left(5-\frac{1}{2}\mu \right)j'_{0,1}(F_p)-j'_{0,2,\mathrm{even}}(F_p) \\
&-\left(2+\frac{1}{4}\mu \right)j'_{0,2,\mathrm{odd}}(F_p)-\frac{3}{4}\mu j'_{0,3}(F_p) \\
&=\left(1-\left(\frac{r}{48}+\frac{1}{12}\right)\mu \right)j'_{0,1}(F_p)+\left(\frac{2}{3}-\left(\frac{r}{96}+\frac{1}{2}\right)\mu \right)j'_{0,2,\mathrm{odd}}(F_p) \\
&+\left(1-\left(\frac{r}{96}+\frac{1}{6}\right)\mu \right)j'_{0,2,\mathrm{even}}(F_p)+\left(2-\left(\frac{r}{96}+\frac{11}{12}\right)\mu \right)j'_{0,3}(F_p) \\
&+2-\left(\frac{13}{96}r+\frac{1}{6}\right)\mu.
\end{align*}
The coefficients of $j'_{0,1}(F_p)$, $j'_{0,2,\mathrm{even}}(F_p)$, $j'_{0,3}(F_p)$ and the constant term are positive since $r\ge 4$, and $j_{0,2,\mathrm{odd}}(F_p)$ is also positive for $r\ge 6$.
Thus the above equation is positive when $r\ge 6$.
If $r=4$, then one can check by an easy computation that \eqref{stn2eq} is also positive, since $j'_{0,2,\mathrm{odd}}(F_p)\le 2$.

\smallskip

\noindent
(iii,4) Assume that $\eta'_p\neq 0$, $\kappa^{(3)}(F_p)=0$, $\delta_{{}_m\mathrm{I}_1,\mathrm{I\hspace{-.1em}I}}=0$, $j'^{t}_{0,2,\mathrm{odd}}(F_p)\le 3$ for any $t$ and 
$$
\chi_{\varphi}(F_p)< \frac{1}{12}(2j'_{0,1}(F_p)+j'_{0,2}(F_p)+j'_{0,3}(F_p)+1).
$$
From Lemma~\ref{chiphilem}~(2), $\Gamma_p$ is of type $({}_m\mathrm{I}_k)$, $(\mathrm{I}^{*}_k)$, $(\mathrm{I\hspace{-.1em}I}^{*})$, $(\mathrm{I\hspace{-.1em}I\hspace{-.1em}I}^{*})$ or $(\mathrm{I\hspace{-.1em}V}^{*})$.
Considering the numbers $j'_{0,1}(F_p)$,  $\chi_{\phi}(F_p)$ and Lemma~\ref{chiphilem}~(2), 
we can see that \eqref{stn2eq} is positive by the same argument as in (iii,3) except $\Gamma_p$ is of type $(\mathrm{I}^{*}_0)$ and $j'_{0,1}(F_p)=4$.

Suppose that $\Gamma_p$ is of type $(\mathrm{I}^{*}_0)$ and $j'_{0,1}(F_p)=4$.
Then the component not contributing to $j'_{0,1}(F_p)$ is a double component in $\Gamma_p$ and intersects with $R_h$.
Thus we have
$$
\alpha^+_0(F_p)\ge \frac{r}{2}+\sum_{k\ge 1}2k\alpha^{\mathrm{co}}_{(2k+1\to 2k+1)}(F_p)
$$
from the proof of Lemma~\ref{alpha0lem}.
Then one can see that $e_f(F_p)-\mu \chi_{\varphi}(F_p)$ is positive by a computation.

\smallskip

\noindent
(iii,5) Assume that  $\eta'_p\neq 0$, $\kappa^{(3)}(F_p)=0$ and $\delta_{{}_m\mathrm{I}_1,\mathrm{I\hspace{-.1em}I}}=1$. 
Then $j'_{0,1}(F_p)=0$, $j'(F_p)=1$, $\eta'_p=1$ and $\chi_{\varphi}(F_p)=1/12$ or $1/6$.
If $j'_{0,2,\mathrm{even}}(F_p)=1$, then $(\Gamma_p)_{\mathrm{red}}$ is blown up just once. 
Thus, the multiplicity of the singularity of $R$ which is singular for $(\Gamma_p)_{\mathrm{red}}$ is even. Hence $\delta_{\mathrm{cyc}}=0$.
Then we can check by a computation that \eqref{stn2eq} is positive.
If $j'_{0,2,\mathrm{even}}(F_p)=0$, then we can also check by a computation that \eqref{stn2eq} is positive.

\smallskip

\noindent
(iii,6) Assume that $\kappa^{(3)}(F_p)=1$.
From Lemma~\ref{3verlem}, we may consider the following $3$ cases.

\smallskip

\noindent
(iii,6,I\hspace{-.1em}I) If $\Gamma_p$ is a singular fiber of type $(\mathrm{I\hspace{-.1em}I})$,
then $\eta'_p=1$, $j'(F_p)=1$, $j'_{0,a}(F_p)=0$ for $a\le 3$, $j''(F_p)-j''_{0,1}(F_p)\ge 3$ and $\chi_{\varphi}(F_p)=1/6$.
Considering the terms of $\delta_{\eta'_p\neq 0}$, $\eta'_p$, $\kappa^{(3)}(F_p)$, $j'_{0,a}(F_p)$, $j''_{0,a}(F_p)$ and $\chi_{\varphi}(F_p)$ in \eqref{stn2eq}, we can check that \eqref{stn2eq} is positive.

\smallskip

\noindent
(iii,6,I\hspace{-.1em}I\hspace{-.1em}I) If $\Gamma_p$ is a singular fiber of type $(\mathrm{I\hspace{-.1em}I\hspace{-.1em}I})$,
then $\eta'_p=1$, $j'(F_p)=2$, $j'_{0,a}(F_p)=0$ for $a\le 2$, $j''(F_p)-j''_{0,1}(F_p)\ge 2$ and $\chi_{\varphi}(F_p)=1/4$.
Then we can also check that \eqref{stn2eq} is positive.

\smallskip

\noindent
(iii,6,I\hspace{-.1em}V) If $\Gamma_p$ is a singular fiber of type $(\mathrm{I\hspace{-.1em}V})$,
then $\eta'_p=1$, $j'(F_p)=3$, $j'_{0,1}(F_p)=j'_{0,2,\mathrm{even}}(F_p)=0$, $j''(F_p)-j''_{0,1}(F_p)\ge 1$ and $\chi_{\varphi}(F_p)=1/3$.
Similarly, we can check that \eqref{stn2eq} is positive. 

\smallskip

From (iii,1) through (iii,6), we have $e_f(F_p)-\mu\chi_f(F_p)\ge 0$.
Thus, Theorem~\ref{h1upperthm}~(3) follows.
\qed
\end{prfofthmh1}

\begin{exa} \label{fiberexa}
There exist singular fiber germs $F_p$ such that $K_f^2(F_p)=(12-\mu)\chi_f(F_p)$ and we can classify them.

\smallskip

\noindent
(i) Assume that $n\ge 4$, or $n=3$ and $g=4$.
Consider the situation that $\Gamma_p$ is smooth and $F_p$ is obtained the following sequence of singularity diagrams associated with $\Gamma_p$ (cf.\ \cite{enoki2}):

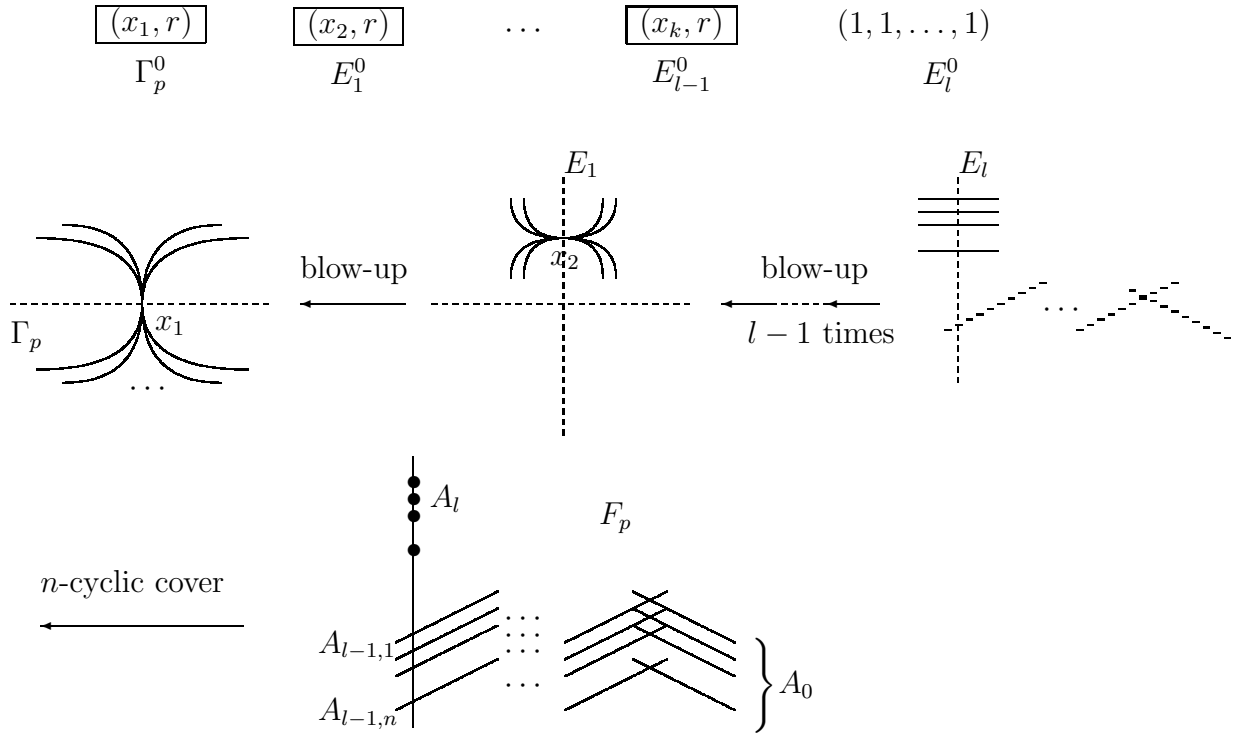
\begin{table}[H]
\begin{center}
\begin{tabular}{c}

 \begin{minipage}{0.15\hsize}
 \begin{tabular}{c}
   \\ \cline{1-1}
 \multicolumn{1}{|c|}{$(x_1,r)$}\\ \cline{1-1}
 \multicolumn{1}{c}{\lower1ex\hbox{$\Gamma_p^0$}}
 \end{tabular}
 \end{minipage}

\begin{minipage}{0.15\hsize}
 \begin{tabular}{c}
   \\ \cline{1-1}
 \multicolumn{1}{|c|}{$(x_2,r)$}\\ \cline{1-1}
 \multicolumn{1}{c}{\lower1ex\hbox{$E_1^0$}}
 \end{tabular}
 \end{minipage}

\begin{minipage}{0.1\hsize}
 \begin{tabular}{c}
  $\cdots$
 \end{tabular}
 \end{minipage}

\begin{minipage}{0.15\hsize}
 \begin{tabular}{c}
   \\ \cline{1-1}
 \multicolumn{1}{|c|}{$(x_k,r)$}\\ \cline{1-1}
 \multicolumn{1}{c}{\lower1ex\hbox{$E_{l-1}^0$}}
 \end{tabular}
 \end{minipage}

 \begin{minipage}{0.25\hsize}
 \begin{tabular}{ccc}
  & & \\
 $(1,1,\ldots,1)$\\
 \multicolumn{3}{c}{\lower1ex\hbox{$E_{l}^0$}}
 \end{tabular}
 \end{minipage}

\end{tabular}
\end{center}
\end{table}

\setlength\unitlength{0.35cm}
\begin{figure}[H]
\begin{center}
\begin{tabular}{c}

 \begin{minipage}{0.6\hsize}
 \begin{center}
\begin{picture}(20,5)
 \put(-7.5,-1){$x_1$}
 \multiput(-13,0)(0.4,0){25}{\line(1,0){0.2}}
 \qbezier(-8,0)(-8,3)(-5,3)
 \qbezier(-8,0)(-8,3)(-11,3)
 \qbezier(-8,0)(-8,2.5)(-4,2.5)
 \qbezier(-8,0)(-8,2.5)(-12,2.5)
 \qbezier(-8,0)(-8,-3)(-5,-3)
 \qbezier(-8,0)(-8,-3)(-11,-3)
 \qbezier(-8,0)(-8,-2.5)(-4,-2.5)
 \qbezier(-8,0)(-8,-2.5)(-12,-2.5)
 \put(-8.5,-3.5){$\cdots$}
 \put(-13,-1.5){$\Gamma_p$}

\put(2,0){\vector(-1,0){4}}
\put(-2,1){blow-up}

 \multiput(3,0)(0.4,0){25}{\line(1,0){0.2}}
 \multiput(8,-5)(0,0.4){25}{\line(0,1){0.2}}
 \qbezier(8,2.5)(10,2.5)(10,4)
 \qbezier(8,2.5)(10,2.5)(10,1)
 \qbezier(8,2.5)(6,2.5)(6,4)
 \qbezier(8,2.5)(6,2.5)(6,1)
 \qbezier(8,2.5)(9.5,2.5)(9.5,4)
 \qbezier(8,2.5)(9.5,2.5)(9.5,1)
 \qbezier(8,2.5)(6.5,2.5)(6.5,4)
 \qbezier(8,2.5)(6.5,2.5)(6.5,1)
 \put(7.5,1.5){$x_2$}
 \put(8,5){$E_1$}

 \multiput(23,-3)(0,0.4){20}{\line(0,1){0.2}}
 \multiput(22.5,-1)(0.4,0.2){10}{\line(1,0){0.25}}
 \multiput(27.5,-1)(0.4,0.2){10}{\line(1,0){0.25}}
 \multiput(29.5,0.5)(0.4,-0.2){10}{\line(1,0){0.25}}
 \put(26.2,-0.5){$\cdots$}


 \put(21.5,2){\line(1,0){3}}
 \put(21.5,3){\line(1,0){3}}
 \put(21.5,3.5){\line(1,0){3}}
 \put(21.5,4){\line(1,0){3}}
 \put(15.5,1){blow-up}
 \put(15,-1.5){$l-1$ times}
 \put(20,0){\vector(-1,0){2}}
 \multiput(15.5,0)(0.4,0){10}{\line(1,0){0.2}}
 \put(16,0){\vector(-1,0){2}}
 \put(23,5){$E_l$}
 \end{picture}
 \end{center}
 \end{minipage}

\end{tabular}
\end{center}
\end{figure}

\ \\

\setlength\unitlength{0.45cm}
\begin{figure}[H]
\begin{center}
\begin{tabular}{c}

 \begin{minipage}{0.6\hsize}
 \begin{center}
\begin{picture}(8,5)

\put(-7,0){\vector(-1,0){6}}
\put(-13,1){$n$-cyclic cover}

\put(3.5,3){$F_p$}

 \put(-2,-3){\line(0,1){8}}
 \put(-1.5,3.5){$A_l$}

 \qbezier(-2.5,-0.5)(-2.5,-0.5)(0.5,1)
 \qbezier(2.5,-0.5)(2.5,-0.5)(5.5,1)
 \qbezier(4.5,1)(4.5,1)(7.5,-0.5)
 \put(0.7,0){$\cdots$}
 \put(8,-2){$\Biggr\}$}
 \put(8.8,-2){$A_0$}
 \put(-4.8,-0.8){$A_{l-1,1}$}

 \qbezier(-2.5,-1)(-2.5,-1)(0.5,0.5)
 \qbezier(2.5,-1)(2.5,-1)(5.5,0.5)
 \qbezier(4.5,0.5)(4.5,0.5)(7.5,-1)
 \put(0.7,-0.5){$\cdots$}

 \qbezier(-2.5,-1.5)(-2.5,-1.5)(0.5,0)
 \qbezier(2.5,-1.5)(2.5,-1.5)(5.5,0)
 \qbezier(4.5,0)(4.5,0)(7.5,-1.5)
 \put(0.7,-1){$\cdots$}

 \qbezier(-2.5,-2.5)(-2.5,-2.5)(0.5,-1)
 \qbezier(2.5,-2.5)(2.5,-2.5)(5.5,-1)
 \qbezier(4.5,-1)(4.5,-1)(7.5,-2.5)
 \put(0.7,-2){$\cdots$}
 \put(-4.8,-2.8){$A_{l-1,n}$}

 \put(-2.2,2){$\bullet$}
 \put(-2.2,3){$\bullet$}
 \put(-2.2,3.5){$\bullet$}
 \put(-2.2,4){$\bullet$}

 \end{picture}
 \end{center}
 \end{minipage}

\end{tabular}
\end{center}
\end{figure}

\ \\

\noindent
Then we can write $F_p=A_0+\sum_{i=1}^{n}(A_{1,i}+\cdots+A_{l-1,i})+A_l$, $p_{a}(A_{0})=0$,  $g(A_{k,i})=0$ for $k=1,\ldots,l-1$ and $g(A_{l})=(r/2-1)(n-1)$ (note that $A_0$ may not be irreducible).
This singular fiber satisfies $K_f^2(F_p)=(12-\mu)\chi_f(F_p)$.
Indeed, $\alpha_k(F_p)=0$ for $k=0,1,\ldots,r/n-1$, $\alpha_{r/n}(F_p)=l$, $\varepsilon(F_p)=0$ and $\chi_{\varphi}(F_p)=0$. 
Thus $\chi_f(F_p)=r(n-1)(n+1)l/12n$, $e_f(F_p)=nl$ and then $e_f(F_p)/\chi_f(F_p)=12n^2/r(n-1)(n+1)=\mu$.
We can see from the proof of Theorem~\ref{h1upperthm} that any singular fiber $F_p$ satisfying $K_f^2(F_p)=(12-\mu)\chi_f(F_p)$ is obtained in this way.

\smallskip

\noindent
(ii) Assume that $n=3$ and $g>4$.
Consider the situation that $\Gamma_p$ is smooth and $F_p$ is obtained the following sequence of singularity diagrams associated with $\Gamma_p$:

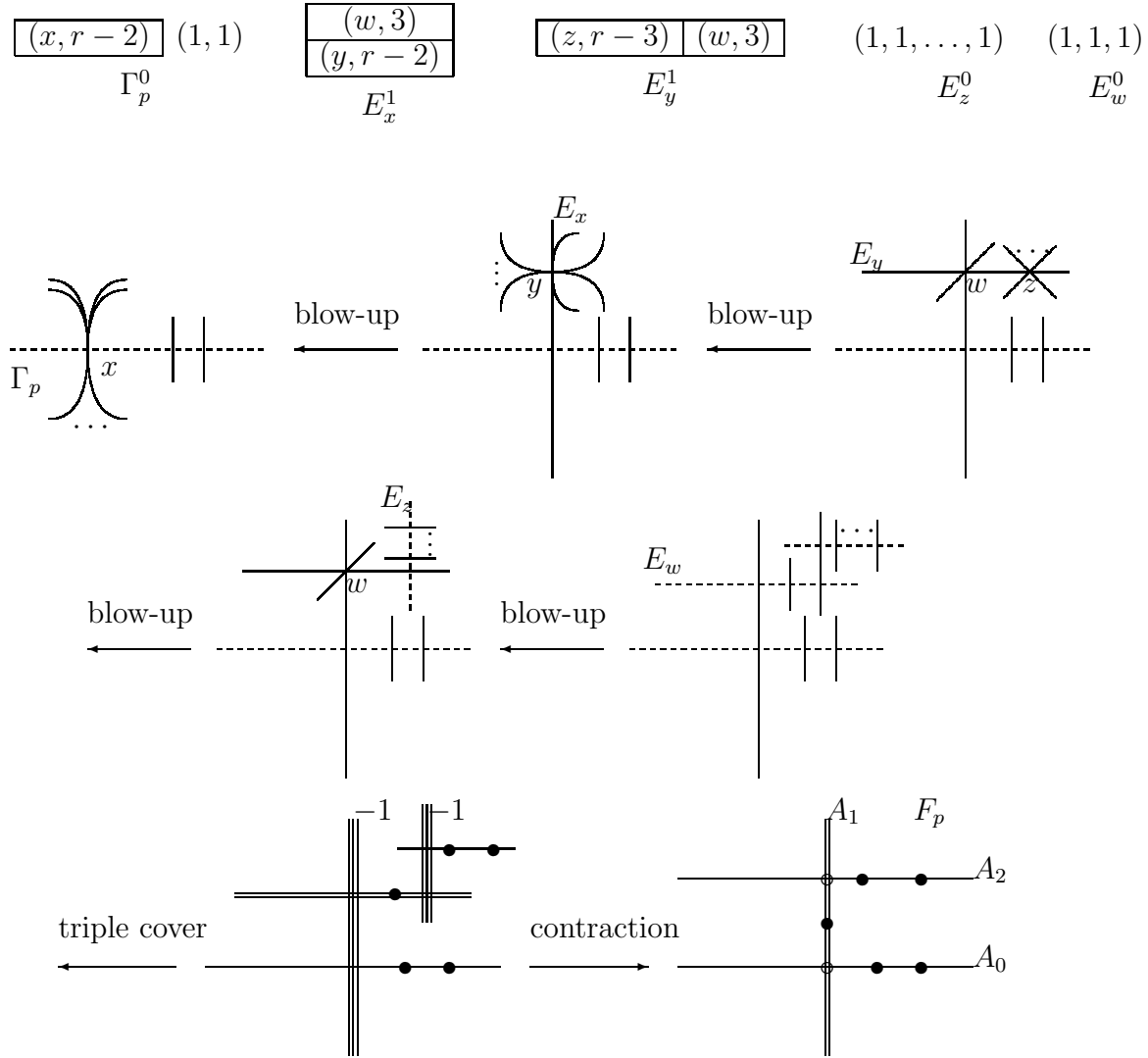
\begin{table}[H]
\begin{center}
\begin{tabular}{c}

 \begin{minipage}{0.23\hsize}
 \begin{tabular}{cc}
   & \\ \cline{1-1}
 \multicolumn{1}{|c|}{$(x,r-2)$}& $(1,1)$ \\ \cline{1-1}
 \multicolumn{2}{c}{\lower1ex\hbox{$\Gamma_p^0$}}
 \end{tabular}
 \end{minipage}

\begin{minipage}{0.18\hsize}
 \begin{tabular}{c}
   \\ \cline{1-1}
 \multicolumn{1}{|c|}{$(w,3)$}\\ \cline{1-1}
 \multicolumn{1}{|c|}{$(y,r-2)$}\\ \cline{1-1}
 \multicolumn{1}{c}{\lower1ex\hbox{$E_x^1$}}
 \end{tabular}
 \end{minipage}

\begin{minipage}{0.24\hsize}
 \begin{tabular}{cc}
   & \\ \cline{1-2}
 \multicolumn{1}{|c|}{$(z,r-3)$}& \multicolumn{1}{|c|}{$(w,3)$} \\ \cline{1-2}
 \multicolumn{2}{c}{\lower1ex\hbox{$E_y^1$}}
 \end{tabular}
 \end{minipage}

 \begin{minipage}{0.15\hsize}
 \begin{tabular}{ccc}
  & & \\
 $(1,1,\ldots,1)$\\
 \multicolumn{3}{c}{\lower1ex\hbox{$E_z^0$}}
 \end{tabular}
 \end{minipage}

\begin{minipage}{0.13\hsize}
 \begin{tabular}{cc}
  & \\
 $(1,1,1)$\\
 \multicolumn{2}{c}{\lower1ex\hbox{$E_w^0$}}
 \end{tabular}
 \end{minipage}

\end{tabular}
\end{center}
\end{table}

\setlength\unitlength{0.35cm}
\begin{figure}[H]
\begin{center}
\begin{tabular}{c}

 \begin{minipage}{0.6\hsize}
 \begin{center}
\begin{picture}(20,5)
 \put(-9.5,-1){$x$}
 \multiput(-13,0)(0.4,0){25}{\line(1,0){0.2}}
 \qbezier(-10,0)(-10,2.7)(-8.5,2.7)
 \qbezier(-10,0)(-10,2.7)(-11.5,2.7)
 \qbezier(-10,0)(-10,-2.7)(-8.5,-2.7)
 \qbezier(-10,0)(-10,-2.7)(-11.5,-2.7)
\qbezier(-10,0)(-10,2.3)(-8.5,2.3)
 \qbezier(-10,0)(-10,2.3)(-11.5,2.3)
 \put(-6.7,-1.25){\line(0,1){2.5}}
 \put(-5.5,-1.25){\line(0,1){2.5}}
 
 \put(-10.6,-3.3){$\cdots$}
 \put(-13,-1.5){$\Gamma_p$}

\put(2,0){\vector(-1,0){4}}
\put(-2,1){blow-up}

 \multiput(3,0)(0.4,0){25}{\line(1,0){0.2}}
 \put(8,-5){\line(0,1){10}}
 \qbezier(8,3)(10,3)(10,4.5)
 \qbezier(8,3)(10,3)(10,1.5)
 \qbezier(8,3)(6,3)(6,4.5)
 \qbezier(8,3)(6,3)(6,1.5)
 \qbezier(8,3)(8,4.5)(9,4.5)
 \qbezier(8,3)(8,1.5)(9,1.5)
 \put(9.8,-1.25){\line(0,1){2.5}}
 \put(11,-1.25){\line(0,1){2.5}}
 \put(7,2.2){$y$}
 \put(8,5.1){$E_x$}
 \put(5.7,2.5){$\vdots$}

 \multiput(19,0)(0.4,0){25}{\line(1,0){0.2}}
 \put(24,-5){\line(0,1){10}}
 \put(20,3){\line(1,0){8}}
 \put(25.8,-1.25){\line(0,1){2.5}}
 \put(27,-1.25){\line(0,1){2.5}}
 \qbezier(22.9,1.9)(24,3)(25.1,4.1)
 \qbezier(25.5,2)(26.5,3)(27.5,4)
 \qbezier(25.5,4)(26.5,3)(27.5,2)
 \put(14,1){blow-up}
 \put(18,0){\vector(-1,0){4}}
 \put(19.5,3.3){$E_y$}
 \put(25.9,3.5){$\cdots$}
 \put(26.2,2.2){$z$}
 \put(24,2.2){$w$}
 \end{picture}
 \end{center}
 \end{minipage}

\end{tabular}
\end{center}
\end{figure}

\ \\

\setlength\unitlength{0.35cm}
\begin{figure}[H]
\begin{center}
\begin{tabular}{c}

 \begin{minipage}{0.6\hsize}
 \begin{center}
\begin{picture}(20,5)

 \multiput(-5,0)(0.4,0){25}{\line(1,0){0.2}}
 \put(0,-5){\line(0,1){10}}
 \put(-4,3){\line(1,0){8}}
 \put(1.8,-1.25){\line(0,1){2.5}}
 \put(3,-1.25){\line(0,1){2.5}}
 \qbezier(-1.1,1.9)(0,3)(1.1,4.1)

 \multiput(2.5,1.5)(0,0.4){11}{\line(0,1){0.2}}
 \put(1.5,3.5){\line(1,0){2}}
 \put(1.5,4.7){\line(1,0){2}}

 \put(-10,1){blow-up}
 \put(-6,0){\vector(-1,0){4}}
 \put(1.3,5.5){$E_z$}
 \put(0,2.2){$w$}
 \put(3.1,3.6){$\vdots$}

 \multiput(11,0)(0.4,0){25}{\line(1,0){0.2}}
 \put(16,-5){\line(0,1){10}}
 \multiput(12,2.5)(0.4,0){20}{\line(1,0){0.2}}
 \put(17.2,1.5){\line(0,1){2}}
 \put(18.4,1.3){\line(0,1){4}}
 \multiput(17,4)(0.4,0){12}{\line(1,0){0.2}}
 \put(19,3){\line(0,1){2}}
 \put(20.6,3){\line(0,1){2}}
 \put(19.1,4.3){$\cdots$}

 \put(17.8,-1.25){\line(0,1){2.5}}
 \put(19,-1.25){\line(0,1){2.5}}
 \put(6,1){blow-up}
 \put(10,0){\vector(-1,0){4}}
 \put(11.5,3.1){$E_w$}

 \end{picture}
 \end{center}
 \end{minipage}

\end{tabular}
\end{center}
\end{figure}

\ \\

\setlength\unitlength{0.4cm}
\begin{figure}[H]
\begin{center}
\begin{tabular}{c}

 \begin{minipage}{0.6\hsize}
 \begin{center}
\begin{picture}(17,5)

 \put(-5,0){\line(1,0){10}}
 \put(0,-3){\line(0,1){8}}
 \put(-0.15,-3){\line(0,1){8}}
 \put(0.15,-3){\line(0,1){8}}
 \put(-4,2.5){\line(1,0){8}}
 \put(-4,2.35){\line(1,0){8}}
 \put(1.5,-0.3){$\bullet$}
 \put(3,-0.3){$\bullet$}
 \put(1.15,2.2){$\bullet$}
 \put(1.5,4){\line(1,0){4}}
 \put(2.5,1.5){\line(0,1){4}}
 \put(2.35,1.5){\line(0,1){4}}
 \put(2.65,1.5){\line(0,1){4}}
 \put(3,3.7){$\bullet$}
 \put(4.5,3.7){$\bullet$}
 \put(-10,1){triple cover}
 \put(-6,0){\vector(-1,0){4}}
 \put(0,5){$-1$}
 \put(2.5,5){$-1$}

 \put(11,0){\line(1,0){10}}
 \put(16,-3){\line(0,1){8}}
 \put(16.15,-3){\line(0,1){8}}
 \put(11,3){\line(1,0){10}}
 \put(17.5,-0.3){$\bullet$}
 \put(19,-0.3){$\bullet$}
 \put(15.8,1.2){$\bullet$}
 \put(17,2.7){$\bullet$}
 \put(19,2.7){$\bullet$}
 \put(15.8,-0.3){$\circ$}
 \put(15.8,2.7){$\circ$}
 \put(21,0){$A_0$}
 \put(16,5){$A_1$}
 \put(21,3){$A_2$}
 \put(19,5){$F_p$}
 
 \put(6,1){contraction}
 \put(6,0){\vector(1,0){4}}

 \end{picture}
 \end{center}
 \end{minipage}

\end{tabular}
\end{center}
\end{figure}

\ \\

\noindent
Then we can write $F_p=A_0+2A_1+A_3$, $p_{a}(A_{0})=4$,  $g(A_{1})=1$ and $g(A_{2})=r-5$.
This singular fiber satisfies $K_f^2(F_p)=(12-\mu)\chi_f(F_p)$.
Indeed, $\alpha_0(F_p)=1$, $\alpha_1(F_p)=l$, $\alpha_{r/3-1}(F_p)=3$, $\alpha_k(F_p)=0$ for $k\neq 1, r/3-1$, $\varepsilon(F_p)=2$ and $\chi_{\varphi}(F_p)=0$. 
Thus $\chi_f(F_p)=(4r-13)/6$, $e_f(F_p)=4$ and then $e_f(F_p)/\chi_f(F_p)=24/(4r-13)=\mu$.
We can see from the proof of Theorem~\ref{h1upperthm} that any singular fiber $F_p$ satisfying $K_f^2(F_p)=(12-\mu)\chi_f(F_p)$ is obtained in this way.

\smallskip

\noindent
(iii) Assume that $n=2$ and $g\ge 3$.
Consider the situation that $\Gamma_p$ is smooth and $F_p$ is obtained the following sequence of singularity diagrams associated with $\Gamma_p$:

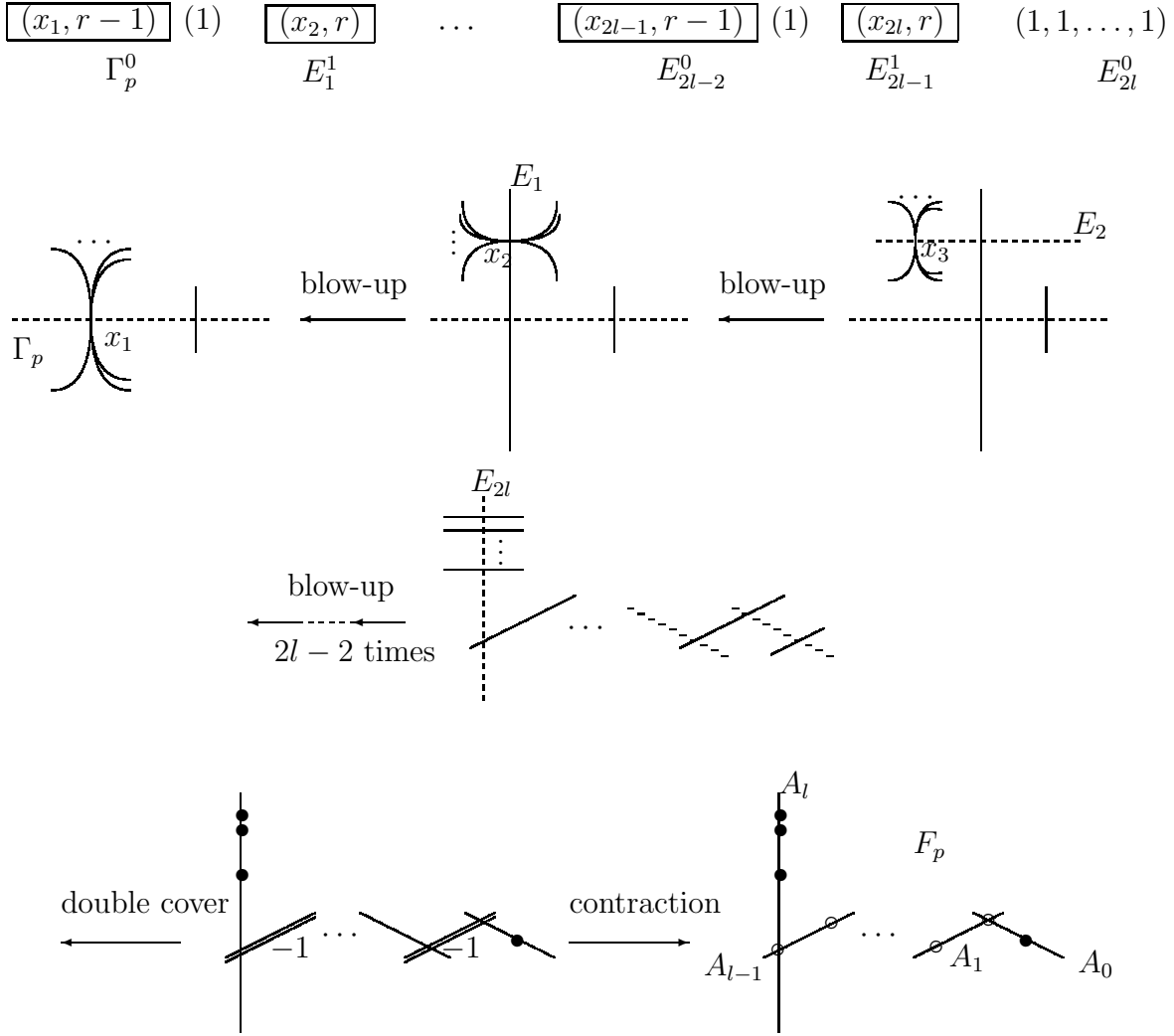
\begin{table}[H]
\begin{center}
\begin{tabular}{c}

 \begin{minipage}{0.2\hsize}
 \begin{tabular}{cc}
   & \\ \cline{1-1}
 \multicolumn{1}{|c|}{$(x_1,r-1)$}& $(1)$ \\ \cline{1-1}
 \multicolumn{2}{c}{\lower1ex\hbox{$\Gamma_p^0$}}
 \end{tabular}
 \end{minipage}

\begin{minipage}{0.12\hsize}
 \begin{tabular}{c}
   \\ \cline{1-1}
 \multicolumn{1}{|c|}{$(x_2,r)$}\\ \cline{1-1}
 \multicolumn{1}{c}{\lower1ex\hbox{$E_1^1$}}
 \end{tabular}
 \end{minipage}

\begin{minipage}{0.1\hsize}
 \begin{tabular}{c}
  $\cdots$
 \end{tabular}
 \end{minipage}

\begin{minipage}{0.22\hsize}
 \begin{tabular}{cc}
   & \\ \cline{1-1}
 \multicolumn{1}{|c|}{$(x_{2l-1},r-1)$}& $(1)$ \\ \cline{1-1}
 \multicolumn{2}{c}{\lower1ex\hbox{$E_{2l-2}^0$}}
 \end{tabular}
 \end{minipage}

\begin{minipage}{0.12\hsize}
 \begin{tabular}{c}
   \\ \cline{1-1}
 \multicolumn{1}{|c|}{$(x_{2l},r)$}\\ \cline{1-1}
 \multicolumn{1}{c}{\lower1ex\hbox{$E_{2l-1}^1$}}
 \end{tabular}
 \end{minipage}

 \begin{minipage}{0.25\hsize}
 \begin{tabular}{ccc}
  & & \\
 $(1,1,\ldots,1)$\\
 \multicolumn{3}{c}{\lower1ex\hbox{$E_{2l}^0$}}
 \end{tabular}
 \end{minipage}

\end{tabular}
\end{center}
\end{table}

\setlength\unitlength{0.35cm}
\begin{figure}[H]
\begin{center}
\begin{tabular}{c}

 \begin{minipage}{0.6\hsize}
 \begin{center}
\begin{picture}(20,5)
 \put(-9.5,-1){$x_1$}
 \multiput(-13,0)(0.4,0){25}{\line(1,0){0.2}}
 \qbezier(-10,0)(-10,2.7)(-8.5,2.7)
 \qbezier(-10,0)(-10,2.7)(-11.5,2.7)
 \qbezier(-10,0)(-10,-2.7)(-8.5,-2.7)
 \qbezier(-10,0)(-10,-2.7)(-11.5,-2.7)
\qbezier(-10,0)(-10,2.3)(-8.5,2.3)
 \qbezier(-10,0)(-10,-2.3)(-8.5,-2.3)
 \put(-6,-1.25){\line(0,1){2.5}}
 
 \put(-10.6,2.7){$\cdots$}
 \put(-13,-1.5){$\Gamma_p$}

\put(2,0){\vector(-1,0){4}}
\put(-2,1){blow-up}

 \multiput(3,0)(0.4,0){25}{\line(1,0){0.2}}
 \put(6,-5){\line(0,1){10}}
 \qbezier(6,3)(7.8,3)(7.8,4.5)
 \qbezier(6,3)(7.8,3)(7.8,1.5)
 \qbezier(6,3)(4.2,3)(4.2,4.5)
 \qbezier(6,3)(4.2,3)(4.2,1.5)
 \qbezier(6,3)(4.1,3)(4.1,4)
 \qbezier(6,3)(7.9,3)(7.9,4)
 \put(10,-1.25){\line(0,1){2.5}}
 \put(5,2.2){$x_2$}
 \put(6,5.1){$E_1$}
 \put(3.7,2.5){$\vdots$}

 \multiput(19,0)(0.4,0){25}{\line(1,0){0.2}}
 \put(24,-5){\line(0,1){10}}
 \multiput(20,3)(0.4,0){20}{\line(1,0){0.2}}
 \put(26.5,-1.25){\line(0,1){2.5}}
 \qbezier(21.5,3)(21.5,4.5)(20.5,4.5)
 \qbezier(21.5,3)(21.5,1.5)(20.5,1.5)
 \qbezier(21.5,3)(21.5,4.5)(22.5,4.5)
 \qbezier(21.5,3)(21.5,1.5)(22.5,1.5)
 \qbezier(21.5,3)(21.5,4.5)(22.5,4.2)
 \qbezier(21.5,3)(21.5,1.5)(22.5,1.8)
 \put(20.8,4.4){$\cdots$}

 \put(14,1){blow-up}
 \put(18,0){\vector(-1,0){4}}
 \put(27.5,3.3){$E_2$}
 \put(21.7,2.4){$x_3$}

 \end{picture}
 \end{center}
 \end{minipage}

\end{tabular}
\end{center}
\end{figure}

\ \\

\setlength\unitlength{0.35cm}
\begin{figure}[H]
\begin{center}
\begin{tabular}{c}

 \begin{minipage}{0.6\hsize}
 \begin{center}
\begin{picture}(20,5)

 \multiput(5,-3)(0,0.4){20}{\line(0,1){0.2}}
 \qbezier(4.5,-1)(4.5,-1)(8.5,1)
 \put(8.2,-0.5){$\cdots$}

 \qbezier(12.5,-1)(12.5,-1)(16.5,1)
 \qbezier(16,-1.25)(16,-1.25)(18,-0.25)
 \multiput(14.5,0.5)(0.4,-0.2){10}{\line(1,0){0.25}}
 \multiput(10.5,0.5)(0.4,-0.2){10}{\line(1,0){0.25}}

 \put(3.5,2){\line(1,0){3}}
 \put(3.5,3.5){\line(1,0){3}}
 \put(3.5,4){\line(1,0){3}}
 \put(5.5,2.2){$\vdots$}
 \put(-2.5,1){blow-up}
 \put(-3,-1.5){$2l-2$ times}
 \put(2,0){\vector(-1,0){2}}
 \multiput(-2.5,0)(0.4,0){10}{\line(1,0){0.2}}
 \put(-2,0){\vector(-1,0){2}}
 \put(4.5,5){$E_{2l}$}

 \end{picture}
 \end{center}
 \end{minipage}

\end{tabular}
\end{center}
\end{figure}

\ \\

\setlength\unitlength{0.4cm}
\begin{figure}[H]
\begin{center}
\begin{tabular}{c}

 \begin{minipage}{0.6\hsize}
 \begin{center}
\begin{picture}(17,5)

 \put(-4,-3){\line(0,1){8}}

 \qbezier(-4.5,-0.5)(-4.5,-0.5)(-1.5,1)
 \qbezier(-4.5,-0.65)(-4.5,-0.65)(-1.5,0.85)
 \put(-3,-0.5){$-1$}
 \put(-1.3,0){$\cdots$}

 \qbezier(1.5,-0.65)(1.5,-0.65)(4.5,0.85)
 \qbezier(1.5,-0.5)(1.5,-0.5)(4.5,1)
 \put(2.7,-0.5){$-1$}
 \qbezier(0,1)(0,1)(3,-0.5)
 \qbezier(3.5,1)(3.5,1)(6.5,-0.5)

 \put(-4.2,2){$\bullet$}
 \put(-4.2,3.5){$\bullet$}
 \put(-4.2,4){$\bullet$}

 \put(5,-0.2){$\bullet$}

 \put(-10,1){double cover}
 \put(-6,0){\vector(-1,0){4}}

\put(18.5,3){$F_p$}

 \put(14,-3){\line(0,1){8}}

 \qbezier(13.5,-0.5)(13.5,-0.5)(16.5,1)
 \qbezier(18.5,-0.5)(18.5,-0.5)(21.5,1)
 \qbezier(20.5,1)(20.5,1)(23.5,-0.5)
 \put(16.7,0){$\cdots$}

 \put(13.8,2){$\bullet$}
 \put(13.8,3.5){$\bullet$}
 \put(13.8,4){$\bullet$}
 \put(13.7,-0.5){$\circ$}
 \put(15.5,0.4){$\circ$}

 \put(19,-0.4){$\circ$}
 \put(20.75,0.5){$\circ$}
 \put(22,-0.2){$\bullet$}

 \put(24,-1){$A_0$}
 \put(19.7,-0.8){$A_1$}
 \put(11.5,-1.1){$A_{l-1}$}
 \put(14,5){$A_{l}$}

 \put(7,1){contraction}
 \put(7,0){\vector(1,0){4}}

 \end{picture}
 \end{center}
 \end{minipage}

\end{tabular}
\end{center}
\end{figure}

\ \\

\noindent
Then we can write $F_p=A_0+A_1+\cdots+A_l$, $p_{a}(A_{0})=2$,  $g(A_{k})=0$ for $k=1,\ldots,l-1$ and $g(A_{l})=r/2-1$.
This singular fiber satisfies $K_f^2(F_p)=(12-\mu)\chi_f(F_p)$.
Indeed, $\alpha_{r/2-1}(F_p)=\alpha_{r/2}(F_p)=l$, $\alpha_k(F_p)=0$ for $k\neq r/2-1,r/2$, $\varepsilon(F_p)=l$ and $\chi_{\varphi}(F_p)=0$. 
Thus $\chi_f(F_p)=(r-2)l/4$, $e_f(F_p)=l$ and then $e_f(F_p)/\chi_f(F_p)=4/(r-2)=\mu$.
We can see from the proof of Theorem~\ref{h1upperthm} that any singular fiber $F_p$ satisfying $K_f^2(F_p)=(12-\mu)\chi_f(F_p)$ is obtained in this way.

\end{exa}
From Theorem~\ref{h1upperthm} and Example~\ref{fiberexa}, we can characterize primitive cyclic covering fibrations of type $(g,1,n)$ whose slope attains the upper bound in Theorem~\ref{h1upperthm}.

\begin{cor}
Let $f\colon S\to B$ be a primitive cyclic covering fibration of type $(g,1,n)$.
Then the slope $\lambda_f$ attains the upper bound in Theorem $\ref{h1upperthm}$
if and only if any singular fiber of $f$ is as in Example $\ref{fiberexa}$.
\end{cor}

\section{Upper bound of the slope: the case of type $(g,0,3)$, $(g,0,2)$}
Let $f\colon S\to B$ be a primitive cyclic covering fibration of type $(g,0,n)$.
We freely use notations in the previous section.
Since $\widetilde{\varphi}\colon \widetilde{W}\to B$ is a ruled surface, a relatively minimal model of it is not unique. By performing elementary transformations, we can choose a standard one:

\begin{lem}[cf.\ \cite{enoki}, \cite{pi1}] \label{eltrlem}
There exists a relatively minimal model $\varphi\colon W\rightarrow B$ of $\widetilde{\varphi}$ such that
if $n=2$ and $g$ is even, then
$$
{\rm mult}_x(R)\le \frac{r}{2}=g+1
$$
for all $x\in R$, and otherwise,
$$
{\rm mult}_x(R_h)\le \frac{r}{2}=\frac{g}{n-1}+1
$$
for all $x\in R_h$, where $R_h$ denotes the horizontal part of $R$, that is, the sum of all $\varphi$-horizontal components of $R$.
\end{lem}
We take a relatively minimal model $\varphi\colon W\rightarrow B$ satisfying the inequality in Lemma~\ref{eltrlem}.
It is easily seen that the value of $\alpha_k(F_p)$ does not depend on the choice of the relatively minimal model of $\widetilde{\varphi}$ satisfying Lemma~\ref{eltrlem}.
Since $\varphi\colon W\to B$ is a relatively minimal ruled surface, we have 
$K_{\varphi}^2=0$ and $R^2=-rK_{\varphi}R$.
Combining these equalities with \eqref{KReq}, \eqref{kfeq}, \eqref{chifeq} and \eqref{efeq},
we get the following lemma:

\begin{lem}[cf.\ \cite{enoki}]
The following equalities hold.

\begin{align*}
K_f^2=&\;\frac{n-1}{r-1}\left(\frac{(n-1)r-2n}{n}(\alpha_0-2\varepsilon)
+(n+1)\sum_{k\ge 1}k(-nk+r)\alpha_k\right)
-n\sum_{k\ge 1}\alpha_k+\varepsilon.\\
\chi_f=&\;\frac{n-1}{12(r-1)}\left(\frac{(2n-1)r-3n}{n}(\alpha_0-2\varepsilon)+
(n+1)\sum_{k\ge 1}k(-nk+r)\alpha_k\right).\\
e_f=&\;(n-1)\alpha_0+n\sum_{k\ge 1}\alpha_k-(2n-1)\varepsilon. \\
\end{align*}
\end{lem}

For $p\in B$, we put
\begin{align*}
K_f^2(F_p)=&\;\frac{n-1}{r-1}\left(\frac{(n-1)r-2n}{n}(\alpha_0(F_p)-2\varepsilon(F_p))
+(n+1)\sum_{k\ge 1}k(-nk+r)\alpha_k(F_p)\right) \\
&-n\sum_{k\ge 1}\alpha_k(F_p)+\varepsilon(F_p),\\
\chi_f(F_p)=&\;\frac{n-1}{12(r-1)}\left(\frac{(2n-1)r-3n}{n}(\alpha_0(F_p)-2\varepsilon(F_p))+
(n+1)\sum_{k\ge 1}k(-nk+r)\alpha_k(F_p)\right),\\
e_f(F_p)=&\;(n-1)\alpha_0(F_p)+n\sum_{k\ge 1}\alpha_k(F_p)-(2n-1)\varepsilon(F_p). 
\end{align*}

In \cite{enoki}, we have established the slope equality for a primitive cyclic covering fibration of type $(g,0,n)$ by using the above representations of $K_f^2$, $\chi_f$ and $e_f$,
 and an upper bound of it's slope when $n\ge 4$.
In this section, we show that a more careful study as in the previous section gives us an upper bound of the slope when $n=2,3$.
Namely, we show the following theorem:
\begin{thm} \label{h0upperthm}
$(1)$ Let $f\colon S\to B$ be a primitive cyclic covering fibration of type $(4,0,3)$. 
Then we have
$$
K_f^2\le \frac{129}{17}\chi_f
$$
with the equality holding if and only if any singular fiber of $f$ is a triple fiber.

\smallskip

\noindent
$(2)$ Let $f\colon S\to B$ be a primitive cyclic covering fibration of type $(g,0,3)$. 
Assume that $g>4$, or $g=4$ and $f$ has no triple fibers.
Then we have
$$
K_f^2\le \left(12-\frac{72(r-1)}{4r^2-15r+27-36\delta}\right)\chi_f,
$$
where $\delta=\left\{\begin{array}{l}
0, \ \text{if}\ r\in 6\mathbb{Z}, \\
1, \ \text{if}\ r\not\in 6\mathbb{Z}. \\
\end{array}
\right.$

\smallskip

\noindent
$(3)$ Let $f\colon S\to B$ be a primitive cyclic covering fibration of type $(g,0,2)$,
 i.e., a relatively minimal hyperelliptic fibration of genus $g$.
Then we have
$$
K_f^2\le \left(12-\frac{4(2g+1)}{g^2-1+\delta}\right)\chi_f,
$$
where $\delta=\left\{\begin{array}{l}
0, \ \text{if $g$ is odd}, \\
1, \ \text{if $g$ is even.} \\
\end{array}
\right.$
\end{thm}

\begin{rem}
Theorem~\ref{h0upperthm}~(3) was shown by Xiao in \cite{xiao_book}
and the upper bound is known to be sharp \cite{liutan}.
\end{rem}

Recall the following lemma from \cite{enoki} (cf.\ Lemmas~\ref{alpha0lem}, \ref{alphaklem} and \ref{gammaiotalem}):

\begin{lem}\label{indexlem}
The following hold:

\smallskip

\noindent
$(1)$ $\iota(F_p)=j(F_p)-\eta_p$.

\smallskip

\noindent
$(2)$ $\alpha^{+}_0(F_p)\ge (n-2)(j(F_p)-\eta_p+2\kappa(F_p))+\delta_{n=2}\sum_{k\ge 1}2k\alpha^{\mathrm{co}}_{(2k+1\to 2k+1)}(F_p)$.

\smallskip

\noindent
$(3)$  $\sum_{k\ge 1}\alpha''_k(F_p)=\sum_{a\ge 1}(an-2)j_{0,a}(F_p)+2\eta_p-\kappa(F_p)$.


\end{lem}

\begin{lem} \label{h0etalem}
The following hold.

\smallskip

\noindent
$(1)$ If $n=3$, then we have $j_{0,1}(F_p)\le 2\eta_p+1+\sum_{a\ge 2}(2a-1)j_{0,a}(F_p)$,
with the equality holding only if $F_p$ is a triple fiber, $r\in 9\mathbb{Z}+6$, $j(F_p)=j_{0,1}(F_p)$, $\kappa(F_p)=1$ and
$$
\alpha^{+}_0(F_p)\ge \frac{5(r-6)}{9}+j(F_p)-\eta_p+2\kappa(F_p).
$$

\smallskip

\noindent
$(2)$ If $n=2$, then 
$$
\sum_{k\ge 1}\alpha^{\mathrm{co},1}_{(2k+1\to 2k+1)}(F_p)\le
\widehat{\eta}_p+2\eta'_p+\sum_{a\ge 3}(a-2)j_{0,a}(F_p)-j'_{0,1}(F_p).
$$

\end{lem}

\begin{proof}
Suppose that $n=3$. By the same argument as in the proof of Lemma~\ref{etalem}, we have
$$
j''_{0,1}(F_p)\le \overline{\eta}_p+\eta''_p+\sum_{a\ge 1}2aj'_{0,a}(F_p)+\sum_{a\ge 2}(2a-1)j''_{0,a}(F_p).
$$
Hence, if $\Gamma_p\not\subset R$, we have
$$
j_{0,1}(F_p)\le 2\eta_p+\sum_{a\ge 2}(2a-1)j_{0,a}(F_p).
$$
If $\Gamma_p\subset R$ and $\Gamma_p$ does not contribute to $j'_{0,1}(F_p)$, we have
$$
j_{0,1}(F_p)\le 2\eta_p-1+\sum_{a\ge 2}(2a-1)j_{0,a}(F_p),
$$
since $j''_{0,1}(F_p)=j_{0,1}(F_p)$, $\sum_{a\ge 1}2aj'_{0,a}(F_p)+\sum_{a\ge 2}(2a-1)j''_{0,a}(F_p)=\sum_{a\ge 2}(2a-1)j_{0,a}(F_p)+1$ and $\overline{\eta}_p\le \eta''_p\le \eta_p-1$.
If $\Gamma_p\subset R$ and $\Gamma_p$ contributes to $j'_{0,1}(F_p)$, we have
$$
j_{0,1}(F_p)\le 2\eta_p+1+\sum_{a\ge 2}(2a-1)j_{0,a}(F_p),
$$
since $j''_{0,1}(F_p)=j_{0,1}(F_p)-1$, $\sum_{a\ge 1}2aj'_{0,a}(F_p)+\sum_{a\ge 2}(2a-1)j''_{0,a}(F_p)=\sum_{a\ge 2}(2a-1)j_{0,a}(F_p)+2$ and $\overline{\eta}_p\le \eta''_p\le \eta_p-1$.
If $j_{0,1}(F_p)= 2\eta_p+1+\sum_{a\ge 2}(2a-1)j_{0,a}(F_p)$, then $\overline{\eta}_p=\eta_p-1$ and the sequence of singularity diagrams associated with $D^{t}(p)$ containing the proper transform of the fiber $\Gamma_p=C^{t,1}$ is the following:

\begin{table}[H]
\begin{center}
\begin{tabular}{c}

 \begin{minipage}{0.2\hsize}
 \begin{tabular}{c}
   \\ \cline{1-1}
 \multicolumn{1}{|c|}{$n_1$} \\ \cline{1-1}
 \multicolumn{1}{|c|}{$m_2$} \\ \cline{1-1}
 \multicolumn{1}{|c|}{$m_1$} \\ \cline{1-1}
 \multicolumn{1}{c}{\lower1ex\hbox{$\mathcal{D}^{t,1}$}}
 \end{tabular}
 \end{minipage}

\begin{minipage}{0.2\hsize}
 \begin{tabular}{c}
   \\ \cline{1-1}
 \multicolumn{1}{|c|}{$n_2$}\\ \cline{1-1}
 \multicolumn{1}{|c|}{$m_2$}\\ \cline{1-1}
 \multicolumn{1}{c}{\lower1ex\hbox{$\mathcal{D}^{t,2}$}}
 \end{tabular}
 \end{minipage}

\begin{minipage}{0.2\hsize}
 \begin{tabular}{cc}
   & \\ \cline{1-2}
 \multicolumn{1}{|c|}{$n_1$} & \multicolumn{1}{|c|}{$n_2$}\\ \cline{1-2}
 \multicolumn{2}{c}{\lower1ex\hbox{$\mathcal{D}^{t,3}$}}
 \end{tabular}
 \end{minipage}

\end{tabular}
\end{center}
\end{table}

\noindent
where $n_i\in 3\mathbb{Z}$, $m_i\in 3\mathbb{Z}+1$ and $C^{t,k}$ is the exceptional curve obtained by the blow-up at the singularity with multiplicity $m_{k-1}$ for $k=2,3$.
From Lemma~\ref{tclem}, we have $r+5=m_1+m_2+n_1$, $m_1+3=m_2+n_2$ and $m_2+2=n_1+n_2$.
In particular, we get $r=3(n_1+n_2-4) \in 9\mathbb{Z}+6$.
Since $\overline{\eta}_p=\eta''_p$, the sequence of singularity diagrams associated with another $D^{t}(p)$ is the following from Example~\ref{diagexa}~(2):

\begin{table}[H]
\begin{center}
\begin{tabular}{c}

\begin{minipage}{0.2\hsize}
 \begin{tabular}{c}
   \\ \cline{1-1}
 \multicolumn{1}{|c|}{$n_1$}\\ \cline{1-1}
 \multicolumn{1}{|c|}{$m_1$}\\ \cline{1-1}
 \multicolumn{1}{c}{\lower1ex\hbox{$\mathcal{D}^{t,1}$}}
 \end{tabular}
 \end{minipage}

\begin{minipage}{0.1\hsize}
 \begin{tabular}{c}
   \\ \cline{1-1}
 \multicolumn{1}{|c|}{$n_2$}\\ \cline{1-1}
 \multicolumn{1}{|c|}{$n_1$}\\ \cline{1-1}
 \multicolumn{1}{c}{\lower1ex\hbox{$\mathcal{D}^{t,2}$}}
 \end{tabular}
 \end{minipage}

\begin{minipage}{0.08\hsize}
 \begin{tabular}{c}
  or
 \end{tabular}
 \end{minipage}

\begin{minipage}{0.2\hsize}
 \begin{tabular}{cc}
   & \\ \cline{1-2}
 \multicolumn{1}{|c|}{$n_1$} & \multicolumn{1}{|c|}{$n_2$}\\ \cline{1-2}
 \multicolumn{2}{c}{\lower1ex\hbox{$\mathcal{D}^{t,2}$}}
 \end{tabular}
 \end{minipage}

\end{tabular}
\end{center}
\end{table}

\noindent
where $n_i\in 3\mathbb{Z}$, $m_i\in 3\mathbb{Z}+1$
and $C^{t,2}$ is the exceptional curve obtained by the blow-up at the singularity with multiplicity $m_{1}$.
From Lemma~\ref{tclem}, we have $m+3=m_1+n_1$ and $m_1+2=n_1+n_2$, where $m$ is the multiplicity of $R$ at the point to which $C^{t,1}$ is contracted.
Clearly, we have $j(F_p)=j_{0,1}(F_p)=2\eta_p+1$, $\kappa(F_p)=1$. 
Moreover, we see that $F_p$ is a triple fiber since the multiplicity of the fiber over $p$ along the exceptional curve obtained by the blow-up at a singularity of type $3\mathbb{Z}$ is $3$-multiple.
Hence, we get $\alpha^{+}_{0}(F_p)\ge 2r/3$.
To prove (1), it is sufficient to show that $2r/3-\left(j(F_p)-\eta_p+2\kappa(F_p)\right)\ge 5(r-6)/9$, i.e., $\eta_p\le (r+3)/9$.
To prove this, we will show that for any $n_1\in 3\mathbb{Z}_{>0}$, the number of singularities of type $3\mathbb{Z}+1$ over a singularity with multiplicity $n_1$ is less than or equal to $2n_1/3-2$ by induction on $n_1/3$.
If $n_1=3$, the claim is obvious. We assume that $n_1>3$ and the claim holds true until $n_1-3$.
Let $x$ be a singularity of $R$ with multiplicity $n_1$. 
Let $x_1,\dots,x_l$ be all singularities of type $3\mathbb{Z}+1$ over $x$ which is not over any singularity of type $3\mathbb{Z}+1$ over $x$. Let $m_i$ denotes the multiplicity of $R$ at $x_i$ for $i=1\dots,l$.
Then, the sequence of singularity diagrams associated with $D^{t}(p)$ obtained by the blow-up at $x_i$ is the following:

\begin{table}[H]
\begin{center}
\begin{tabular}{c}

\begin{minipage}{0.2\hsize}
 \begin{tabular}{c}
   \\ \cline{1-1}
 \multicolumn{1}{|c|}{$n_{i,1}$}\\ \cline{1-1}
 \multicolumn{1}{|c|}{$m_{i,1}$}\\ \cline{1-1}
 \multicolumn{1}{c}{\lower1ex\hbox{$\mathcal{D}^{t,1}$}}
 \end{tabular}
 \end{minipage}

\begin{minipage}{0.1\hsize}
 \begin{tabular}{c}
   \\ \cline{1-1}
 \multicolumn{1}{|c|}{$n_{i,2}$}\\ \cline{1-1}
 \multicolumn{1}{|c|}{$n_{i,1}$}\\ \cline{1-1}
 \multicolumn{1}{c}{\lower1ex\hbox{$\mathcal{D}^{t,2}$}}
 \end{tabular}
 \end{minipage}

\begin{minipage}{0.08\hsize}
 \begin{tabular}{c}
  or
 \end{tabular}
 \end{minipage}

\begin{minipage}{0.2\hsize}
 \begin{tabular}{cc}
   & \\ \cline{1-2}
 \multicolumn{1}{|c|}{$n_{i,1}$} & \multicolumn{1}{|c|}{$n_{i,2}$}\\ \cline{1-2}
 \multicolumn{2}{c}{\lower1ex\hbox{$\mathcal{D}^{t,2}$}}
 \end{tabular}
 \end{minipage}

\end{tabular}
\end{center}
\end{table}

\noindent
where $n_{i,j}\in 3\mathbb{Z}$, $m_{i,j}\in 3\mathbb{Z}+1$ satisfy that $m_i+3=m_{i,1}+n_{i,1}$, $m_{i,1}+2=n_{i,1}+n_{i,2}$
and $C_{t,1}$, $C^{t,2}$ respectively denote the exceptional curves obtained by the blow-up at the singularity $x_i$, the singularity with multiplicity $m_{i,1}$.
By the assumption of induction, the number of singularities of type $3\mathbb{Z}+1$ over $x_i$ is less than or equal to $2(n_{i,1}+n_{i,2})/3-3$. 
Hence the number of singularities of type $3\mathbb{Z}+1$ over $x$ is less than or equal to $\sum_{i=1}^{l}2(n_{i,1}+n_{i,2})/3-2l$. 
If $l=1$, then we have $n_1\ge m_1+2=2n_{1,1}+n_{1,2}-3$. Thus, we get

\begin{align*}
\frac{2}{3}n_1-2&\ge \frac{2}{3}(2n_{1,1}+n_{1,2})-4 \\
                     &\ge \frac{2}{3}(n_{1,1}+n_{1,2})-2.
\end{align*}
If $l=2$, then we have $n_1\ge m_1+m_2+1=2n_{1,1}+n_{1,2}+2n_{2,1}+n_{2,2}-9$.
Thus, we get

\begin{align*}
\frac{2}{3}n_1-2&\ge \frac{2}{3}(2n_{1,1}+n_{1,2}+2n_{2,1}+n_{2,2})-8 \\
                     &\ge \frac{2}{3}(n_{1,1}+n_{1,2}+n_{2,1}+n_{2,2})-4.
\end{align*}
If $l\ge 3$, then we have $n_1\ge \sum_{i=1}^{l}m_i=\sum_{i=1}^{l}(2n_{i,1}+n_{i,2})-5l$.
Thus, we get

\begin{align*}
\frac{2}{3}n_1-2&\ge \frac{2}{3}\sum_{i=1}^{l}(2n_{i,1}+n_{i,2})-\frac{10}{3}l-2 \\
                     &\ge \frac{2}{3}\sum_{i=1}^{l}(n_{i,1}+n_{i,2})-2l
\end{align*}
with equality holds only if $l=3$, $n_1=m_{1}+m_{2}+m_{3}$ and $n_{i,1}=3$ for any $i$.
Hence, the number of singularities of type $3\mathbb{Z}+1$ over a singularity with multiplicity $n_1$ is less than or equal to $2n_1/3-2$.
Now, $r=3(n_1+n_2-4)$ and the number of singularities of type $3\mathbb{Z}+1$ over $p$ is less than or equal to $2(n_1+n_2)/3-2$. Namely, we have
$$
j(F_p)-1=j''_{0,1}(F_p)\le \frac{2}{9}r+\frac{2}{3}.
$$
Hence, we have $\eta_p=(j(F_p)-1)/2=(r+3)/9$, which is the desired inequality.

Suppose that $n=2$. By the same argument as in the proof of Lemma~\ref{etalem}, we have
$$
\sum_{k\ge 1}\alpha^{\mathrm{co},1}_{(2k+1\to 2k+1)}(F_p)\le \widehat{\eta}_p+\sum_{a\ge 1}aj'_{0,a}(F_p)+\sum_{a\ge 3}(a-2)j''_{0,a}(F_p).
$$
Since $\sum_{a\ge 1}aj'_{0,a}(F_p)+\sum_{a\ge 3}(a-2)j''_{0,a}(F_p)=2\eta'_p+\sum_{a\ge 3}(a-2)j_{0,a}(F_p)-j'_{0,1}(F_p)$, the claim (2) follows.
\end{proof}

\begin{lem}[cf.\ Lemma~\ref{n2h1kappalem}] \label{n2kappalem}
If $n=2$, then we have
$$
\kappa(F_p)\le \frac{2}{3}\sum_{a\ge 2}(a-1)j_{0,a}(F_p).
$$
\end{lem}

\begin{proof}
It is sufficient to show that
$$
\kappa^{t}(F_p)\le \frac{2}{3}\sum_{a\ge 2}(a-1)j^{t}_{0,a}(F_p)
$$
for any $t$.
If $\kappa^{t}(F_p)=0$, then it is clear.
Thus, we may assume $\kappa^{t}(F_p)>0$.
Then clearly we have $j^{t}(F_p)\ge \kappa^{t}(F_p)+2$.
Since any blow-up at a $(t,2)$-vertical type singularity contributes $-2$ to the number $\sum_{k\ge 1}(L^{t,k})^2=-\sum_{a\ge 1}2aj^{t}_{0,a}(F_p)$
and $\Gamma_p$ contains no $2$-vertical type singularity, we get
$$
-\sum_{a\ge 1}2aj^{t}_{0,a}(F_p)\le -j^{t}(F_p)+(-2)\left(\iota^{t}(F_p)+\kappa^{t}(F_p)\right)=-\left(3j^{t}(F_p)+2\kappa^{t}(F_p)-2\right).
$$
Combining this equality with $j^{t}(F_p)\ge \kappa^{t}(F_p)+2$, we have
$$
\sum_{a\ge 1}2aj^{t}_{0,a}(F_p)\ge 2j^{t}(F_p)+3\kappa^{t}(F_p).
$$
Hence we get $\kappa^{t}(F_p)\le \frac{2}{3}\sum_{a\ge 2}(a-1)j^{t}_{0,a}(F_p)$.
\end{proof}

Now, we are ready to prove Theorem~\ref{h0upperthm}.

\begin{prfofthmh0}
Let $f\colon S\to B$ be a primitive cyclic covering fibration of type $(g,0,n)$.
For any $p\ge B$ and $\mu\ge 0$, we have from Lemma~\ref{indexlem}~(3) that
\begin{align}
&e_f(F_p)-\mu\chi_f(F_p) \nonumber \\
=&\;(n-1)\alpha_0(F_p)+n\sum_{k\ge 1}\alpha_k(F_p)-(2n-1)
 \varepsilon(F_p) \nonumber \\
&-\mu'\left(\frac{r(2n-1)-3n}{n}\alpha_0(F_p)+(n+1)\sum_{k\ge 1}(-nk^2+rk)\alpha_k(F_p)
-\frac{2(r(2n-1)-3n)}{n}\varepsilon(F_p)\right) \nonumber \\
=&\;A_n\alpha_0(F_p) 
+\sum_{k\ge 1}\left(Q(k)+B_n\right)\alpha_k(F_p)
-\left(2A_n+1\right)\varepsilon(F_p) \nonumber \\
=&\;A_n\alpha^{+}_0(F_p) 
+\sum_{k\ge 1}\left(Q(k)+B_n\right)\alpha'_k(F_p)+\sum_{k\ge 1}\left(Q(k)+B_n\right)\alpha''_k(F_p)-2A_nj(F_p)-j_{0,1}(F_p) \nonumber \\
=&\;A_n\alpha^{+}_0(F_p) 
+\sum_{k\ge 1}\left(Q(k)+B_n\right)\alpha'_k(F_p)+\sum_{k\ge 1}Q(k)\alpha''_k(F_p) \nonumber \\
&+\sum_{a\ge 1}\left((an-2)B_n-2A_n-\delta_{a=1}\right)j_{0,a}(F_p)+2B_n\eta_p-B_n\kappa(F_p), \nonumber
\end{align}
where 
$$
\mu'=\frac{n-1}{12(r-1)}\mu,\quad A_n=n-1-\frac{r(2n-1)-3n}{n}\mu',\quad B_n=n-\frac{(n+1)(r^2-\delta n^2)}{4n}\mu'
$$
and
$$
Q(k)=\mu'\left(n(n+1)\left(k-\frac{r}{2n}\right)^2-\frac{n(n+1)\delta}{4}\right)\ge 0.
$$

Assume that $A_n\ge 0$ and $B_n\ge 0$. From Lemma~\ref{indexlem}~(2), $e_f(F_p)-\mu\chi_f(F_p)$ is greater than or equal to
\begin{align}
&\left((n-2)\left(j(F_p)-\eta_p+2\kappa(F_p)\right)+\delta_{n=2}\sum_{k\ge 1}2k\alpha^{\mathrm{co}}_{(2k+1\to 2k+1)}(F_p)\right)A_n+\sum_{k\ge 1}Q(k)\alpha''_k(F_p) \nonumber \\
&+\sum_{a\ge 1}\left((an-2)B_n-2A_n-\delta_{a=1}\right)j_{0,a}(F_p)+2B_n\eta_p-B_n\kappa(F_p)  \nonumber \\
=&\;\delta_{n=2}\sum_{k\ge 1}2k\alpha^{\mathrm{co}}_{(2k+1\to 2k+1)}(F_p)A_n+\sum_{a\ge 1}\left((n-4)A_n+(an-2)B_n-\delta_{a=1}\right)j_{0,a}(F_p) \nonumber \\
&+\sum_{k\ge 1}Q(k)\alpha''_k(F_p)+(2B_n-(n-2)A_n)\eta_p+(2(n-2)A_n-B_n)\kappa(F_p) \label{h0steq}. 
\end{align}

\smallskip

\noindent
(i) We first assume that $n=3$. 

\smallskip

\noindent
(i,1) Assume that $j_{0,1}(F_p)\le 2\eta_p+\sum_{a\ge 2}(2a-1)j_{0,a}(F_p)$.
We put 
$$
\mu=\frac{72(r-1)}{4r^2-15r+27-36\delta}.
$$
Then, we have $A_3>0$ and $B_3>0$.
Since the coefficient $-A_3+2B_3$ of $\eta_p$ in \eqref{h0steq} is also positive,
\eqref{h0steq} is greater than or equal to
\begin{align}
\sum_{a\ge 2}\left(\left(a-\frac{3}{2}\right)A_3+(a-1)B_3\right)j_{0,a}(F_p)+\left(-\frac{3}{2}A_3+2B_3-1\right)j_{0,1}(F_p)+(2A_3-B_3)\kappa(F_p). \label{h0i1eq}
\end{align}
By the definition of $\mu$, we have 
$-\displaystyle{\frac{3}{2}}A_3+2B_3-1=0$ and $2A_3-B_3>0$.
Thus \eqref{h0i1eq} is non-negative.

\smallskip

\noindent
(i,2) Assume that $j_{0,1}(F_p)=2\eta_p+1+\sum_{a\ge 2}(2a-1)j_{0,a}(F_p)$ and $r>6$.
We also put
$$
\mu=\frac{72(r-1)}{4r^2-15r+27-36\delta}.
$$
From Lemma~\ref{h0etalem}~(1), we have
\begin{align*}
&e_f(F_p)-\mu\chi_f(F_p) \\
\ge &\frac{5(r-6)}{9}A_3+(-A_3+B_3-1)j_{0,1}(F_p)
+(-A_3+2B_3)\eta_p+(2A_3-B_3) \\
=&\;\frac{5(r-6)}{9}A_3+(-\frac{3}{2}A_3+2B_3-1)j_{0,1}(F_p)
+\frac{5}{2}A_3-2B_3 \\
\ge &\;\frac{15}{2}A_3-2B_3 \\
>& 0.
\end{align*}

\smallskip

\noindent
(i,3) Assume that $j_{0,1}(F_p)=2\eta_p+1+\sum_{a\ge 2}(2a-1)j_{0,a}(F_p)$ and $r=6$.
We put $\mu=\displaystyle{\frac{129}{17}}$.
From Lemma~\ref{h0etalem}~(1) and $j_{0,1}(F_p)=3$, we have
$$
e_f(F_p)-\mu\chi_f(F_p)\ge -2A_3+4B_3-3=0,
$$
where the last equality follows by the definition of $\mu$.

\medskip

\noindent
(ii) We assume that $n=2$. 
Put
$$
\mu=\frac{4(2g+1)}{g^2-1+\delta}.
$$
Then we have $2B_2-2A_2-1-6(1-\delta)\mu'=0$, $A_2>0$ and $B_2\ge 1$.
From \eqref{n2j''01eq} and \eqref{n2etaeq}, \eqref{h0steq} is equal to
\begin{align*}
&\sum_{k\ge 1}Q(k)\alpha''_k(F_p)+\sum_{a\ge 2}\left(-2A_2+(2a-2)B_2\right)j_{0,a}(F_p)-B_2\kappa(F_p) \\
&+(2B_2-2A_2-1)\sum_{k\ge 1}\alpha^{\mathrm{tr}}_{(2k+1\to 2k+1)}(F_p)+\sum_{k\ge 1}\left(2(k-1)A_2+2B_2-1\right)\alpha^{\mathrm{co},0}_{(2k+1\to 2k+1)}(F_p) \\
&+\sum_{k\ge 1}\left(2(k-1)A_2-1\right)\alpha^{\mathrm{co},1}_{(2k+1\to 2k+1)}(F_p)-(2A_2+1)j'_{0,1}(F_p)+2B_2\left(\eta'_p+\widehat{\eta}_p\right).
\end{align*}
Using Lemma~\ref{h0etalem}~(2) for $-\sum_{k\ge 1}\alpha^{\mathrm{co},1}_{(2k+1\to 2k+1)}(F_p)$, \eqref{h0steq} is greater than or equal to
\begin{align}
&\sum_{k\ge 1}Q(k)\alpha''_k(F_p)+\sum_{a\ge 2}\left(-2A_2+(2a-2)B_2-(a-2)\right)j_{0,a}(F_p)-B_2\kappa(F_p) \nonumber \\
&+(2B_2-2A_2-1)\sum_{k\ge 1}\alpha^{\mathrm{tr}}_{(2k+1\to 2k+1)}(F_p)+\sum_{k\ge 1}\left(2(k-1)A_2+2B_2-1\right)\alpha^{\mathrm{co},0}_{(2k+1\to 2k+1)}(F_p) \nonumber \\
&+\sum_{k\ge 1}2(k-1)A_2\alpha^{\mathrm{co},1}_{(2k+1\to 2k+1)}(F_p)-2A_2j'_{0,1}(F_p)+(2B_2-2)\eta'_p+(2B_2-1)\widehat{\eta}_p. \label{n2h0steq}
\end{align}
By the definition of $\mu$, we see that
\begin{equation} \label{n2trineq}
\sum_{k\ge 1}Q(k)\alpha''_k(F_p)+(2B_2-2A_2-1)\sum_{k\ge 1}\alpha^{\mathrm{tr}}_{(2k+1\to 2k+1)}(F_p)\ge 0
\end{equation}
with the equality holding if and only if $\alpha''_{(g-1)/2}(F_p)=\alpha^{\mathrm{tr}}_{(g\to g)}(F_p)$ and $\alpha''_k(F_p)=0$ for $k\neq (g-1)/2$ when $g$ is odd, 
or $\alpha''_{g/2}(F_p)=\alpha^{\mathrm{tr}}_{(g+1\to g+1)}(F_p)$ and $\alpha''_k(F_p)=0$ for $k\neq g/2$ when $g$ is even.
If $j'_{0,1}(F_p)=1$, then $g$ is odd and $\Gamma_p\subset R$ contains a singularity of type $(g+2\to g+2)$ since $\Gamma_p$ is blown up just twice and Lemma~\ref{eltrlem}.
Thus, $\Gamma_p$ contributes to $\alpha^{\mathrm{co},1}_{(g+2\to g+2)}(F_p)$ and then we have
\begin{equation} \label{n2j'01ineq}
\sum_{k\ge 1}2(k-1)A_2\alpha^{\mathrm{co},1}_{(2k+1\to 2k+1)}(F_p)-2A_2j'_{0,1}(F_p)\ge 0.
\end{equation}
From Lemma~\ref{n2kappalem} and
$$
-2A_2+\frac{4}{3}(a-1)B_2-(a-2)>0,
$$
we have
\begin{equation} \label{n2kappaineq}
\sum_{a\ge 2}\left(-2A_2+(2a-2)B_2-(a-2)\right)j_{0,a}(F_p)-B_2\kappa(F_p)\ge 0.
\end{equation}
From \eqref{n2trineq}, \eqref{n2j'01ineq}, \eqref{n2kappaineq} and $B_2\ge 1$, \eqref{n2h0steq} is non-negative.
\qed
\end{prfofthmh0}

\end{document}